\documentclass[10pt]{amsart}
\usepackage{ucs} 
\usepackage[T1]{fontenc}
\usepackage{lmodern,amsfonts,amsmath,amstext,amsbsy,amssymb,
  amsopn,amsthm,upref,eucal,bm}

\usepackage{soul}
\usepackage{hyperref}
\RequirePackage[dvipsnames]{xcolor} 
\definecolor{halfgray}{gray}{0.55} 
\definecolor{webgreen}{rgb}{0,0.5,0}
\definecolor{webbrown}{rgb}{.6,0,0} 
\hypersetup{%
  colorlinks=true, linktocpage=true, pdfstartpage=3,
  pdfstartview=FitV,%
  breaklinks=true, pdfpagemode=UseNone, pageanchor=true, 
  pdfpagemode=UseOutlines,%
  plainpages=false, bookmarksnumbered, bookmarksopen=true,
  bookmarksopenlevel=1,%
  hypertexnames=true,
  pdfhighlight=/O,
  urlcolor=RoyalBlue, linkcolor=webbrown,
  citecolor=webgreen, 
  pdftitle={On effective equidistribution for higher step nilflows},%
  pdfauthor={Livio Flaminio and Giovanni Forni},%
  pdfsubject={2000 Mathematical Subject Classification: Primary:
    37-XX, 37C15, 37C40},%
  pdfkeywords={Cohomological Equations, Nilflows},%
  pdfcreator={pdfLaTeX},%
  pdfproducer={LaTeX with hyperref}%
}
\newcommand{\red}{\textcolor{red}}

\newenvironment{modifenv}{\textcolor{blue}\bgroup}{\egroup}

\newtheorem{theorem}{Theorem}[section]
\newtheorem{lemma}[theorem]{Lemma}
\newtheorem{corollary}[theorem]{Corollary}
\newtheorem{proposition}[theorem]{Proposition}

\newtheoremstyle{style2}%
{7pt}{9pt}{\normalfont}%
{0pt}{\bf}{.\enspace}{0pt}%
{}
\theoremstyle{style2}
\newtheorem{definition}[theorem]{Definition}
\newtheorem{notation}[theorem]{Notation}

\newtheoremstyle{style3}%
{7pt}{9pt}{\normalfont}%
{0pt}{\it}{.\enspace}{0pt}%
{}
\theoremstyle{style3}

\newtheorem{remark}[theorem]{Remark}

\newtheoremstyle{style4}%
{7pt}{9pt}{\itshape}%
{0pt}{\it}{.\enspace}{0pt}%
{}
\theoremstyle{style4}
\newtheorem*{sublemma}{Sub-lemma}

\newcommand{\field}[1]{\mathbb{#1}}
\newcommand{\R}{\field{R}}
\newcommand{\N}{\field{N}}
\newcommand{\Z}{\field{Z}}
\newcommand{\Q}{\field{Q}}
\newcommand{\T}{\field{T}}
\newcommand{\Cal}{\mathcal}

\newcommand{\bs}{\mathbf s}

\newcommand{\ad}{\hbox{\rm ad}}
\newcommand{\Ad}{\hbox{\rm Ad}}

\newcommand{\<}{{\langle}}

\newcommand{\pref}[1]{(\ref{#1})}

\newcommand{\lstep}{{k}}
\newcommand{\adim}{{a}}
\newcommand{\return}{{N}}

\renewcommand{\>}{{\rangle}}
\renewcommand{\|}{\,\Vert\,}

\def\Ind{\operatorname{Ind}}

\newcommand{\torus}{\mathbb T^\adim}

\newcommand{\fil}{\operatorname{\mathfrak f \mathfrak i \mathfrak
  l}}%
\newcommand{\filk}{\fil_\lstep}
\newcommand{\Fil}{\operatorname{Fil}}%
\newcommand{\Filk}{\Fil_\lstep}
\newcommand{\pr}{\operatorname{pr}}

\newcommand{\meas}{\mathcal L}

\renewcommand{\include}[1]{\input #1} \def\AP{\operatorname{AP}}

\def\vol{\operatorname{Vol}}
\def\AP{\operatorname{AP}} 
\def\APS#1#2#3{{\AP^{#1}_{ #2,#3}}}
\def\deltaone#1#2{{\epsilon_{#1,#2}}}
\def\deltaother#1#2{{\delta_{#1,#2}}} \def\D{\mathrm d}
\def\vol{\operatorname{Vol}}
\def\rhobar{\bar \rho}
\def\Jminus{J^{-}}
\def\PolyVar{N}
\def\error{\varepsilon} \def\dev{\zeta}
\def\r{\lambda_{\mathcal F}(\rho)}
\def\wpar{w}
\def\injconst{I}
  
\begin{document}

\title[Effective Equidistribution of Nilflows]%
{On effective equidistribution \\ for higher step nilflows}

\author{Livio Flaminio} \author{Giovanni Forni}

\address{Math\'ematiques\\
  Universit\'e Lille 1\\
  F59655 Villeneuve d'Asq CEDEX\\
  FRANCE}

\address{Department  of Mathematics\\
  University of Maryland \\
  College Park, MD USA}

\email {livio.flaminio@math.univ-lille1.fr} \email
{gforni@math.umd.edu}  

\date{\today}
    
\begin{abstract}
The main goal of this paper is to obtain optimal estimates on
the speed of equidistribution of nilflows on higher step nilmanifolds.
Under a Diophantine condition on the frequencies of the toral projection
of the flow, we prove that for almost all points on the nilmanifold
orbits become equidistributed at polynomial speed with exponent which
decays quadratically as a function of the number of steps. The main
novelty is the introduction of new techniques of renormalization
(rescaling) in absence of a truly recurrent renormalization dynamics.
Quantitative equidistribution estimates are derived from bounds on the
scaling of invariant distributions (in Sobolev norms) and on the
geometry of the nilmanifold under the rescaling.
\end{abstract}

\maketitle

{\small\tableofcontents}

\section{Introduction}
\label{sec:intro}

In this paper we prove estimates on the speed of ergodicity for a
class of nilflows on higher step nilmanifolds, under Diophantine
conditions on the frequencies of their toral projections. By the
classical theory, nilflows with minimal, hence uniquely ergodic, toral
projection are uniquely ergodic. Their ergodic theory is closely
related to questions in number theory, in particular the problem of
bounds on exponential sums along polynomial sequences, known as
\emph{Weyl sums}.  By a relatively recent far reaching generalisation
by B.~Green and T.~Tao~\cite{GreenTao} of classical results and
methods, all orbits of Diophantine nilflows on any nilmanifold
become equidistributed at polynomial speed, but the exponent in their theorem
is far from optimal and presumably decays exponentially as a function
of the number of steps of the nilmanifold.

We are especially concerned with the optimal speed of equidistribution
for nilmanifolds of higher step. Our main result proves
equidistribution at a polynomial speed with exponent which decays
{\it quadratically} as a function of the number of steps. However, we
only establish our result for \emph{almost all points} on the
nilmanifold. In other terms, we prove a rather sharp result on
quantitative ergodicity, but for reasons that will be explained below,
we are unable to prove an effective \emph{unique} ergodicity
theorem. Our result can be better appreciated by comparing its
application to Weyl sums, stated below, with recent results proved by
T.~D.~Wooley~\cite{wooley} with methods of analytic number theory. In
fact, we derive a virtually identical bound on the growth of Weyl sums
for polynomials of higher degree under a comparable (but somewhat
stronger) Diophantine condition on the leading coefficient. However,
our result only holds for almost all choices of coefficients of lower
degree.

We do not consider general nilmanifolds, but only a class of them
which we call \emph{quasi-Abelian}. This class is in a sense the
simplest class of nilmanifolds of arbitrarily high step. A
quasi-Abelian nilpotent group is a nilpotent group which contains an
Abelian normal subgroup of codimension one. This class of
quasi-Abelian nilpotent groups is chosen since on the one hand their
irreducible unitary representations, which can be described as an
application of Kirillov theory, are particularly simple, and on the
other hand this class contains groups of arbitrarily high step, which
allow us to derive results on Weyl sums for polynomials of
arbitrarily high degree.  There is no reason in principle that
prevents a generalisation to arbitrary nilflows on arbitrary
nilmanifolds, except that require estimates in representations would
be very complicated and difficult to carry out.

\smallskip Let $G^{(\lstep)}_n$ denote a quasi-Abelian $\lstep$-step
nilpotent group on $n+1$ generators, let $\Gamma^{(\lstep)}_n \subset
G^{(\lstep)}_n$ be a lattice and let $M^{(\lstep)}_n:=
\Gamma^{(\lstep)}_n \backslash G^{(\lstep)}_n$ denote the
corresponding nilmanifold.  Since the Abelianisation
$G^{(\lstep)}_n/ [G^{(\lstep)}_n, G^{(\lstep)}_n]$ of the   group~$G$ is isomorphic to
$\R^{n+1}$, there is a natural projection $M^{(\lstep)}_n \to {\bar
  M}^{(\lstep)}_n$ onto an $n+1$-dimensional torus.  By the classical
theory, a nilflow $M^{(\lstep)}_n$ is uniquely ergodic if and only if
the projected toral flow on ${\bar M}^{(\lstep)}_n$ has rationally
independent frequencies.

Effective equidistribution results require a Diophantine condition on
the frequencies. We formulate below our condition (see Definition
\ref{def:newD}).  Let $\Vert \cdot \Vert_\Z$ denote the distance from
the nearest integer. For any $\alpha:=(\alpha_1, \dots, \alpha_n) \in
\R^n$, for any $N\in \N$ and for every $\delta>0$, let
$$
\mathcal R^{(n)}_\alpha(N, \delta)= \{ r\in [-N, N] \cap
\Z\setminus\{0\} \vert \max_{1\leq i\leq n} \vert \Vert r\alpha_i
\Vert_\Z \leq \delta^{\frac{1}{n}}\}\,.
$$
For every $\nu \geq 1$, let $D_n(\nu) \subset (\R\setminus\Q)^n$ be
the subset defined as follows: $\alpha \in D_n(\nu) $ if and only if
there exists a constant $C(\alpha)>0$ such that, for all $N\in \N$ and
for all $\delta>0$,
$$
\# \mathcal R^{(n)}_\alpha(N, \delta) \leq C(\alpha) \max\{
N^{1-\frac{1}{\nu}}, N\delta \}\,.
$$

For a single frequency the above Diophantine condition is a
consequence of the well-known following Diophantine condition.  
A number $a\in \R\setminus \Q$ is called Diophantine of exponent $\nu
\geq 1$ if there exists a constant $c(a)>0$ such that the following
bound holds:
$$
\Vert N a \Vert_\Z \geq \frac{c(a)}{N^\nu} \,, \quad \text{ \rm for
  all } N \in \N\setminus \{0\}\,.
$$
By an elementary argument based on continued fractions, it can be
proved that our set $D_1(\nu)$ introduced above contains all
Diophantine irrational numbers of Diophantine exponent $\nu\geq 1$,
according to the above classical definition. In higher dimension
$n\geq 2$ our set $D_n(\nu)$ contains the set of simultaneously
Diophantine vectors of sufficiently small exponent, hence we can prove
that the set $D_n(\nu)$ has full measure for sufficiently large
$\nu\geq 1$ (see Lemma~\ref{lem:oldnewD}).

\smallskip Our main result is the following bound on the speed of
convergence of ergodic averages along almost all orbits of Diophantine
quasi-Abelian nilflows.

\begin{theorem}
  \label{thm:main_intro}
  Let $(\phi^t_\alpha)$ be a nilflow on a $\lstep$-step quasi-Abelian
  nilmanifold $M^{(\lstep)}_n$ on $n+1$ generators such that the
  projected toral flow $({\bar \phi}^t_\alpha)$ is a linear linear
  flow with frequency vector $\alpha := (1, \alpha_1, \dots, \alpha_n)
  \in \R\times \R^n$.  Under the assumption that the vector
  $\alpha':=(\alpha_1, \dots, \alpha_n) \in D_n(\nu)$ for some $\nu
  \leq \lstep/2$, there exists a (Sobolev) norm $\Vert \cdot \Vert$ on
  the space $C^\infty(M^{(\lstep)}_n)$ of smooth function on
  $M^{(\lstep)}_n$ and for every $\epsilon>0$ there exists a positive
  measurable function $K_\epsilon \in L^p (M^{(\lstep)}_n)$ for all
  $p\in [1, 2)$, such that the following bound holds. For every smooth
  zero-average function $f\in C^\infty(M^{(\lstep)}_n)$, for almost
  every $x\in M^{(\lstep)}_n$ and for every $L\geq 1$,
  $$
  \vert \frac{1}{L} \int_0^L f \circ \phi^t_\alpha (x) \,\D{t} \vert \leq
  K_\epsilon (x) L^{- \frac{2}{3 (\lstep +2n-2)(\lstep-1)}
    +\epsilon}\, \Vert f \Vert \,.
 $$
\end{theorem}
  
The above theorem is best appreciated by its main corollary on Weyl
sums, in comparison with available results proved by analytic number
theory. We recall that given a polynomial $P_\lstep(N)$ of degree
$\lstep\geq 2$, written as $P_\lstep(N):= \sum_{j=0}^{\lstep} a_j N^j
$ the corresponding Weyl sums are the exponential sums
$$
W (a_\lstep, \dots, a_0, N):= \sum_{n=0}^{N-1} \exp (2\pi \imath
P_\lstep (n))\,.
$$
By the well-known relation between Weyl sums and nilflows, we derive
the following bound.

\begin{corollary}
  \label{cor:main_intro}
  Let $a_\lstep\in \R\setminus \Q$ be a Diophantine number of exponent
  $\nu \leq \lstep/2$.  For every $\epsilon>0$, there exists a
  measurable positive function $K_\epsilon\in L^p(\T^{\lstep-2})$, for
  all $p\in [1, 2)$, such that the following bound holds.  For all
  $a_0, a_1 \in \R^2$, for almost all $(a_2, \dots, a_{\lstep-1}) \in
  \R^{\lstep-2}$ and for every $L\geq 1$,
  $$
  \vert W (a_\lstep, \dots, a_0, N) \vert \leq K_\epsilon(a_2, \dots,
  a_{\lstep-1}) N^{ 1 - \frac{2}{3 \lstep(\lstep-1)} +\epsilon}\,.
 $$
\end{corollary}

As we have anticipated above, this result should be compared with
bounds proved by T.~D.~Wooley (see in particular Theorem 1.5 in
\cite{wooley}). From Wooley's theorem, one can easily derive a
\emph{uniform} bound on Weyl sums with essentially the same exponent
as above, but for \emph{all} $(a_0, \dots, a_{\lstep-1}) \in
\R^\lstep$, also under a Diophantine condition of exponent $\nu \leq
\lstep-1$.  Wooley's theorem comes as a refinement and sharpening of
techniques that have been developed over the span of a century, and
especially since Vinogradov's contribution in the 30's, and in
Wooley's words come ``within a stone's throw of the sharpest possible
bounds'' (for large degree $\lstep \geq 8$, otherwise the classical
Weyl's bound is still unsurpassed in general).

\medskip Our approach is significantly different from the methods of
analytic number theory and the circle methods of Wooley's
\cite{wooley} as well as from the classical methods based on induction
on the number of steps and on Van der Corput lemma greatly refined
recently in the work of Green and Tao \cite{GreenTao}. In their work
quantitative equidistribution with polynomial speed is proved for
general nilpotent sequences.  However, there is no effort to determine
the optimal exponent as a function of the number of steps. In fact,
since their methods are a generalisation of the Weyl's method, it is
reasonable to expect that the best exponent available in their work
would decay \emph{exponentially} with the number of steps (we recall
that the classical Weyl bound on Weyl sums holds with exponent $1-
1/2^{\lstep-1}$ for polynomial sequences of degree $\lstep \geq 2$).

In this paper we generalise the renormalisation method of our earlier
work \cite{flafor:heis} on Heisenberg nilflows (the $2$-step nilpotent
case, which corresponds to polynomial sequences of degree $2$). Our
main goal is to develop an approach which is not restricted to
nilflows and can applied to quantitative equidistribution problems for
more general parabolic flows.  The method is based on estimates on the
scaling of invariant distributions for the flow under a deformation of
the nilmanifold. In the Heisenberg case, it is possible to define a
deformation given by a one-parameter group of automorphism of the Lie
algebra, which implies that the deformation group induces a
renormalisation group action on a suitable moduli space. This is not
surprising since it has been known for a long time that quadratic
polynomial sequences (as well as linear ones) have self-similarity
properties.

In the higher step case, we were unable to define an effective
renormalisation group dynamics and our deformation does not come from
a group of Lie algebra automorphisms. As a consequence, it does not
induce a recurrent flow on a moduli space. However, quantitative
equidistribution estimates can still be derived from bounds on the
scaling of invariant distributions (in Sobolev norms) and on the
geometry of the nilmanifold under the deformation.  Given that no
(recurrent) renormalisation is available, the task of proving
geometric bounds is in fact the most delicate part of the
argument. Our proof is based on average estimates and on a
Borel-Cantelli argument, which explains why our geometric bounds, and
consequently our equidistribution results, only holds almost
everywhere. In fact, the deformation is chosen in a way that optimises
the scaling of invariant distributions. Sobolev estimates on the
scaling of invariant distributions can be proved by an analysis of the
cohomological equation and of invariant distributions in every
irreducible representation of the (quasi-Abelian) Lie group (the
quasi-Abelian case is in fact much simpler than the general case
treated in~\cite{MR2261071} and allows us to prove explicit sharp
bounds). This analysis leads to the polynomial decay of ergodic
averages (with the exponent given in our main theorem) for all
``good'' points for which uniform bounds on the degeneration of the
geometry hold. It is a plausible conjecture that in fact under a
Diophantine condition on the nilflow \emph{all} points of the
nilmanifold, not just almost all, are ``good''.

The degeneration of the geometry at a given point on the nilmanifold
is measured in terms of a notion of \emph{average width} of an orbit
segment of a nilflow, with respect to a basis of the Lie algebra (see
Definition~\ref{def:av_width}).  This notion arises from a new version
of the Sobolev trace theorem adapted to Sobolev estimates on orbit
segments of flows (see Theorem~\ref{thm:BA_apriori}). From a dynamical
standpoint, the average width is a measure of the frequency of close
returns along an orbit segment. Roughly, the width of an orbit segment
is the maximal transverse volume of a rectangular tubular
neighbourhood, measured with respect to a given, possibly deformed,
transverse metric. The inverse of the square root of the width bounds
the constant in the Sobolev trace theorem, which provides an a priori
Sobolev bound for the distribution given by an orbit segment. The
average width is an averaged version of the width, with the average
taken along the orbit segment itself. The tubular neighbourhood is
allowed to have a rectangular cross section of variable transverse
volume. The average width is a defined as the reciprocal of the
average along the orbit segment of the reciprocal of the transverse
area of the tubular neighbourhood.  The point is that if the very close
returns of the orbit segment are not too frequent then the average
width can still be (uniformly) bounded, while the width may be
arbitrarily large. We prove that the average width still gives an
upper bound for the constant in the Sobolev trace theorem.

The paper is organised as follows. In section~\ref{sec:Quasi-Abelian}
we define quasi-Abelian groups, nilmanifolds and nilflows and recall
the well-known relations between Weyl sums and ergodic averages of
nilflows. In section~\ref{sec:Sob_trace} we introduce the notion of
average width and prove a Sobolev trace theorem. In
section~\ref{sec:CE} we carry out Sobolev estimates on solutions of
the cohomological equation and on invariant distributions as an
application of Kirillov theory of unitary representations of nilpotent
groups. In section~\ref{sec:width} we prove bounds on the average
width of orbit segment of nilflows, in mean over the initial point of
the orbit. Finally, in section~\ref{sec:last_step} we prove an
effective equidistribution theorem for ``good'' points and derive that
``good'' points form a set of full measure, by a Borel-Cantelli
argument based on the estimates in mean on the average
width. Estimates on Weyl sums (for almost all lower degree
coefficients) then follow from our equidistribution theorem.

\smallskip
\noindent \textit{Acknowledgements} Livio Flaminio was supported in
part by the Labex~CEMPI (ANR-11-LABX-07). Giovanni Forni was supported
by NSF grant DMS~1201534. This work was completed at the Isaac Newton
Institute in Cambridge, UK. The authors wish to thank the Institute and
the organisers of the programme \emph{Interactions between Dynamics of
  Group Actions and Number Theory} for their hospitality.

\section{Quasi-Abelian nilpotent flows and Weyl sums}
\label{sec:Quasi-Abelian}
In this section we introduce quasi-Abelian nilpotent Lie algebras and
groups and collect basic material on their structure. We recall
classical Weyl sums and a well-known reduction of Weyl sums to ergodic
integrals of quasi-Abelian nilpotent flows \cite{MR603625}.

\subsection{Quasi-Abelian nilpotent Lie algebras and groups}
\label{ss:QAalgebras}

\subsubsection{Quasi-Abelian nilpotent Lie algebras} In this paper a
nilpotent Lie algebra $\mathfrak g$ will be called
\emph{quasi-Abelian} if it has a \emph{maximal} Abelian ideal
$\mathfrak a$ of codimension one.  We denote by $\adim$ the dimension
of~$\mathfrak a$.  Any quasi-Abelian nilpotent Lie algebra admits
bases
\begin{equation}
  \label{eq:Jbasis_1}
  \{\xi, \eta^{(1)}_1, \dots , \eta^{(1)}_{i_1}, \dots, 
  \eta^{(n)}_1, \dots , \eta^{(n)}_{i_n}\} \,,
\end{equation}
such that the only non-trivial commutation relations are of the form
\begin{equation}
  \label{eq:Jbasis_2}
  [\xi,\eta^{(m)}_i]=\eta^{(m)}_{i+1}, 
  \quad \text{ \rm for all } \,m=1, \dots, n, 
  \text{ and } \,   i=1,\dots,i_m-1\,. 
\end{equation}
Such bases can be constructed as follows. Let $\xi \not \in \mathfrak
a$, let
\begin{equation*}
  \mathfrak a= \mathfrak a^{(1)} \oplus \dots \oplus  \mathfrak a^{(n)}\,,
\end{equation*}
be the splitting of the vector space $\mathfrak a$ into generalised
eigenspaces of the linear map $\text{ \rm ad}(\xi): \mathfrak a \to
\mathfrak a$ and for every $m=1, \dots, n$ let
\begin{equation}
  \label{eq:a_basis}
  \{ \eta^{(m)}_1, \dots , \eta^{(m)}_{i_m}\} \subset \mathfrak a^{(m)}\, 
\end{equation}
be a Jordan basis for the linear map $\text{ad}(X): \mathfrak a^{(m)}
\to \mathfrak a^{(m)}$.  For this reason, bases satisfying the
commutation relations~\eqref{eq:Jbasis_2} will be called \emph{Jordan
  bases}.

Any quasi-Abelian nilpotent Lie algebra with a Jordan basis of the
form \eqref{eq:Jbasis_1}, \eqref{eq:Jbasis_2} is generated (as a Lie
algebra) by the system $\{\xi, \eta^{(1)}_1, \dots, \eta^{(n)}_1\}$,
hence it has $n+1$ generators, and it has an Abelian ideal $\mathfrak
a$ of dimension $\adim= i_1+ \dots +i_n$ generated (as a vector space)
by the system in formula \eqref{eq:a_basis}, hence the Lie algebra
$\mathfrak g$ has dimension $\adim+1$. Finally, the Lie algebra is
$\lstep$-step nilpotent for
\[
\lstep := \max\{i_1, \dots, i_n\}.
\]
In this paper we are interested in quasi-Abelian nilpotent Lie
algebras of step $\lstep\geq 3$.

\begin{notation}
  It will be convenient to consider the set of indices
  \[
  J:=\{(m,i)\mid m=1, \dots, n, \, i=1,\dots,i_{m}\}.
  \]
  endowed with the lexicographic order. By $\Jminus$ we denote the
  ordered subset $\{(m,j)\in J\mid i\le i_m-1\}$.
\end{notation}


\begin{definition}
  \label{def:genJbasis}
  \begin{sloppypar}
    An ordered basis $(X,Y):= (X,\dots, Y_i^{(m)},\dots)_{(m,i)\in J}$
    of the quasi-Abelian Lie algebra $\mathfrak g$ is a
    \emph{generalised Jordan basis} if $X \not \in \mathfrak a$,
    $Y_i^{(m)} \in \mathfrak a$ for all $(m,i)\in J$ and, for some
    strictly positive reals $c=(c_i^{(m)})_{(m,i)\in \Jminus}$, the
    following commutation relations hold true
    \begin{equation*}
      [X,Y_i^{(m)}]=c_{i+1}^{(m)} Y^{(m)}_{i+1},
      \qquad \text{\ for all\ }(m,i)\in \Jminus,
    \end{equation*} 
    (all other commutators being equal to zero). The constants $c$ are
    are called the \emph{structural constants of the basis}.
  \end{sloppypar}
\end{definition}

 
 


 \subsubsection{Quasi-Abelian nilpotent Lie groups}
 A nilpotent group $G$ will be called a \emph{quasi-Abelian}
 $\lstep$-step nilpotent Lie group (on $n+1$ generators) if it is a
 simply connected, connected Lie group whose Lie algebra $\mathfrak g$
 is a quasi-Abelian $\lstep$-step nilpotent Lie group (on $n+1$
 generators), as above.  A quasi-Abelian nilpotent Lie group $G$ has
 an Abelian normal subgroup $A$ of codimension one, namely
 the exponential of the codimension one, Abelian ideal $\mathfrak
 a $ of $\mathfrak g$.

 The $\lstep$-step quasi-Abelian nilpotent groups $G$ on $n+1$
 generators have also another description.  Let $(i_1, \dots,i_n) \in
 \Z_+^n$ be positive integers such that $\lstep=\max
 \{i_1,\dots,i_n\}$ and let $\adim =i_1+ \dots +i_n$.

 For any $j\in\Z_+$, let $h_j:\R \to \text{\rm Aut}(\R^j)$ be the
 (unique) one-parameter group of automorphisms of $\R^j$ such that
 \begin{equation}
   \label{eq:h_j}  
   h_j(1) (s_1, \dots, s_j) = 
   (s_1, s_2 + s_1,\dots, s_i +s_{i-1},\dots,s_j+ s_j)
 \end{equation}
 and let $h: \R \to \text{\rm Aut}(\R^\adim)$ be the product
 one-parameter group
 \begin{equation}
   \label{eq:h}  
   h = h_{i_1} \times \dots \times  h_{i_n}  \quad 
   \text{ \rm on }  \, \R^\adim= \R^{i_1} \times \dots \R^{i_n} \,.
 \end{equation}
 Let $G$ be the twisted product $ \R\ltimes_h \R^\adim$.  We can view
 $G$ as an algebraic subgroup of the real algebraic group $ \text {\rm
   GL}_d(\R)\ltimes \R^\adim $.  Since $h(\Z)\subset \text{\rm Aut}
 (\Z^\adim)$, the twisted product $\Gamma:=\Z\ltimes_{h|\Z} \Z^\adim$
 is a well-defined, Zariski dense, discrete subgroup of $G$, hence a
 lattice of $G$.  It is generated by elements
 \begin{equation*}
   x, y^{(1)}_1, \dots, y^{(1)}_{i_1},  \dots, y^{(n)}_1, \dots,    y^{(n)}_{i_n}\,,
 \end{equation*}
 such that the only non-trivial commutation relations are
 \begin{equation*}
   x y^{(m)}_i x^{-1}=   
   \begin{cases} \,y^{(m)}_i\,y^{(m)}_{i+1}\,, 
     &\text{for $ \,\, 1\le i< i_m $}\,\\
     \,y^{(m)}_{i_m}\,,&\text{for $\,\,i=i_m $.}
   \end{cases}
 \end{equation*}
 (We have taken for $x$ the element $(1,(0,\dots,0)) \in \Gamma$ which
 acts by conjugation on $ \Z^\adim $ by the automorphism $h(1)$
 defined in \eqref{eq:h_j} and \eqref{eq:h}, and for elements
 $y^{(1)}_1, \dots, y^{(1)}_{i_1}$, $\dots$, $y^{(n)}_1, \dots,
 y^{(n)}_{i_n}$ the elements of the standard basis $ (0,(1,0,\dots,0))
 $, \dots, $ (0,(0,\dots,0,1)) $ of~$\{0\}\times\Z^\adim$).  The
 codimension one, Abelian normal subgroup $A\subset G$ is generated by
 the elements
 $$y^{(1)}_1, \dots,
 y^{(1)}_{i_1}, \dots, y^{(n)}_1, \dots, y^{(n)}_{i_n}\,.$$
 

 Let $\mathfrak g$ be the Lie algebra of $G$ and let $\log: G \to
 \mathfrak g$ the inverse of the exponential map $\exp: \mathfrak g
 \to G$.
 The elements
 \begin{equation}
   \label{eq:LieGen_1}
   \xi:=\log x\qquad 
   \tilde \eta_i^{(m)}:=\log y^{(m)}_i,\quad (m,i)\in J,
 \end{equation}
 form a basis of $\mathfrak g$ and satisfy the commutation relations
 \begin{equation}
   \label{eq:LieGen_2}
   [\xi,\tilde \eta_j^{(m)}]  = \sum_{i=j+1}^{i_m}
   \tfrac{(-1)^{i-j-1}}{i-j}
   \, \tilde \eta_i^{(m)} \,, \quad  (m,j)\in \Jminus\,,
 \end{equation}
 all other commutators being equal to zero.  We obtain a Jordan basis
 defining by induction
 \begin{equation}
   \label{eq:LieGen_3}
   \eta_1^{(m)} =\tilde \eta_1^{(m)}, \quad \eta_{i+1}^{(m)} = [\xi, \eta_i^{(m)}]
   \,, \quad (m,i)\in \Jminus.
 \end{equation}
 Thus $\mathfrak g$ is a quasi-Abelian $\lstep$-step nilpotent Lie
 algebra on $n+1$ generators, hence $G$ is a quasi-Abelian
 $\lstep$-step nilpotent Lie group on $n+1$ generators.

 Clearly, for all $m=1, \dots, n$ there exists strictly upper
 triangular rational matrices $R^{(m)},S^{(m)} \in M_{i_m}(\Q)$ such
 that
 \begin{gather}
   \label{eq:basis_change_1}
   \eta^{(m)}_j = \tilde \eta^{(m)}_j + \sum_{i=j+1} ^{i_m}
   R^{(m)}_{ij}
   \tilde \eta^{(m)}_j \\
   \label{eq:basis_change_2} \tilde \eta^{(m)}_j = \eta^{(m)}_j +
   \sum_{i=j+1} ^{i_m} S^{(m)}_{ij} \eta^{(m)}_j \,.
 \end{gather}
 for all $ j =1,\dots, i_m -1$. Thus via the
 formulas~\eqref{eq:basis_change_2} and by taking exponentials, we can
 associate a lattice of the quasi-Abelian nilpotent Lie group $G$ to
 each Jordan basis $(\xi,\eta)=(\xi, \dots, \eta_i^{(m)},\dots)$ of
 its Lie algebra $\mathfrak g$.

 \emph{Henceforth we shall assume that we have fixed once and for all
   a Jordan basis $(\xi, \eta)=(\xi, \dots, \eta_i^{(m)}, \dots)$ of
   the Lie algebra $\mathfrak g$ and we shall define $\Gamma$ to be
   the lattice generated by the system
   \begin{equation}
     \label{eq:lattice_def}
     \{ x:=\exp \xi, \, \dots , y_i^{(m)}:=\exp \tilde \eta_i^{(m)}, \dots\}\,,  
   \end{equation}
   where the elements $\tilde \eta_i^{(m)} \in \mathfrak g$ are given
   by the formulas~\eqref{eq:basis_change_2} and satisfy the
   commutation relations in formula~\eqref{eq:LieGen_2}.}

 \subsection{Quasi-Abelian nilmanifolds and flows}
 \label{ss:QAnilmanifolds}

\subsubsection{Quasi-Abelian nilmanifolds}
Since by construction the subgroup $ \Gamma$ is discrete and Zariski
dense in $G$ the quotient $ \Gamma \backslash G $ is a compact
nilmanifold.
  
\begin{definition}
  \label{def:QA_nilmanifold}
  The quotient $M= \Gamma \backslash G$ will be called a {\em
    quasi-Abelian $\lstep$-step nilmanifold on $n+1$ generators}.
\end{definition}

Observe that for any Jordan basis $(\xi,\eta)$ the centre $\mathfrak
z(\mathfrak g)$ of a quasi-Abelian $\lstep$-step nilpotent Lie algebra
$\mathfrak g$ is spanned by the system~$\{\eta^{(1)}_{i_1}, \dots,
\eta^{(n)}_{i_n}\}$ and therefore the system $(\xi,\dots,
\eta_i^{(m)},\dots)$, with $(m,i)\in \Jminus$, projects onto a Jordan
basis of the Lie algebra $\mathfrak g':=\mathfrak g /\mathfrak
z(\mathfrak g)$, which is quasi-Abelian $(\lstep-1)$-step nilpotent on
$n' \leq n$ generators.  At the group level, the centre $Z(G)$ of $G$,
that is the group
$$
\{\exp (t_1 \eta^{(1)}_{i_1} + \dots + t_n \eta^{(n)}_{i_n})\}
_{(t_1,\dots, t_n) \in \R^n}
$$
meets the lattice $\Gamma$ into the subgroup generated by the system
$$\{\exp( \eta^{(1)}_{i_1}), \dots,  \exp( \eta^{(n)}_{i_n})\}\,, $$ 
which is the centre $Z(\Gamma)$ of $\Gamma$. Hence $G':=G/Z(G)$ is a
quasi-Abelian $(\lstep-1)$-nilpotent Lie group on $n'\leq n$
generators and the subgroup $\Gamma':= \Gamma/Z(\Gamma)$ is a lattice
in $G$. In fact, the elements
$$x , y^{(1)}_1, \dots, y^{(1)}_{i_1-1}, \dots, 
y^{(n)}_1, \dots, y^{(n)}_{i_n-1}$$ project onto generators of the
lattice $\Gamma'$ in $G'$. The above discussion implies that a
quasi-Abelian $\lstep$-step nilmanifold $M$ has a structure of toral
bundle over a quasi-Abelian $(\lstep-1)$-step nilmanifold $M'$ the
fibres of this fibration being the orbits of the right action of the
centre $Z(G)$ on $M$.

We introduce two other important fibrations of the nilmanifold
$M$. Since the Abelianisation $G/[G, G ]$ of the group $G$ is
isomorphic to $\R^{n+1}$ and contains the subgroup $ \Gamma / [\Gamma,
\Gamma ]$ as a co-compact lattice we obtain the following fibration
\begin{equation}
  \label{eq:fibration_1}
  0 \to \T^{\adim-n} \to M 
  \xrightarrow{\operatorname{pr}_{1}} M_1\approx \T^{n+1}\to 0. 
\end{equation}
Another fibration arises from the canonical homomorphism $G\to
G/A\approx\<\exp \xi\>$; passing to the quotient by the corresponding
lattices, we see that $M$ is a torus bundle over a circle with
monodromy given by the map $h(1)$ of formulas~\eqref{eq:h_j} and
\eqref{eq:h}, that is,
\begin{equation}
  \label{eq:fibration_2}
  0 \to \torus  \to M 
  \xrightarrow{\operatorname{pr}_{2}} M_2\approx \T^1 \to 0\,.
\end{equation}


For all $m=1, \dots, n$ we denote by $\T^{i_m}_0 \subset M$ the
$i_m$-dimensional torus
\[
\Gamma \exp (s^{(m)}_1 \tilde\eta^{(m)}_1+\cdots +s^{(m)}_{i_m}
\tilde\eta_{i_m}), \quad (s^{(m)}_1, \dots, s^{(m)}_{i_m}) \in
\R^{i_m}.
\]
By construction the fibre $\torus_0$ of the
fibration~\eqref{eq:fibration_2} above the coset of the identity has a
product structure
\begin{equation}
  \label{eq:product_0}
  \torus_0 \approx \T^{i_1}_0 \times \cdots \times \T^{i_n}_0 \,.
\end{equation}
For $m=1, \dots, n$, let us denote $\bs^{(m)}:= (s^{(m)}_1, \dots,
s^{(m)}_{i_m}) \in (\R/\T)^{i_m}$ and $\bs =(\bs^{(1)}, \dots,
\bs^{(n)}) \in (\R/\Z)^\adim $.  The map
\begin{equation*}
  \bs^{(m)} \in (\R/\Z)^{i_m} \mapsto  \Gamma  
  \exp(\sum_{i=0}^{i_m}s^{(m)}_i \tilde\eta^{(m)}_i ) \in  \T^{i_m}_0
\end{equation*}
is a diffeomorphism, and so is the map
\begin{equation*}
  \bs \in (\R/\Z)^\adim \mapsto  \Gamma  
  \exp(\sum_{(m,i)\in J} s^{(m)}_i \tilde\eta^{(m)}_i ) \in \torus_0 .
\end{equation*}
Points on the tori $\T^{i_m}_0$ and $\torus_0 \subset M$ will be
denoted by their coordinates $\bs^{(m)} \in (\R/\Z)^{i_m} $ and $\bs
\in (\R/\Z)^\adim $.


By the definition~\eqref{eq:lattice_def} of the lattice $\Gamma$ the
left (and right) invariant volume form $\text{vol}_{\adim+1}$ on $G$,
normalised by the condition $\text{vol}_{\adim+1} (\xi,\dots,
\tilde\eta^{(m)}_i,\dots)=1$ pushes down to a right-invariant volume
form on $\omega$ on $M$, whose density yields a right-invariant
\emph{probability} measure $\meas$ on $M$.  Since the formulas
\eqref{eq:basis_change_1} and \eqref{eq:basis_change_2} imply that the
wedge products $\xi\wedge\cdots\wedge\tilde\eta^{(m)}_i \wedge \cdots$
and $\xi\wedge\cdots \wedge\eta^{(m)}_i \wedge \cdots$ coincide, we
conclude that the normalisation condition is equivalent to
\begin{equation*}
  \omega (\xi, \dots, \eta^{(m)}_i,\dots)=1\,.
\end{equation*}
\emph{Henceforth a quasi-Abelian nilmanifold will be equipped with the
  right invariant volume form $\omega$ satisfying the normalisation
  condition above, and with the associated probability measure
  $\meas$.}
\subsubsection{Quasi-Abelian nilflows}
For any element $X\in \mathfrak g$, let $(\phi_X^t)_{t\in \R}$ denote
the flow on $M$ generated by $X$, that is, the flow given by right
multiplication by the one-parameter subgroup $(\exp(t X) )_{t\in \R}$:
\begin{equation}
  \label{eq:nilflows}
  \phi_X^t(\Gamma g) = \Gamma g \exp(t X), 
  \quad \text{\ for all \ } \Gamma g \in M=\Gamma \backslash G.
\end{equation}
Clearly this flow has an interesting dynamics only if $X\not\in
\mathfrak a$. Otherwise it is a linear flow on a toral fibre of the
fibration in formula~\eqref{eq:fibration_1}.  It is a well-known
classical result that the flow $(\phi^t_X)$ is ergodic, uniquely
ergodic and minimal if and only if the projection of $X$ in the
Abelianised Lie algebra $\mathfrak g/[\mathfrak g,\mathfrak g]$ is
rationally independent of the lattice $\Gamma/[\Gamma,\Gamma]$.

Let $\alpha:=(\alpha_i^{(m)}) \in \R^{J}$ and let
\begin{equation}
  \label{eq:X_alpha} 
  X_{\alpha}:=\log \Big[\,x^{-1}\exp\Big( \sum_{(m,i)\in J}  \alpha_i^{(m)} \tilde \eta_i^{(m)}\Big)\,\Big]\,.
\end{equation} 
Return maps (to global transverse sections) of the flow generated by
the above vector fields are readily computed as follows.

For $\theta \in \T^1$ let
$\torus_\theta=\operatorname{pr}^{-1}_{2}(\{\theta\})$ denote the
toral fibre above $\theta\in \T^1$ of the fibration
$\operatorname{pr}_{2}$ in formula~\eqref{eq:fibration_2}.

For $m=1,\dots, n$ let $\Z^{i_m}_\theta$ be the lattice defined as
follows
\begin{equation}
  \label{eq:theta_lattice}
  \bs^{(m)}  \in \Z^{i_m}_\theta \quad\Leftrightarrow
  \quad \sum_{i=1}^i s^{(m)}_i \tilde \eta^{(m)}_i \in  \mathfrak a^{(m)}\cap \Ad(e^{-\theta \xi}) (\Gamma)\,.
\end{equation}
If $\Z^\adim_\theta= \Z^{i_1}_\theta \times \dots \times
\Z^{i_n}_\theta $ then we have
\[
\torus_\theta = \{\Gamma\exp(\bs \tilde \eta) \exp(\theta \xi) \mid
\bs\in \R^\adim/\Z^\adim\} = \{\Gamma \exp(\theta \xi)\exp(\bs \tilde
\eta)\mid \bs \in \R^\adim/\Z^\adim_\theta\}.
\]
For all $m=1, \dots, n$, let $\T^{i_m}_\theta \subset \torus_\theta$
be the sub-torus defined as follows:
\begin{equation}
  \begin{aligned}
    \label{eq:theta_subtorus}
    \T^{i_m}_\theta &= \{\Gamma\exp(\bs^{(m)} \tilde \eta^{(m)})
    \exp(\theta \xi) \mid
    \bs^{(m)} \in \R^{i_m}/\Z^{i_m}\} \\
    &= \{\Gamma \exp(\theta \xi)\exp(\bs^{(m)} \tilde \eta^{(m)})\mid
    \bs^{(m)} \in \R^{i_m}/\Z^{i_m}_\theta\}\,.
  \end{aligned}
\end{equation}
By construction, there exists a product decomposition
\begin{equation}
  \label{eq:product_theta}
  \torus_\theta \approx \T^{i_1}_\theta \times \cdots \times \T^{i_n}_\theta \,.
\end{equation}
The torus~$ \torus_\theta $ is a global section of the nilflow
$(\phi_{X_\alpha}^t)_{t\in\R}$ on $M$, generated by $X_{\alpha}\in
\mathfrak g$ (see formula \eqref{eq:nilflows}), and the product
decomposition in formula~\eqref{eq:product_theta} is
$\{\phi_{X_\alpha}^t\}$-invariant.  The lemma below, which is
classical (see \cite{MR603625}), makes explicit its return maps.

\begin{lemma}
  \label{lem:return_maps} The flow $(\phi_{X_\alpha}^t)_{t\in\R}$ on
  $M$ is isomorphic to the suspension of its first return map
  $\Phi_{\alpha,\theta}:\torus_\theta \to \torus_\theta$, hence all
  return times are constant integer-valued functions on
  $\torus_\theta$.  The first return map is a product
  $$
  \Phi_{\alpha,\theta}\approx \Phi_{\alpha^{(1)},\theta} \times \dots
  \times \Phi_{\alpha^{(n)},\theta} \quad \text{ \rm on } \quad
  \T^{i_1}_\theta \times \cdots \times \T^{i_n}_\theta \,,
  $$
  and for every $m=1, \dots, n$ the factor map $
  \Phi_{\alpha^{(m)},\theta}$ is given in the coordinates $\bs^{(m)}
  \in \R^{i_m}\mod\Z^{i_m}_\theta$ by the formulas
  \begin{equation}
    \label{eq:first_return_maps}
    \begin{aligned}
      \Phi_{\alpha^{(m)},\theta}(\bs^{(m)}) &=
      \Phi_{\alpha^{(m)}}(\bs^{(m)}) = (s^{(m)}_1+\alpha^{(m)}_1,
      \dots, \\ &s^{(m)}_j + s^{(m)}_{j-1} + \alpha^{(m)}_j ,\dots,
      s^{(m)}_{i_m}+ s^{(m)}_{i_m-1} +\alpha^{(m)}_{i_m} ) \,.
    \end{aligned}
  \end{equation}
  For any $\return \in \Z$, the $\return $-th return map is a product
  $$
  \Phi^\return_{\alpha,\theta}\approx \Phi^N_{\alpha^{(1)},\theta}
  \times \dots \times \Phi^N_{\alpha^{(n)},\theta} \quad \text{ \rm on
  } \quad \T^{i_1}_\theta \times \cdots \times \T^{i_n}_\theta\,,
  $$
  and for every $m=1, \dots, n$ the factor map $
  \Phi^\return_{\alpha^{(m)},\theta}$ is given in the coordinates
  $\bs^{(m)} \in \R^{i_m}\mod\Z^{i_m}_\theta$ by the formulas
  \begin{equation}
    \label{eq:return_maps}
    \begin{aligned}
      \Phi^\return_{\alpha^{(m)}, \theta}(\bs^{(m)})=
      \big(&s^{(m)}_1+\return\,\alpha^{(m)}_1,
      s^{(m)}_2+\return(s^{(m)}_1+\alpha^{(m)}_2)+
      \tbinom{\return}{2}\,\alpha^{(m)}_1, \\ \dots, &s^{(m)}_{i_m}+
      \sum_{i=1}^{i_m-1}
      \tbinom{\return}{i}(s^{(m)}_{i_m-i}+\alpha^{(m)}_{i_m-i+1}) +
      \tbinom{\return}{i_m}\alpha^{(m)}_1\big)\,.
    \end{aligned}
  \end{equation}
\end{lemma}
\begin{proof}
  By construction, the factorisations in formula \eqref{eq:product_0}
  and, more generally, in formula \eqref{eq:product_theta} are induced
  by the splitting of the vector space $\mathfrak a$ into generalised
  eigenspaces of linear map $\text{ \rm ad}(X):\mathfrak a \to
  \mathfrak a$. It follows that the all time-$t$ maps of the flows
  $(\phi_{X_\alpha}^t)_{t\in\R}$ on $\torus_\theta$ are isomorphic to
  the products of their restrictions to the tori $\T^{i_m}_\theta
  \subset \torus_\theta$. For any given $m=1, \dots, n$, let us
  compute the restriction of the time-$1$ map of the flow to the torus
  $\T^{i_m}_\theta\subset \torus_\theta$.

  Let $\Gamma \exp(\theta \xi)\exp( \sum_{j=1}^{i_m} s^{(m)}_j \tilde
  \eta^{(m)}_j) \in \T^{i_m}_\theta$. We have
  \begin{equation*}
    \begin{aligned}
      \exp( \sum_{j=1}^{i_m}& s^{(m)}_j \tilde \eta^{(m)}_j) \exp (
      X_{\alpha}) = \exp( \sum_{j=1}^{i_m} s^{(m)}_j \tilde
      \eta^{(m)}_j)
      x^{-1} \exp( \sum_{j=1}^{i_m} \alpha^{(m)}_j \tilde \eta^{(m)}_j) \\
      &= x^{-1} \exp \bigl[(s^{(m)}_1 +\alpha^{(m)}_1) \tilde
      \eta^{(m)}_1 + \sum_{j=1}^{i_m-1} (s^{(m)}_j+ s^{(m)}_{j+1} +
      \alpha^{(m)}_{j+1}) \tilde \eta^{(m)}_{j+1} \bigr] \,.
    \end{aligned}
  \end{equation*}
  In fact, the following identity holds:
  \begin{equation*}
    \begin{aligned}
      x \exp ( \sum_{j=1}^{i_m} s^{(m)}_j \tilde \eta^{(m)}_j) x^{-1}
      &= \exp \Bigl[ e^{\text{ \rm ad}(\xi)}
      (\sum_{j=1}^{i_m} s^{(m)}_j \tilde \eta^{(m)}_j)\Bigr] \\
      &= \exp \bigl[ s^{(m)}_1 \tilde \eta^{(m)}_1 +
      \sum_{j=1}^{i_m-1} (s^{(m)}_j +s^{(m)}_{j+1}) \tilde
      \eta^{(m)}_{j+1} \bigr]\,.
    \end{aligned}
  \end{equation*} Since $x\in \Gamma$, it follows that
  \begin{equation*}
    \begin{aligned}
      &\Gamma \exp(\theta \xi)\exp( \sum_{j=1}^{i_m} s^{(m)}_j \tilde
      \eta^{(m)}_j) \exp ( X_{\alpha})\\ &= \Gamma \exp(\theta
      \xi)\exp \bigl[ (s^{(m)}_1+\alpha^{(m)}_1) \tilde \eta^{(m)}_1 +
      \sum_{j=1}^{i_m-1} (s^{(m)}_j + s^{(m)}_{j+1}
      +\alpha^{(m)}_{j+1}) \tilde \eta^{(m)}_{j+1} \bigr]\,.
    \end{aligned}
  \end{equation*}
  The above formula implies that $t=1$ is a return time of the
  restriction of the flow $(\phi_{X_\alpha}^t)_{t\in\R}$ to~$
  \T^{i_m}_\theta \subset M$, for all $m=1, \dots, n$, and the
  map~\pref{eq:first_return_maps} is the corresponding return map. In
  addition, $t=1$ is the first return time, since it is the first
  return time of the projection onto $M^{(\lstep)} \approx \T^{n+1}$
  of the restriction of the flow $(\phi_{X_\alpha}^t)_{t\in\R}$ to the
  torus $\T^{i_m}\subset M$.
  
  Finally, formula~\eqref{eq:return_maps} for the $\return$-th return
  map follows from formula~\eqref{eq:first_return_maps} by induction
  on $\return \in \N$.
\end{proof}

\subsection{Weyl sums as ergodic integrals}
\label{sec:Weyl_sums}

Let $P_\lstep:=P_\lstep(\PolyVar)\in \R[\PolyVar]$ be a polynomial of
degree $\lstep\geq 2$:
\[
P_\lstep(\PolyVar):= \sum_{j=0} ^\lstep a_j \PolyVar^j \,.
\]
A \emph{Weyl sum} of degree $\lstep\geq 2$ is the sum
\begin{equation*}
  W(P_\lstep,f;\return)  = 
  \sum_{\ell=0}^{\return -1}   f \bigl( P_\lstep(\ell)\bigr) \,,  
\end{equation*}
for any $\return \in \N$ and for any smooth periodic function $f \in
C^\infty(\T^1)$.  Classical (complete) Weyl sums are obtained as a
particular case when the function $f$ is the exponential function,
that is,
$$
f(s)= e(s):=\exp (2\pi \imath s) \,, \quad s\in \T^1\,.
$$

For any $(\alpha,\bs) \in \R^\lstep \times \R^\lstep/\Z^\lstep$, let
$P_\lstep (\alpha,\bs,\PolyVar) \in \R[\PolyVar]$ be the polynomial of
degree $\lstep \geq 1$ defined (modulo $\Z$) as follows:
\begin{equation*}
  P_\lstep (\alpha,\bs, \PolyVar) :=  
  \tbinom{\PolyVar}{\lstep}\alpha_1 +  \sum_{j=1}^{\lstep-1} 
  \tbinom{\PolyVar}{j}(s_{\lstep-j}+\alpha_{\lstep-j+1})  + s_\lstep \,.
\end{equation*}
The following elementary result holds.

\begin{lemma}
  \label{lem:coefficients}
  The map $(\alpha,\bs) \to P_\lstep(\alpha,\bs,\PolyVar)$ sends
  $\R^\lstep \times \R^\lstep/\Z^\lstep$ onto the space $\R[\PolyVar]$
  of real polynomials (modulo $\Z$) of degree $\lstep \geq 1$. The
  leading coefficient $a_\lstep\in \R$ of the polynomial
  $P_\lstep(\alpha,\bs,\PolyVar)$ is given by the formula:
  \[
  a_\lstep = \frac{ \alpha_1} {\lstep!}\,.
  \]
  More generally, the coefficient $a_j$ of the the term of degree $j$
  of $P_\lstep(\alpha,\bs,\PolyVar)$ is function
  \[
  a_j = a_j ( \alpha_1, \alpha_2+ s_1, \dots, \alpha_{\lstep -j +1}+
  s_{\lstep -j}), \quad \text{\ for\ } j=1,\dots,\lstep-1.
  \]
  linear in each variable.  For the constant term we have $a_0=s_k$.
\end{lemma}
Let $\alpha \in \R^J$ and let $X_{\alpha}\in \mathfrak g$ be the
vector field on $M$ given by formula~\pref{eq:X_alpha} and let
$B_{\alpha}^T$ be the Birkhoff averaging operator defined, for all
$f\in C^\infty(M)$ and all $(x,T)\in M \times \R$, by the formula
\begin{equation}
  \label{eq:Birkhoff_Av} 
  B^T_{X_\alpha}(x)(f)= \frac{1}{T} \int_0^T f \circ \phi_{X_\alpha}^t
  (x) \, \D{}t \,.
\end{equation}
The Weyl sums $\{W(P_\lstep, f; \return)\}_{\return\in \N}$ for any
smooth function $f \in C^\infty(\T^1)$ are special Birkhoff averages
$B^T_{\alpha}(x)(F)$, hence it possible to derive bounds on Weyl sums
from Sobolev bounds on on the Birkhoff averaging operators
$B^T_{\alpha}$ introduced above.

\begin{definition}
  \label{def:filiform}
  \begin{sloppypar} A quasi-Abelian Lie algebra on two generators
    $(X,Y_1)$ is a \emph{filiform} Lie algebra. A (generalised) Jordan
    basis $\{X,Y_1, \dots, Y_\lstep\}$ of a quasi-Abelian filiform Lie
    algebra will be called a \emph{(generalised) filiform} basis.  A
    quasi-Abelian Lie group on two generators is a filiform Lie
    group. By definition, the Lie algebra of a quasi-Abelian filiform
    Lie group is a quasi-Abelian filiform Lie algebra.  The quotient
    of a quasi-Abelian filiform group by a co-compact lattice is
    called a quasi-Abelian filiform nilmanifold.
  \end{sloppypar}
\end{definition}
Let $M=\Gamma\backslash G$ be a compact quasi-Abelian filiform
nilmanifold.  The torus ${\mathbb T}^\lstep_0$ denotes, as above, the
set $(\Gamma\cap A)\backslash A\subset M$, that is, the orbit of the
coset $\Gamma$ under the action of the Abelian subgroup $A$.  Let
$\{\xi, \tilde \eta_1, \dots, \tilde \eta_\lstep\}$ be the basis of
the quasi-Abelian filiform Lie algebra $\mathfrak g$ of $G$ given (for
the general quasi-Abelian case) in formulas~\eqref{eq:LieGen_1},
\eqref{eq:LieGen_2}.

Let $\bs=(s_1,\dots ,s_\lstep)\in \R^\lstep/\Z^{\lstep}$ denote the
point $\Gamma \, \exp (s_1 \tilde \eta_1 + \dots + s_\lstep \tilde
\eta_\lstep)\in {\mathbb T}^\lstep_0$.

\begin{lemma}
  \label{lem:reduction} For any $\alpha \in \R^\lstep$ there exist
  bounded injective linear operator
  \[F=F_{\alpha}:L^2(\T^1) \to L^2(M)\] such that the following holds.
  For any $r\geq 0$, the operator $F$ maps $H^r(\T^1)$ continuously
  into $H^r(M)$; moreover, for any $r>1/2$, there exists a constant
  $C_{r} >0$ such that, for any function $f\in H^r(\T^1)$, for all
  $(\bs, \return) \in {\mathbb T}^\lstep_0\times \N$, we have
  \begin{equation}
    \label{eq:reduction}
    \vert  \sum_{\ell=0}^\return  f \left ( P_{\lstep }(\alpha,\bs, \ell) \right)  - 
    \return \,B^\return_{X_\alpha}(\bs)(F(f)) \vert \,\leq \,  C_{r} \, 
    \| f \|_{H^r(\T^1)}\,.
  \end{equation}
\end{lemma}

\begin{proof} 
  For any $\varepsilon \in{} ]0, 1/2[$ the map
  \begin{equation}
    \label{eq:embed_0}
    (\bs ,t)\in {\mathbb T}^\lstep_0\times {]}-\varepsilon, \varepsilon{[}  
    \,\,\mapsto  \phi_{X_\alpha}^t(\bs)=    \bs \exp (tX_{\alpha})
  \end{equation}
  is an embedding of ${\mathbb T}^\lstep_0\times {]}-\varepsilon,
  \varepsilon{[}$ onto a tubular neighbourhood $\Cal U_\varepsilon$ of
  ${\mathbb T}^\lstep_0\subset M$.

  Let $pr:{\mathbb T}^\lstep_0 \to \T^1 $ be the projection on the
  circle $\T^1$ defined as follows
  \begin{equation*}
    pr (\bs)
    = s_{\lstep }  \,, \quad \text{ \rm for all }\, \bs \in {\mathbb T}^\lstep_0\,. 
  \end{equation*}
  Let $\chi \in C^\infty_0 ({]}-\varepsilon, \varepsilon{[})$ be any
  function such that $\int_\R \chi(\tau)\, d\tau =1$ and let $C_1=
  \|\chi\|_\infty$.  For any $f \in L^2(\T^1)$, let $F(f) \in L^2(M)$
  be the function defined on the open set $\Cal U_\varepsilon$ as
  \begin{equation}
    \label{eq:Fm}
    F (f)\left( \phi_{X_\alpha}^t(\bs)\right)  =  \chi (t)\, (f( pr (\bs))  \,,
    \qquad (\bs ,t)\in {\mathbb T}^\lstep_0\times {]}-\varepsilon, \varepsilon{[}.
  \end{equation}
  We then extend the function $F (f)$ as zero on $M\setminus \Cal
  U_\varepsilon$.\smallskip

  The function $F(f)$ is well-defined and square-integrable on $M$
  since $\chi$ is smooth and the map \pref{eq:embed_0} is an
  embedding.  Moreover, it follows from the definition~\eqref{eq:Fm}
  that $F (f)\in C^0(M)$ whenever $f\in C^0(\T^1)$ and $F (f)\in
  H^r(M)$ whenever $f\in H^r(\T^1)$, for any $r\geq 0$.

  \smallskip Let $f\in C^0(\T^1)$. We claim that, by the definition
  \eqref{eq:Fm} of the function $F (f)$, for all $(\bs,\return) \in
  {\mathbb T}^\lstep_0 \times \N$, we have
  \begin{equation}
    \label{eq:int_Fm}
    \int_{-\varepsilon}^{\return + \varepsilon}  
    F (f)\circ \phi_{X_\alpha}^t( \bs) dt  =  
    \sum_{\ell=0}^\return f\left(P_{\lstep }(\alpha,\bs, \ell)\right)\,.
  \end{equation}
  In fact, let $ \Phi^\ell_{\alpha}: {\mathbb T}^\lstep_0 \to {\mathbb
    T}^\lstep_0$ be the $\ell$-th return map of the flow
  $\{\phi_{X_\alpha}^t\}_{t\in \R}$. By Lemma~\ref{lem:return_maps}
  and by definition~\eqref{eq:Fm}, for all $(\bs,\ell) \in {\mathbb
    T}^\lstep_0 \times \N$,
  \begin{equation*}
    pr \circ  \Phi^\ell_{\alpha} (\bs) =    P_{\lstep }(\alpha,\bs, \ell) \,,
  \end{equation*}
  hence, for all $f\in C^0(\T^1)$,
  \begin{equation*}
    \begin{aligned}
      &\int_{\ell-\varepsilon}^{\ell+1-\varepsilon} F (f)\circ
      \phi_{X_\alpha}^t(\bs) \,\D{t} =
      \int_{\ell-\varepsilon}^{\ell+\varepsilon} F (f)\circ  \phi_{X_\alpha}^t(x) \, \D{t} \\
      &= \int_{-\varepsilon}^{\varepsilon} F (f)\circ
      \phi_{X_\alpha}^{\tau}\left( \Phi^\ell_{\alpha}(\bs)\right)
       \D \tau =\left( \int_{-\varepsilon}^{\varepsilon}
        \chi(\tau)\,\D\tau \right) f\left(P_{\lstep }
        (\alpha,\bs,\ell)\right)\,.
    \end{aligned}
  \end{equation*}
  The claim is therefore proved.
 
  It follows by formula \pref{eq:int_Fm} that
  \begin{equation*}
    \left\vert \int_0^{\return}  F(f) \circ \phi_{X_\alpha}^t(\bs) \,\D{t}  -   
      \sum_{\ell=0}^\return f\left(P_{\lstep }(\alpha,\bs, \ell)\right)
    \right\vert  \leq  2 \epsilon  \|F(f)\|_\infty \,.
  \end{equation*}
  By the Sobolev embedding theorem $H^r(\T^1)\subset C^0(\T^1)$, for
  any $r>1/2$, and there exists a constant $c_r>0$ such that
  $\|f\|_\infty \le c_r \| f\|_{H^r(\T^1)}$; since by definition
  $\|F(f)\|_\infty \le \|\chi\|_\infty \, \|f\|_\infty $, the
  inequality~\eqref{eq:reduction} follows and the argument is
  concluded.
\end{proof}

The problem of establishing bounds on Weyl sums is thus reduced to
that of bounds for the nilpotent averages~\eqref{eq:Birkhoff_Av} .

\section{A Sobolev trace theorem}
\label{sec:Sob_trace}

We prove below a Sobolev trace theorem for nilpotent orbits. According
to this theorem, the uniform norm of an orbital (ergodic) integral is
bounded in terms of the \emph{average width} of the orbit segment
times the \emph{transverse Sobolev norms} of the function, with
respect to a given basis of the Lie algebra.  The average width of an
orbit segment is a a positive number which measure of the weighted
frequency of close returns of the orbit segment close of itself. Its
definition is in fact very general and our theorems can be generalised
to the case of Lie groups with a codimension one ideal (Abelian or
not). Bounds on the average width (for rescaled bases) of orbits
segment of quasi-Abelian nilflows will be established in Section
\ref{sec:width}.

\smallskip Let $M=\Gamma\backslash G$ be a quasi-Abelian nilmanifold
and let $\omega$ be the associated volume form.  We shall consider
general ordered bases $\mathcal F:= (X, Y)$ of $\mathfrak g$ such that
$Y:=( Y_1,\dots,Y_\adim)$ is a basis of $\mathfrak a$.

\begin{definition}
  \label{def:adaptedbasis}
  An \emph{adapted basis} of the Lie algebra $\mathfrak g$ is an
  ordered basis $(X,Y):= (X,Y_1, \dots, Y_\adim)$ of $\mathfrak g$
  such that $X \not \in \mathfrak a$ and $Y:=(Y_1, \dots, Y_\adim)$ is
  a basis of the Abelian ideal $\mathfrak a\subset \mathfrak g$.
     
  A \emph{strongly adapted basis} $(X,Y):= (X,Y_1, \dots, Y_\adim)$ is
  an adapted basis such that the following holds:
  \begin{itemize}
  \item the system $(X, Y_1, \dots, Y_n)$ is a system of generators of
    $\mathfrak g$, hence its projection is a basis of the
    Abelianisation $\mathfrak g/[\mathfrak g, \mathfrak g]$ of the Lie
    algebra $\mathfrak g$;
  \item the system $(Y_{n+1} ,\dots Y_\adim)$ is a basis of the ideal
    $[\mathfrak g, \mathfrak
    g]$. 
  \end{itemize}
  An adapted basis $(X,Y_1, \dots, Y_\adim)$ is \emph{normalised} if
  \begin{equation*}
    \omega(X, Y_1, \dots , Y_\adim) = 1    \quad \text{\ \rm on } \, M.
  \end{equation*}
\end{definition}
Note that, according to the above definition, all (generalised) Jordan
bases (see Definition~\ref{def:genJbasis}) are strongly adapted and
normalised.


For any basis $Y=(Y_1, \dots ,Y_\adim)$ of $\mathfrak a$ and for all
$\bs:=(s_1, \dots , s_\adim) \in \R^\adim$, it will be convenient to
use the notations
\[
\bs \cdot Y := s_1Y_1+ \dots +s_\adim Y_\adim\,.
\]

\subsection{Width of nilpotent orbits}
\label{sec:Width_def}

Let $\mathcal F:=(X,Y)$ be any normalised adapted basis of the
quasi-Abelian nilpotent Lie algebra $\mathfrak g$. For any $x\in M$,
let $\phi_x: \R\times \R^\adim\to M$ be the local embedding defined by
\begin{equation}
  \label{eq:phi_map}
  \phi_x(t,\bs ) =  x \exp{(tX})\exp{(\bs\cdot Y)}.
\end{equation}
We omit the proofs of the following two elementary lemmata.
\begin{lemma}
  \label{lem:phi_der}
  For any $x\in M$ and any $f\in C^\infty(M)$ we have
  \begin{equation*}
    \begin{aligned}
      \frac{\partial \phi_x^*(f)}{\partial t}(t,\bs ) &= \phi^*_x(
      Xf)(t,\bs ) +
      \sum_j s_j\phi_x^*([X, Y_j] f)(t,\bs ) \,; \\
      \frac{\partial \phi^*_x(f)}{\partial s_j} &= \phi^*_x(Y_jf) \, ,
      \quad \text{\ \rm for all}\,\, j=1, \dots,\adim\,.
    \end{aligned}
  \end{equation*}
\end{lemma}

\begin{lemma}
  \label{lem:phi_vol}
  For any $x\in M$, we have
  \begin{equation*}
    \phi_x^\ast (\omega) =   dt \wedge ds_1 \wedge \dots \wedge ds_\adim \,.
  \end{equation*}
\end{lemma}
Let $\text{Leb}_\adim$ denote the $a$-dimensional Lebesgue measure on
$ \R^\adim$.

\begin{definition}
  \label{def:inner_width}
  For any open neighbourhood of the origin $O \subset \R^\adim$, let
  $\Cal R_O$ be the family of all $\adim$-dimensional symmetric (i.e.\
  centred at the origin) rectangles $R \subset [-1/2,1/2]^\adim$ such
  that $R \subset O$.  The \emph{inner width} of the open set $O
  \subset \R^\adim$ is the positive number
  \[
  w(O):= \sup \{ \text{Leb}_\adim(R) \mid R \in \Cal R_O\} \,.
  \]
  The \emph{width function} of a set $\Omega \subset \R \times
  \R^\adim$ containing the line $\R \times \{0\}$ is the function
  $w_\Omega: \R \to [0,1]$ defined as follows:
  \[
  w_\Omega (\tau) := w(\{ \bs \in \R^\adim \mid (\tau,\bs) \in \Omega
  \}) \,, \quad \text{ for all }\, \tau\in \R\,.
  \]
\end{definition}

\begin{definition}
  \label{def:av_width} Let $\mathcal F=(X,Y)$ be any normalised
  adapted basis.  For any $x\in M$ and $T>1$, we consider the family
  $\mathcal O_{x,T} $ of all open sets $\Omega \subset \R \times
  \R^\adim$ satisfying:
  \begin{itemize}
  \item $[0,T] \times \{0\} \subset \Omega \subset \R \times
    [-1/2,1/2]^\adim$;
  \item the map
    \[
    \phi_x : \Omega \to M
    \]
    defined by formula~\eqref{eq:phi_map} is injective.
  \end{itemize}
  The \emph{average width} of the orbit segment
  \[
  \gamma_X(x,T):= \{ x \exp(tX) \mid 0\leq t \leq T\}= \{ \phi_x(t,0)
  \mid 0\leq t \leq T\} ,\] relative to the normalised adapted basis
  $\mathcal F$, is the positive real number
  \begin{equation}
    \label{eq:av_width}
    w_{\mathcal F}(x,T):= \sup_{\Omega\in \mathcal O_{x,T}} \left(
      \frac {1} {T} \int_0^T \frac{ds}{w_\Omega(s)} \right)^{-1}. 
  \end{equation}
  The \emph{average width} of the nilmanifold $M$, relative to the
  normalised adapted basis $\mathcal F$, at a point $y\in M$ is the
  positive real number
  \begin{equation}
    \label{eq:av_width_point}
    w_{\mathcal F}(y):= \sup \{   w_{\mathcal F}(x, 1) \vert  y \in \gamma_X(x,1)\}   \,.
  \end{equation}
\end{definition}

\subsection{Sobolev a priori bounds}
\label{sec:Sob_bounds}

Let $\Delta_\bs$ denote the Euclidean Laplace operator on $\R^\adim$:
\[
\Delta_\bs:= - \sum_{j=1}^\adim \frac{\partial^2} {\partial s^2_j} \,.
\]
 
For any $\sigma \in \R$, let $W^\sigma(R)$ denote the standard Sobolev
space on a bounded open rectangle $R \subset \R^\adim$.  The following
lemma can be derived from the standard Sobolev embedding theorem for
the unit cube $[-1/2,1/2]^\adim$ using a rescaling argument.

\begin{lemma}
  \label{lem:Sob_embed_1}
  For any $\sigma > \adim/2$, there exists a positive constant
  $C:=C(\adim,\sigma)$ such that for any (symmetric) open rectangle $R
  \subset [-1/2,1/2]^\adim$ and for any function $f\in W^\sigma(R)$,
  we have
  \begin{equation*}
    \vert f(0) \vert \leq    \frac{C}{\vol(R)^{1/2}}  \Bigl( \int_{R} 
    \vert (I+\Delta_\bs)^{\sigma/2} f(\bs) \vert^2\, \D^\adim \bs \Bigr)^{1/2}\,.
  \end{equation*}
\end{lemma}
\begin{lemma}
  \label{lem:Sob_embed_2}
  Let $I \subset \R$ be a bounded interval and let $\Omega\subset
  \R\times \R^\adim$ be a Borel set containing the segment $I \times
  \{0\} \subset \R\times \R^\adim$.  For any $\sigma>\adim/2$, there
  exists a positive constant $C_{\adim,\sigma}$ such that, for all
  functions $F\in C^\infty(\Omega)$ and for all $t \in I$, we have
  \begin{equation}
    \label{eq:Sob_embed_2} 
    \left\vert \int_I F(t, 0)
      \,\D t \right\vert^2 \leq C^2_{\adim,\sigma} \left( \int_I
      \frac{\D \tau}{w_\Omega(\tau)} \right)  \int_{\Omega} \vert
    (I+\Delta_\bs)^{\frac{\sigma}{2}} F(\tau,\bs)  \vert^2 \,\D\tau\,\D^\adim \bs  \,;
  \end{equation}
  and
  \begin{equation}
    \label{eq:Sob_embed_3}
    \begin{aligned}
      | &F(t, 0)| \leq \frac{C_{\adim,\sigma}}{\vert I \vert} \left(
        \int_I \frac{\D \tau}{w_\Omega(\tau)} \right)^{1/2} \left[
        \left( \int_{\Omega} \vert (I+\Delta_\bs)^{\frac{\sigma}{2}}
          F(\tau,\bs) \vert^2 \,\D\tau\,\D^\adim \bs \,\right)^{1/2}\right.  \\
      &\left. \qquad \qquad \qquad \qquad \quad + \vert I\vert
        \left(\int_{\Omega} \left|
            (I+\Delta_\bs)^{\frac{\sigma}{2}} \partial_t F(\tau,\bs)
          \right|^2 \,\D\tau\,\D^\adim \bs \, \right)^{1/2}\right] \,.
    \end{aligned}
  \end{equation}
\end{lemma}
\begin{proof}
  For any $t\in \R$, let $\Omega_t:= \{\bs \in \R^\adim\mid (t,\bs)
  \in \Omega\}$. By the definition of the width function (see
  Definition~\ref{def:inner_width}) and by the standard Sobolev
  embedding theorem for bounded rectangles in $\R^\adim$, it follows
  that there exists a constant $C_{\adim,\sigma}>0$ such that, for any
  function $G\in C^\infty(\Omega)$ and for any $\tau \in I$,
  \begin{equation*}
    \vert G(\tau, 0) \vert \leq    \frac{C_{\adim,\sigma}}{w_\Omega(\tau)^{1/2}}  \Bigl( \int_{\Omega_\tau} \vert (I+\Delta_\bs)^{\frac{\sigma}{2}} G(\tau,\bs) \vert^2 \D^\adim \bs \Bigr)^{1/2}\,.
  \end{equation*} 
  Then by H\"older's inequality it follows that
  \begin{equation}
    \label{eq:int_stand_Sob}
    \left(\int_I  \vert G(\tau, 0) \vert \D\tau\right)^2 \leq    C^2_{\adim,\sigma} 
    \left(\int_I  \frac{\D\tau}{w_\Omega(\tau)} \right)
    \int_{\Omega} \vert (I+\Delta_\bs)^{\frac{\sigma}{2}} G(\tau,\bs) \vert^2 \D^\adim \bs \,\D\tau .
  \end{equation}
  Taking $G=F$ in the above formula yields the estimate in formula
  \eqref{eq:Sob_embed_2}.  To see~\eqref{eq:Sob_embed_3} observe that,
  by the fundamental theorem of calculus and the mean value theorem,
  for any $t\in \R$ there exists $t_0=t_0(t) \in I$ such that
  \[
  F(t,0) = \frac{1}{\vert I\vert} \int_I F(\tau,0) \,\D\tau \,+\,
  \int_{t_0}^t \partial_t F(\tau,0) \,\D\tau\,,
  \]
  which implies
  \[
  |F(t,0)| \le \int_I \left|\frac {F(\tau,0)}{\vert I\vert}\right|
  \,\D\tau + \int_I |\partial_t F(\tau,0)| \,\D\tau.
  \]
  The estimate in formula~\eqref{eq:Sob_embed_3} then follows by
  applying the bound in formula~\eqref{eq:int_stand_Sob} to the
  functions $G=F/|I|$ and $G=\partial_tF$.  The statement is proved.
\end{proof}

\begin{definition}
  \label{def:sobolev-trace:1}
  Given any normalised adapted basis $\mathcal F=(X,Y)$, let
  $\Delta_{Y}$ be the second order differential operator defined as
  follows:
  \begin{equation}
    \label{eq:Y_Lapl}
    \Delta_{Y} := - \sum_{j=1}^\adim Y_j^2  \,.
  \end{equation}
 
  For any $\sigma\geq 0$, let $\vert \cdot \vert_{\mathcal F, \sigma}$
  be the \emph{transverse Sobolev norm} defined as follows: for all
  functions $f \in C^\infty(M)$, let
  \begin{equation}
    \label{eq:transv_Sob_norms}
    \vert f \vert_{\mathcal F, \sigma} := 
    \Vert (I+ \Delta_{Y})^{\frac{\sigma}{2}} f \Vert_{L^2(M)}\,.
  \end{equation}
  The completion of $C^\infty(M)$ with respect to the norm $\vert
  \cdot \vert_{\mathcal F, \sigma}$ is denoted $W^\sigma(M,\mathcal
  F)$. Endowed with this norm $W^\sigma(M,\mathcal F)$ is a Hilbert
  space.
\end{definition}

The following version of the Sobolev embedding theorem holds.
 
\begin{theorem}
  \label{thm:SEnil}
  Let $\mathcal F= (X,Y)$ be any normalised adapted basis. For any
  $\sigma> \adim/2$, there exist positive constants
  $C_{\adim,\sigma}$, $C_\sigma$ such that, for all functions $u \in
  W^{\sigma+1}(M,\mathcal F)$ such that $Xu\in W^{\sigma}(M,\mathcal
  F)$ and for all $y\in M$, we have
  \begin{equation*}
    \vert u (y)\vert \leq \frac{C_{\adim,\sigma}}
    {w_{\mathcal F}(y)^{1/2}}  \Big\{  \vert
    u \vert_{ \mathcal F, \sigma}  + \vert Xu \vert_{\mathcal F, \sigma}
    + C_\sigma \sum_{j=1}^a\big\vert [X,Y_j] u \big\vert_{\mathcal F, \sigma} \Big\} \,.
  \end{equation*}
\end{theorem}
\begin{proof}
  Let $y\in M$ be a given point and let $x \in M$ be any point such
  that $y\in \gamma_X(x,1)$. Let $\Omega\subset
  \R\times[-1/2,1/2]^\adim$ be an open set containing
  $[0,1]\times\{0\}$ such that the map $\phi_x$ defined
  by~(\ref{eq:phi_map}) is injective on $\Omega$. Let $F(t,\bs) =
  u\circ \phi_x (t,\bs)$ for all $(t,\bs)\in \Omega$. By
  Lemma~\ref{lem:phi_der} we have
  \begin{equation}
    \begin{aligned} 
      \partial_t F(t,\bs) &= (X u + \sum_j s_j[X, Y_j]u)\circ \phi_x \,, \\
      \Delta_{\bs} F &= (\Delta_{Y} u )\circ \phi_x\,.
    \end{aligned}
  \end{equation}
  By Lemma~\ref{lem:phi_vol} and by the fact that the basis~$\mathcal
  F$ is normalised, it follows that the map $\phi_x$ maps the measure
  $ \D{t} \D^\adim\bs$ to the measure $\meas$; thus
  \begin{equation*}
    \begin{split}
      \left(\int_{\Omega} \vert (I+\Delta_\bs)^{\frac{\sigma}{2}}
        F(\tau,\bs) \vert ^2 \D\tau \D^\adim \bs\right)^{1/2} &=
      \left( \int_{\phi_x(\Omega)} \vert
        (I+\Delta_{Y})^{\frac{\sigma}{2}}u
        \vert ^2 \,\D\meas  \right)^{1/2} \\
      &\le \Vert (I+\Delta_{Y})^{\frac{\sigma}{2}} u \Vert_{L^2(M)}
      \,.
    \end{split}
  \end{equation*}
  For all $\tau \in I$, let $\Omega_\tau:=\{ \bs \in \R^\adim \vert
  (\tau, \bs)\in \Omega\}$.  By a direct computation for $\sigma\in
  \N$ and by the interpolation property of Sobolev norms in the
  general case, there exists a constant $C_\sigma>0$ such that, for
  all $\tau \in I$, for all $j\in \{1, \dots,\adim\}$ and for all $f
  \in W^\sigma (\Omega_\tau)$, we have
  $$
  \int_{\Omega_\tau} \vert (I+\Delta_\bs)^{\frac{\sigma}{2}} ( s_j f )
  \vert^2 \D^{\adim} \bs \leq C^2_\sigma\int_{\Omega_\tau} \vert
  (I+\Delta_\bs)^{\frac{\sigma}{2}} f \vert^2 \D^{\adim} \bs \,.
  $$
  By the above estimates it then follows that
  \begin{equation*}
    \begin{split}
      &\left( \int_{\Omega} \big\vert
        (I+\Delta_\bs)^{\frac{\sigma}{2}}
        \partial_t F(\tau,\bs) \big\vert^2 \,\D\tau \D^{\adim} \bs
      \right)^{1/2}
      \\
      &= \left( \int_{\Omega} \Big\vert
        (I+\Delta_\bs)^{\frac{\sigma}{2}} (X u +{ \textstyle \sum_j
        }s_j [X,Y_j]u)\circ \phi_x(\tau,\bs) \Big\vert^2 \,\D\tau
        \D^{\adim} \bs \right)^{1/2}
      \\
      &\le \Vert (I+\Delta_{Y})^{\frac{\sigma}{2}} Xu \Vert _{L^2(M)}+
      C_\sigma { \textstyle
        \sum_{j=1}^a}\big\Vert(I+\Delta_{Y})^{\frac{\sigma}{2}}
      [X,Y_j] u \big\Vert _{L^2(M)}.
    \end{split}
  \end{equation*}

  By applying the estimate~(\ref{eq:Sob_embed_3}) of
  Lemma~\ref{lem:Sob_embed_2} to the function $\phi^*_x u$, we obtain
  \[
  \begin{split}
    \vert u (y)\vert &\leq C_{\adim,\sigma} \left( \int_0^1 \frac{\D
        \tau}{w_\Omega(\tau)} \right)^{1/2} \Bigg[\Vert
    (I+\Delta_{Y})^{\frac{\sigma}{2}}
    u \Vert _{L^2(M)} \\
    &\qquad\qquad\qquad\qquad\qquad\qquad+ \Vert
    (I+\Delta_{Y})^{\frac{\sigma}{2}} Xu \Vert _{L^2(M)}
    \\&\qquad\qquad\quad\qquad\qquad\qquad\qquad + C_\sigma
    \sum_{j=1}^a\big\Vert (I+\Delta_{Y})^{\frac{\sigma}{2}} [X,Y_j]u
    \big\Vert _{L^2(M)}\Bigg].
  \end{split}
  \]
  Since the above inequality holds for all open sets $\Omega\subset
  \R\times[-1/2,1/2]^\adim$ containing $[0,1]\times\{0\}$, for which
  the restriction to $\Omega$ of the map $\phi_x$ is injective, we can
  replace the term $\left( \int_0^1 {\D\tau}/{ w_\Omega(\tau)}
  \right)^{1/2}$ in the above inequality by its infimum over all such
  sets, that is, by the lower bound $1/w_{\mathcal F}(x,1)^{1/2}$ (see
  formula~\eqref{eq:av_width} of Definition~\ref{def:av_width}).
  Finally, the statement follows by taking the infimum over all points
  $x\in M$ such that $y\in \gamma_X(x,1)$ (see
  formula~\eqref{eq:av_width_point} of Definition~\ref{def:av_width}).
\end{proof}

\subsection{Nilpotent averages}

For a vector field $X$ on $M$ and $x\in M$ the Birkhoff ergodic
average $B^T_X(x)$ is defined as follows: for all $f\in L^2(M)$,
\begin{equation*}
  B^T_X(x)(f):=  \frac{1}{T } \int_0^T f \circ \phi^t_X(x) \,\D t
  \,, \quad \text{\ \rm for all }\, T\in \R_+\,.
\end{equation*}
where $\phi_X^t$ is the flow generated by the vector field~$X$.  The
following Sobolev estimates for the linear functional $B^T_X(x)$ holds
on $W^\sigma(M,\mathcal F)$.
\begin{theorem}
  \label{thm:BA_apriori}
  Let $\mathcal F=(X,Y)$ be any normalised adapted basis. For any
  $\sigma> \adim/2$, there exists a positive constant
  $C_{\adim,\sigma}$ such that, for all functions $f \in
  W^\sigma(M,\mathcal F)$, for all $T\in{[} 1, +\infty)$ and all $x\in
  M$ we have
  \begin{equation*}
    \vert B^T_X (x)(f)\vert\leq   \frac{C_{\adim,\sigma}}{T^{1/2}w_{\mathcal F}(x,T)^{1/2}} 
    \vert  f \vert_{\mathcal F, \sigma}  \,. 
  \end{equation*}
\end{theorem}
\begin{proof} Let $\Omega\subset \R\times[-1/2,1/2]^\adim$ be any open
  set containing $[0,T]\times\{0\}$ for which the restriction to
  $\Omega$ of the map $\phi_x$ is injective (that is, let $\Omega\in
  \mathcal O_{x,T}$).  The estimate~\eqref{eq:Sob_embed_2} of
  Lemma~\ref{lem:Sob_embed_2} applied to the function $F = f\circ
  \phi_x $ yields
  \[
  \begin{split}
    \vert B^T_X (x)(f)\vert &= \left \vert\frac{1}{T }
      \int_0^T F(t,0) \,\D t \right\vert\\
    &\leq C_{\adim,\sigma} \, \frac{1}{T}\left( \int_I \frac{\D
        \tau}{w_\Omega(\tau)} \right)^{1/2} \left(\int_{\Omega} \vert
      (I+\Delta_\bs)^{\frac{\sigma}{2}} F(\tau,\bs) \vert^2
      \,\D\tau\,\D^\adim \bs  \right)^{1/2} \\
    &\leq C_{\adim,\sigma} \, \frac{1}{T}\left( \int_I \frac{\D
        \tau}{w_\Omega(\tau)} \right)^{1/2} \Vert
    (I+\Delta_Y)^{\frac{\sigma}{2}} f \Vert_{L^2(M)} \,.
  \end{split}
  \]
  As this inequality holds true for all $\Omega\in \mathcal O_{x,T}$,
  we can replace the term $\left( \int_0^T {\D\tau}/{ w_\Omega(\tau)}
  \right)^{1/2}$ in the above inequality by its infimum over all such
  sets, i.e.\ the lower bound $T^{1/2}/w_{\mathcal F}(x,T)^{1/2}$.
\end{proof}

\section{The cohomological equation}
\label{sec:CE}
  
In this section we prove a priori Sobolev estimates on the Green
operator for the cohomological equation $Xu=f$ of a quasi-Abelian
nilflow with generator $X\in \mathfrak g\setminus \mathfrak a$ and on
the distributional obstructions to existence of solutions (that is, on
invariant distributions). We then derive bounds on Sobolev norms the
Green operator and on the scaling of invariant distributions under a
group of dilations of the quasi-Abelian Lie algebra. We recall that
this analysis is motivated, on the one hand, by the well-known
elementary fact that ergodic integrals of coboundaries with bounded
transfer function (that is, of all derivatives of bounded functions
along the flow) are uniformly bounded, on the other hand, by the
heuristic principle that the growth of ergodic integrals is related to
the scaling of the invariant distributions under an appropriate
renormalisation group action.
    
\subsection{Irreducible unitary representations}
\subsubsection{Representation models} Kirillov's theory yields the
following complete classification of irreducible unitary
representations of filiform Lie groups (up to unitary equivalence).

Let $\mathfrak a^*$ be the space of $\R$-linear forms
on $\mathfrak a$.
For any $\Lambda\in\mathfrak a^*$ denote by $\exp \imath
\Lambda$ the character $\chi_\Lambda$ of $A$ defined by
$\chi_\Lambda(g) := \exp(\imath \Lambda(Y))$, for $g=\exp Y$ with
$Y\in\mathfrak a$.

The infinite dimensional irreducible repre\-sen\-ta\-tions of~$G$ are
unitarily equivalent to the representations $\Ind_A^G(\Lambda)$,
obtained by inducing from $A$ to $G$ a character $\chi=\exp \imath
\Lambda$ not vanishing on $[\mathfrak a,\mathfrak a]$. In addition,
two linear forms $\Lambda$ and $\Lambda'$ determine unitarily
equivalent representations if and only if they belong to the same
co-ajoint orbit. 

Restricting the function of $\Ind_A^G(\Lambda)$ to
the subgroup $\exp(tX)$, $t\in \R$, yields the following models for
the unitary representations $\Ind_A^G(\Lambda)$.

For $X\in \mathfrak g\setminus \mathfrak a$, $Y\in \mathfrak a$ and
$\Lambda\in \mathfrak a^*$, we denote by $ P_{\Lambda,Y}$ the
polynomial function $ x\to \Lambda(\Ad(e^{xX} )Y )$.  Let
$\pi^X_\Lambda$ be the unitary representation of the quasi-Abelian
$\lstep$-step nilpotent Lie group $G$ on the Hilbert space $L^2(\R)$
uniquely determined by the derived representation $D\pi^X_\Lambda$ of
the filiform Lie algebra $\mathfrak g$ given by the following
formulas:
\begin{equation}
  \label{eq:irreps_1}
  D\pi^X_\Lambda:\, \begin{cases}
    X\mapsto \frac{d}{dx}\\
    Y\mapsto \imath  P_{\Lambda,Y}(x)\quad \text{ for all } Y\in \mathfrak a\,.
  \end{cases}
\end{equation}
For each $\Lambda \in \mathfrak a^*$, not vanishing on $[\mathfrak g,
\mathfrak g]$, the unitary representation $\pi^X_\Lambda$ is
irreducible and, by Kirillov's theory, each irreducible unitary
representation of the quasi-Abelian $\lstep$-step nilpotent Lie group
$G$, which does not factor through a unitary representation of the
Abelian quotient $G/ [G,G]$, is unitarily equivalent to a
representation of the form $\pi^X_\Lambda$ described above.

Let
\[
\mathfrak a_0^* = \{ \Lambda\in \mathfrak a^* | \Lambda([\mathfrak g,
\mathfrak g])\neq 0\}.
\]

\begin{definition}
  \label{def:degreeY}
  For any $Y\in \mathfrak a$ we define its degree $d_Y\in \N$ with
  respect to the representation $\pi^X_\Lambda$ to be the degree of
  the polynomial $P_{\Lambda, Y}$. For any adapted basis $\mathcal F=
  (X, Y)$ of the Lie algebra $\mathfrak g$ let $(d_1, \dots,
  d_\adim)\in \N^\adim$ denote the degrees of the elements $(Y_1,
  \dots, Y_\adim)$.  The degree of the representation $\pi_\Lambda$ is
  then defined as the maximum of the degrees of the elements of any
  adapted basis.
\end{definition}
Observe that the condition $\Lambda \in \mathfrak a_0^*$ is equivalent
to $(d_1, \dots, d_\adim)\neq 0$.

For all $i=1,\dots, \adim$ and $j=1, \dots, d_i$, we let
\begin{equation}
  \label{eq:parameters}
  \Lambda^{(j)}_i(\mathcal F) = (\Lambda \circ \ad^j(X))(Y_i)\,.
\end{equation}
Then
the representation $\pi^X_\Lambda$ can be written as follows:
\begin{equation}
  \label{eq:irreps_2}
  D\pi^X_\Lambda:\, \begin{cases}
    X\mapsto \frac{d}{dx}\\
    Y_i \mapsto \imath \sum_{j=0}^{d_i}  \frac{\Lambda_i^{(j)}(\mathcal F) }{j !}  \, x^j \,. 
  \end{cases}
\end{equation}

For any linear form $\Lambda \in \mathfrak a_0^*$, let $\mathfrak
I_\Lambda \subset \mathfrak a$ be the subset defined as follows:
\begin{equation}
  \label{eq:Lambda_ideal0}
  \mathfrak I_\Lambda :=  \bigcap_{i=0}^{\lstep-1} \text{ker} (\Lambda \circ \ad^i(X)) \,.
\end{equation}
Since $\mathfrak g$ is quasi-Abelian, the set $\mathfrak I_\Lambda
\subset \mathfrak g$ is an ideal of the Lie algebra $\mathfrak g$. Let
$G_\Lambda \subset G$ the normal subgroup defined by exponentiation of
the ideal $\mathfrak I_\Lambda$.
It is clear from the above definition that the ideal $\mathfrak
I_\Lambda$, hence the subgroup $G_\Lambda$, depends only on the
co-adjoint orbit of the form $\Lambda\in \mathfrak a^* $.

\begin{lemma}
  \label{lemma:reduction}
  The irreducible unitary representation $\pi^X_\Lambda$ of the
  quasi-Abelian Lie group $G$ factors through a representation of the
  filiform Lie group $G/G_\Lambda$. In fact, for any adapted basis
  $\Cal F:= (X, Y)$ of the Lie algebra and for any element $Y_\ast \in
  \Cal F \cap \mathfrak a$ of maximal degree $d\geq 1$ for the
  representation, the system
$$(Y'_1, \dots, Y'_{d+1})= \left(Y_\ast, \ad(X)(Y_\ast), \dots, \ad^{d}(X)(Y_\ast)\right)$$
can be extended to an adapted basis $\Cal F'_\Lambda := (X, Y'_1,
\dots, Y'_\adim)$ with
$$
\pi^X_\Lambda (Y'_i) = 0 \,, \quad \text{for all } i\in \{d+2, \dots,
\adim\}\,.
$$
The basis $(X,Y')$ is strongly adapted or Jordan if the basis $(X,Y)$
is respectively strongly adapted or Jordan. In addition the
coefficients of the change of basis matrix $(C_{ij}^{Y,Y'}) \in
M_{\adim}(\R)$, that is, of the matrix such that
$$
(Y'_1, \dots, Y'_\adim) = (Y_1, \dots, Y_\adim) \cdot C^{Y,Y'}\,,
$$
can be estimated as follows.  Let
\[
C_{\Cal F,\Lambda}:= \Lambda(Y'_{d+1})^{-1} \max_{1\leq i \leq a,
  1\leq j \leq d_i} \vert \Lambda_i^{(j)}(\mathcal F)\vert\,. \] There
exists a constant $K_{\Cal F} \geq 1$ (equal to one if the basis $\Cal
F$ is Jordan) such that the following upper bound holds:
\begin{equation}
  \label{eq:cohomolequation:1}
  \vert C^{Y,Y'}_{i,j} \vert \leq  K_{\Cal F}  C_{\Cal F,\Lambda} \left(1+C_{\Cal F,\Lambda}\right)^{d}\,, \quad
  \text{ for all } i, j \in \{1, \dots, \adim\}\,.
\end{equation}
\end{lemma}

\begin{proof} Let $Y_\ast\in\{Y_1, \dots,Y_\adim\}$ be an element of
  maximal degree $d\geq 1$ and let
$$
(Y'_1, \dots, Y'_{d+1})= \left(Y_\ast, \ad(X)(Y_\ast), \dots,
  \ad^{d}(X)(Y_\ast)\right)\,.
$$ 
By construction, it follows that $\Lambda(Y'_{i})\neq 0$ for all
$i=1,\dots, d+1$. Since the representation $\pi^X_\Lambda$ has degree
$d$, the ideal $\mathfrak I_\Lambda \subset \mathfrak g$ defined in
formula~\eqref{eq:Lambda_ideal0} has at most codimension $d+1$. Hence
$$
\mathfrak a = \oplus_{i=1}^{d+1} \R Y'_i \bigoplus \mathfrak
I_\Lambda\,.
$$
The representation $\pi^X_\Lambda$ factorises through a representation
of the filiform Lie group $G_\Lambda$ of Lie algebra $\mathfrak
g/\mathfrak I_\Lambda$. In fact, by construction and by the definition
of induced representation
$$
\pi^X_\Lambda (Y) =0 \,, \quad \text{ for all } Y\in \mathfrak
I_\Lambda\,.
$$
For every $i\in \{1, \dots, \adim\}$, let $d_i\in \N$ denote the
degree of the vector field $Y_i \in Y$.  Let us consider the the
following system of linear equations :
\begin{equation}
  \label{eq:Lambda_ideal1}
  \Lambda \left( \ad^\ell(X) (Y_i) \right)=  \sum_{j=\ell}^{d_i}   c^{(i)}_j
  \Lambda(Y'_{d-j+\ell+1}) \,, \quad \forall\,\ell =0, \dots, d_i, \, \forall i
  =d+2,\dots \adim.
\end{equation}
We claim that this system has a unique solution $(c^{(i)}_j)$, ($i
=d+2,\dots \adim$, $j=0,\dots,d_i$) which satisfies the upper bounds
\begin{equation}
  \label{eq:cohomolequation:2}
  \vert c^{(i)}_{j} \vert \leq C_{\Cal F, \Lambda}(1+C_{\Cal F, \Lambda})^{d_i-j}.
\end{equation}
The proof of this claim proceeds by induction.  For $\ell =d_i$ we
have
$$
\Lambda \left( \ad^{d_i}(X) (Y_i) \right) = c^{(i)}_{d_i}
\Lambda(Y'_{d+1})\,.
$$
By the above equation, since $ \Lambda(Y'_{d+1}) \not=0$, the
coefficients $c^{(i)}_{d_i}\in \R$ are uniquely defined for all $i
=d+2,\dots, \adim$ and the following upper bound holds by definition:
$$
\vert c^{(i)}_{d_i} \vert \leq \vert \Lambda(Y'_{d+1}) \vert^{-1}
\vert \Lambda( \ad^{d_i}(X) (Y_i) )\vert \leq C_{\Cal F, \Lambda}.
$$
Let us assume the induction hypothesis that for $d_i \geq j > \ell$
the coefficients $c^{(i)}_{j}\in \R$ are uniquely defined and satisfy
the upper bounds in formula~\eqref{eq:cohomolequation:2}. The
coefficients $c^{(i)}_{\ell}\in \R$ can then be found for all $i
=d+2,\dots \adim$ by formula~\eqref{eq:Lambda_ideal1}, which also
implies that the following estimates hold:
 $$
 \vert c^{(i)}_{\ell} \vert \leq \vert \Lambda(Y'_{d+1}) \vert^{-1}
 \bigl( \vert \Lambda( \ad^\ell(X) (Y_i)) + \sum_{j=\ell+1}^{d_i}
 \vert c^{(i)}_j\vert\cdot \vert \Lambda(Y'_{d-j+\ell+1})
 \vert\bigr)\,.
$$
Since by definition we have the bound
\[
\vert \Lambda(Y'_{j}) \vert \le C_{\Cal F, \Lambda}\cdot
\Lambda(Y'_{d+1}), \quad \forall j=1,\dots,d+1,
\]
by the induction hypothesis we conclude that
$$
\vert c^{(i)}_{\ell} \vert \leq C_{\Cal F, \Lambda} \Big(1+
\sum_{j=\ell+1}^{d_i} C_{\Cal F, \Lambda}(1+C_{\Cal F,
  \Lambda})^{d_i-j}\Big)= C_{\Cal F, \Lambda} (1+C_{\Cal F,
  \Lambda})^{d_i-\ell}.
$$
We have therefore proved that the system in
formula~\eqref{eq:Lambda_ideal1} has a unique solution which satisfies
the estimates in formula~\eqref{eq:cohomolequation:2} hold.

We now complete the system $\{Y'_{1}, \dots, Y'_{d+1}\}$ to obtain a
basis of the Abelian Lie sub-algebra~$\mathfrak a$.

Up to reordering the elements of the basis $Y\subset \mathfrak a$, it
is not restrictive to assume that $\{Y'_1, \dots, Y'_{d+1}, Y_{d+2},
\dots, Y_\adim\}$ is a basis of $\mathfrak a$.  For $i=d+2,
\dots,\adim$, let
\begin{equation}
  \label{eq:Lambda_ideal2}
  Y'_i := Y_i - \sum_{j=0}^{d_i}   c^{(i)}_j  Y'_{d-j+1}  .
\end{equation}
It follows from formula~\eqref{eq:Lambda_ideal2} that $\{Y'_1, \dots,
Y'_{d+1}, Y'_{d+2}, \dots, Y'_\adim\}$ is a basis of $\mathfrak a$ and
from formulas~\eqref{eq:Lambda_ideal1} and~\eqref{eq:Lambda_ideal2}
that $\{Y'_{d+2}, \dots,Y'_{\adim}\} \subset \mathfrak I_\Lambda$.  If
the basis $(X, Y)$ is adapted, so is the the basis $(X,Y'):=(X, Y'_1,
\dots,Y'_\adim)$ since the systems $\{Y'_1, \dots, Y'_\adim\}$ and
$\{Y_1, \dots, Y_\adim\}$ span the same subspace.  If $(X,Y)$ is
strongly adapted, then up to reordering the elements of the basis
$Y'\subset \mathfrak a$ the basis $(X, Y')$ is also strongly adapted.
In fact, by definition $Y'_2, \dots, Y'_{d+1}\in [\mathfrak g,
\mathfrak g]$, hence the element $Y'_i \in [\mathfrak g, \mathfrak g]$
whenever $Y_i\in [\mathfrak g, \mathfrak g]$, for all $i\in \{d+2,
\dots,\adim\}$. It follows from formulas~\eqref{eq:Lambda_ideal1}
and~\eqref{eq:Lambda_ideal2} that if the basis $\mathcal F= (X,Y)$ is
Jordan, then, up to reordering the elements of the basis $Y'\subset
\mathfrak a$, the basis $(X,Y')$ is Jordan as well.

The estimates in formula~\eqref{eq:cohomolequation:1} can be derived
from the upper bounds in formula~\eqref{eq:cohomolequation:2} by
formula~\eqref{eq:Lambda_ideal2}. The constant $K_\Cal F \geq 1$ is
defined as follows. Let $(a_{ij})$ denote the matrix of the
coordinates of the vectors $Y'_1, \dots, Y'_{d+1}$ with respect to the
basis $Y:=\{Y_1, \dots, Y_{\adim}\} \subset \mathfrak a$ and let
$\chi$ denote the indicator function of the set $\{d+2, \dots,
\adim\}$. Let us define the constant
$$
K_\Cal F := \max_{1\leq j \leq \adim} (\chi_j+ \sum_{i=1}^{d+1} \vert
a_{ij} \vert) \,.
$$
The estimates in formula~\eqref{eq:cohomolequation:1} then follow from
the estimates in formula ~\eqref{eq:cohomolequation:2} by the above
definition and by formula~\eqref{eq:Lambda_ideal2}.  In the special
case that the basis $\Cal F$ is Jordan, the above constant $K_\Cal
F=1$ since by construction the set $\{Y'_1, \dots, Y'_{d+1}\}$ is a
subset of $\{Y_1, \dots, Y_\adim\}$, disjoint from the subset
$\{Y_{d+2}, \dots, Y_\adim\}$.

\end{proof}

Motivated by the above lemma we introduce the following
\begin{definition}
  \label{def:gen_fil}
  A \emph{generalised filiform basis} for an induced irreducible
  unitary representation $\pi^X_\Lambda$ of degree $d\geq 1$ of a
  quasi-Abelian nilpotent Lie group is an adapted basis $(X, Y_1,
  \dots, Y_\adim)$ such that $(X, Y_1, \dots, Y_{d+1})$ is a filiform
  basis for the generated filiform sub-algebra and $\{Y_{d+2}, \dots,
  Y_\adim\}$ is a basis of $\text{ker}(\pi^X_\Lambda)$, that is,
  \begin{equation*} {[}X,Y_i{]} =Y_{i+1} \,, \text{ for } 1\leq i \leq
    d\,, \quad \text{ and } \quad \pi^X_\Lambda(Y_i)=0 \,, \text{ for
    } d+2 \leq i \leq \adim \,.
  \end{equation*}
\end{definition}
According to Lemma~\ref{lemma:reduction}, generalised filiform bases
exist for all irreducible unitary representations $\pi^X_\Lambda$ of
non-zero degree of quasi-Abelian Lie groups and their norm can be
bounded in terms of the linear functional $\Lambda\in \mathfrak
a_0^*$.


\subsubsection{Sobolev norms} We denote by $ C^\infty(\pi^X_\Lambda)$
the space of $C^\infty$ vectors of the irreducible unitary
representation $\pi^X_\Lambda$ defined by the formulas
\eqref{eq:irreps_1} and~\eqref{eq:irreps_2}.

The transverse Sobolev norms introduced in formula
\eqref{eq:transv_Sob_norms} can be written in representation as
follows. For the representation $\pi^X_\Lambda$, the transverse
Laplace operator~$\Delta_Y$, introduced in formula~\eqref{eq:Y_Lapl},
is represented as the operator of multiplication by the non-negative
polynomial function
\begin{equation}
  \label{eq:Psquare}
  \Delta_{\Lambda,\mathcal F}( x):=  
  \sum_{i=1}^{\adim} \vert P_{\Lambda,Y_i}(x)\vert ^2=
  \sum_{i=1}^{\adim} \left\vert 
    \sum_{j=0}^{d_i}  \frac{\Lambda_i^{(j)}(\mathcal F) }{j !}  \,
    x^j
  \right\vert ^2 .
\end{equation}
Thus, the transverse Sobolev norms can be written as follows: for
every $\sigma \geq 0$ and for every $f \in C^{\infty}(\pi^X_\Lambda)$,
\begin{equation*}
  \vert f \vert_{\mathcal F,\sigma}:= \left (\int_\R  [1 + \Delta_{\Lambda,\mathcal F}( x)]^{\frac{\sigma}{2}} \vert
    f(x)\vert^2 \, \D{x} \right)^{1/2}\,.
\end{equation*}

\subsection{A priori estimates} The unique distributional obstruction
to the existence of solutions of the cohomological equation
\begin{equation}
  \label{eq:Xcohomeq}
  Xu=f
\end{equation}
in a given irreducible unitary representation $\pi^X_\Lambda$ is the
normalised $X$-invariant distribution $\Cal D^X_\Lambda \in {\Cal
  D'}(\pi^X_\Lambda)$ which can be written as
\begin{equation}
  \label{eq:invdist}
  \Cal D^X_\Lambda (f) := \int_{\R}  f(x) \, \D{x}  \,, \quad \text{ \rm for all }\, f\in
  C^\infty(\pi^X_\Lambda)\,.
\end{equation}
The formal Green operator $G^X_\Lambda$ for the the cohomological
equation \pref{eq:Xcohomeq} is given by the formula
\begin{equation}
  G^X_\Lambda (f)(x) := \int_{-\infty}^x f(y) \, \D{y}  \,, 
  \quad \text{ \rm for all }\, f\in C^\infty(\pi^X_\Lambda)\,.
\end{equation}
It is not difficult to prove that the Green operator is well-defined
on the kernel $\Cal K^\infty(\pi^X_\Lambda)$ of the distribution $\Cal
D^X_\Lambda$ on $C^\infty(\pi^X_\Lambda)$: for all $f\in \Cal
K^\infty(\pi^X_\Lambda)$, the function $ G^X_\Lambda(f) \in
C^\infty(\pi_\Lambda)$ and the following identities hold:
\begin{equation}
  \label{eq:Green}
  G^X_\Lambda (f)(x) := \int_{-\infty}^x  f(y) \, \D{y} = - 
  \int_x^{+\infty}  f(y) \, \D{y}\,.
\end{equation}
We prove below bounds on the transverse Sobolev norms $\Vert
G^X_\Lambda(f)\Vert_{\tau, \mathcal F}$ for all functions $f\in \Cal
K^\infty(\pi^X_\Lambda)$.
    
For any $\sigma$, $\tau \in \R_+$ let
\begin{equation}
  \label{eq:IJint}
  \begin{aligned}
    I_\sigma(\Lambda, \mathcal F) &:=
    \left( \int_\R   \frac{ \D{x} } {[1+\Delta_{\Lambda, \mathcal F}(x)]^\sigma} \right)^{1/2} \,; \\
    J^\tau_\sigma (\Lambda, \mathcal F) &:= \left( \iint_{\vert y
        \vert \geq \vert x \vert } \frac{[1+\Delta_{\Lambda, \mathcal
          F}(x)]^\tau}{[1+\Delta_{\Lambda, \mathcal F}( y)]^\sigma}
      \,\D{x} \D{y} \right)^{1/2} \,.
  \end{aligned}
\end{equation}

\begin{lemma}
  \label{lemma:Distnorm}
  Let $\Cal D^X_\Lambda \in \Cal D'(\pi^X_\Lambda)$ be the
  distribution defined in formula \pref{eq:invdist}.  For any
  $\sigma\in \R_+$, the following holds:
  \begin{equation}
    \label{eq:Distnorm}
    \vert \Cal D^X_\Lambda \vert_{\mathcal F,-\sigma}   :=   \sup_{f\not=0} 
    \frac{ \vert \Cal D^X_\Lambda (f)\vert}{ \,\,\vert f \vert_{\mathcal F,\sigma}} =  I_\sigma(\Lambda, \mathcal F)\,.
  \end{equation}
\end{lemma}
\begin{proof}
  It follows from the definitions by H\"older inequality. In fact,
  \begin{equation}
    \Cal D^X_\Lambda (f) = \<(1+ \Delta_{\Lambda, \mathcal F})^{-\frac{\sigma}{2}},
    (1+\Delta_{\Lambda, \mathcal F})^{\frac{\sigma}{2}} f \>_{L^2(\R)}\,.
  \end{equation}
  Since $ \vert f \vert_{\mathcal F,\sigma} = \vert
  (1+\Delta_{\Lambda, \mathcal F})^{\frac{\sigma}{2}} f \vert_0\,$, it
  follows that
 $$
 \sup_{f\not=0} \frac{ \vert \Cal D^X_\Lambda(f)\vert}{ \,\,\vert f
   \vert_{\mathcal F,\sigma}} = \vert (1+ \Delta_{\Lambda, \mathcal
   F})^{-\frac{\sigma}{2}} \vert_{L^2(\R)} = I_\sigma(\Lambda,
 \mathcal F)\,.
 $$
 The identity \pref{eq:Distnorm} is thus proved.
\end{proof}
  
\begin{lemma}
  \label{lemma:Greenbound}
  For any $\sigma \geq \tau$ and for all $f\in \Cal
  K^\infty(\pi^X_\Lambda)$,
  \begin{equation}
    \vert G^X_\Lambda (f)\vert_{\mathcal F, \tau } \leq J^\tau_\sigma(\Lambda, \mathcal F)] \,
    \vert f \vert_{\mathcal F,\sigma}\,.
  \end{equation}
\end{lemma}
\begin{proof}
  It follows by H\"older inequality from formula \pref{eq:Green} for
  the Green operator, in fact, for all $x\in \R$, by H\"older
  inequality we have
$$
\vert G^X_\Lambda(f)(x)\vert^2 \leq \left( \int_{\vert y\vert \geq
    \vert x \vert} \frac{\D{y}} {(1+\Delta_{\Lambda, \mathcal
      F}(y))^\sigma} \right) \vert f \vert^2_{\mathcal F, \sigma} \,.
$$
Another application of H\"older inequality yields the result.
\end{proof}

We have thus reduced Sobolev bounds on the Green operator for the
cohomological equation and on the ergodic averages operator (in each
irreducible representation) to bounds on the integrals defined in
formula~\eqref{eq:IJint}.

Let $\mathcal F=(X,Y_1, \dots, Y_\adim)$ be any adapted basis. Let
$(d_1, \dots, d_\adim)\in \N^\adim$ be the degrees of the elements
$(Y_1, \dots, Y_\adim)$, respectively; for all $\Lambda \in \mathfrak
a^*_0$, let $\Lambda^{(j)}_i (\mathcal F)= \Lambda(\ad(X)^j Y_i)$ be
the coefficients appearing in formula~\eqref{eq:irreps_2} and set
\begin{equation}
  \label{eq:Lambda_Fnorm_nohat}
  |\Lambda(\mathcal F)|:=  
  \sup_{\{(i,j) \,: \,1\le i\le \adim,\, 0\le j\le d_i\}} \left|  \frac 1 {j!} \Lambda_i^{(j)}(\mathcal F)\right|\,.
\end{equation}

We introduce on $\mathfrak a^*_0$ the following weight. For all
$\Lambda \in \mathfrak a^*_0$, let
\begin{equation}
  \label{eq:rep_weight}
  w_{\mathcal F}(\Lambda) := \min_{\{i \,:\, d_i\not=0\}} 
  \left\vert \frac{\Lambda^{(d_i)}_i (\mathcal F)}{d_i!}\right\vert^{-\frac{1}{d_i}}
\end{equation}
We will prove below estimates for the integrals $I_\sigma (\Lambda,
\mathcal F)$ and $J^\tau_\sigma (\Lambda, \mathcal F)$ of
formula~\pref{eq:IJint} in terms of the above weight.

For all $i=1, \dots, \adim$ and $j=1, \dots , d_i$ we define the
rescaled coefficients
\begin{equation}
  \label{eq:Lambda_iFnorm}
  \hat \Lambda_i^{(j)}(\mathcal F):=   \Lambda^{(j)}_i (\mathcal F)
  \, \big(w_{\mathcal F}(\Lambda)\big)^j,
\end{equation}
and set, in analogy with~\eqref{eq:Lambda_Fnorm_nohat},
\begin{equation}
  \label{eq:Lambda_Fnorm}
  |\hat \Lambda(\mathcal F)|:=
  \sup _{\{(i,j) \,: \,1\le i\le \adim,\,
    0\le j\le d_i\}} \left|  \frac 1 {j!} \hat \Lambda_i^{(j)}(\mathcal F)\right|\,.
\end{equation}

\begin{lemma}
  \label{lemma:IJbounds}
  For all $\sigma>1/2$, there exists a constant $C_{\lstep,\sigma}>0$
  such that, for all $\Lambda \in \mathfrak a^*_0$, the following
  bounds hold:
  \begin{equation}
    \label{eq:Isbounds}
    \frac{ C^{-1}_{\lstep,\sigma} }
    {(1+\vert \hat \Lambda(\mathcal F)\vert)^{\sigma}}
    \, \leq \,     
    \frac{ I_\sigma (\Lambda, \mathcal F) \,\, }{w^{1/2}_{\mathcal F}(\Lambda)} \,\leq \,  
    C_{\lstep,\sigma} (1+\vert \hat \Lambda(\mathcal F)\vert) \,. 
  \end{equation}
  For all $\sigma>\tau(\lstep-1) +1$, there exists a constant
  $C_{\lstep,\sigma,\tau}>0$ such that, for all $\Lambda \in \mathfrak
  a^*_0$, the following bounds hold:
  \begin{equation}
    \label{eq:Jstbounds}
    \frac{ C^{-1}_{\lstep,\sigma,\tau}}
    {(1+\vert \hat \Lambda(\mathcal F)\vert)^{\sigma}}
    \, \leq \,     
    \frac{ J^\tau_\sigma (\Lambda, \mathcal F) \,\, }{w_{\mathcal F}(\Lambda)} \,\leq \,  
    C_{\lstep,\sigma,\tau} (1+\vert \hat \Lambda(\mathcal F)\vert)^{\tau \lstep+1} \,. 
  \end{equation}
\end{lemma}
\begin{proof} By change of variables, for any $w>0$,
  \begin{equation}
    \label{eq:varchange}
    \begin{aligned}
      I_\sigma(\Lambda, \mathcal F) &= w^{1/2}\, \left( \int_\R
        \frac{ \D{x} } {[1+\Delta_{\Lambda, \mathcal F} (w x)]^\sigma} \right)^{1/2} \,; \\
      J^\tau_\sigma (\Lambda, \mathcal F) &= w\, \left( \int
        \int_{\vert y \vert \geq \vert x\vert }
        \frac{[1+\Delta_{\Lambda, \mathcal F} (w x)]^\tau} {[1+
          \Delta_{\Lambda, \mathcal F} (w y)]^\sigma} \,\D{x} \D{y}
      \right)^{1/2}\,.
    \end{aligned}
  \end{equation}
  Let $w:=w_{\mathcal F}(\Lambda) >0$.  By definitions
  \eqref{eq:rep_weight}, \eqref{eq:Lambda_iFnorm} and
  \eqref{eq:Psquare}, for all $i=1, \dots, \adim$, the coefficients of
  the polynomial map $P_{\Lambda, Y_i}(wx)$ are the numbers $\hat
  \Lambda_i^{(j)}(\mathcal F)/j!$. Thus all these coefficients are
  bounded by $|\hat \Lambda(\mathcal F)|$ and there exists $i_0\in
  \{1, \dots, \adim\}$ such that the polynomial $P_{\Lambda, Y_{i_0}}
  (w x)$ is monic.  The following inequalities therefore hold: for all
  $x\in \R$,
  \begin{equation}
    \label{eq:Delta_bounds}
    1+ P^2_{\Lambda, Y_{i_0}} (w x)  \leq  1+ \Delta_{\Lambda, \mathcal F} (w x)  \leq  
    (1+\vert \hat \Lambda(\mathcal F)\vert)^2 (1+ x^{2(\lstep-1)})\,.
  \end{equation}

  Let $P(x)$ be any non-constant monic polynomial of degree $d\geq 1$
  and let $\|P\|$ denote the maximum modulus of its coefficients. We
  claim that, if $d \sigma >1/2$, there exists a constant
  $C_{d,\sigma}>0$ such that
  \begin{equation}
    \label{eq:main_int_est_1}
    \int_{\R}   \frac{ \D{x}} { (1+ P^2(x))^{\sigma} }    \leq  C_{d,\sigma}  (1+ \Vert P \Vert) \,,
  \end{equation}
  and, if $d \sigma > (\lstep-1)\tau +1/2$, there exists a constant
  $C_{\lstep,d,\sigma,\tau}>0$ such that
  \begin{equation}
    \label{eq:main_int_est_2}
    \int\int_{\vert y\vert \geq \vert x\vert }   \frac{(1+ x^{2(\lstep-1)})^\tau} { (1+ P^2(y))^{\sigma} } \,\D{x}\D{y}
    \leq  C_{\lstep,d,\sigma,\tau} (1+ \Vert P \Vert)^{2+2\tau(\lstep-1)}\,.
  \end{equation}
  In fact, since $P$ is monic, there exists $s \in [1, (1+ \Vert P
  \Vert)]$ such that the polynomial $P_s(x) := s^{-d} P(sx)$ is monic
  and has all coefficients in he unit ball. It follows that, if
  $d\sigma >1/2$, there exists a constant $C_{d,\sigma}>0$ such that
$$
\int_{\R} \frac{ \D{x}} { (1+ P^2(x))^{\sigma} } = \int_{\R} \frac{
  s\, \D{x}} { (1+ P^2(s x))^{\sigma} } \leq \int_{\R} \frac{ s\,
  \D{x}} { (1+ P_s^2(x))^{\sigma} } \leq C_{d,\sigma} s \,,
$$
hence the bound in formula~\eqref{eq:main_int_est_1} is
proved. Similarly, if $\sigma>\tau(\lstep-1) +1$, there exists a
constant $C_{\lstep,d,\sigma,\tau}>0$ such that
$$
\begin{aligned}
  \int\int_{\vert y\vert \geq \vert x\vert } & \frac{(1+
    x^{2(\lstep-1)})^\tau} { (1+ P^2(y))^{\sigma} } \,\D{x}\D{y} =
  \int\int_{\vert y\vert \geq \vert x\vert } s^2 \frac{(1+ (s
    x)^{2(\lstep-1)})^\tau} { (1+ P^2(sy))^{\sigma} } \,\D{x}\D{y} \\
  &\leq \int\int_{\vert y\vert \geq \vert x\vert } s^2 \frac{(1+
    (sx)^{2(\lstep-1)})^\tau} { (1+ P_s^2(sy))^{\sigma} } \,\D{x}\D{y}
  \leq C_{\lstep,d,\sigma,\tau} s^{2+2\tau(\lstep-1)} \,,
\end{aligned}
$$
hence the bound in formula~\eqref{eq:main_int_est_2} is proved as
well.
 
Finally, applying the the bounds in \eqref{eq:main_int_est_1} and
\eqref{eq:main_int_est_2} to the polynomial $P_{\Lambda,Y_i}(wx)$ and
taking into account the formulas \eqref{eq:varchange} and the
estimates~\eqref{eq:Delta_bounds} we obtain the upper
bounds~\eqref{eq:Isbounds} and \eqref{eq:Jstbounds}.

The lower bounds are an immediate consequence of the upper bound in
formula~\eqref{eq:Delta_bounds}, hence the argument is complete.
\end{proof}
 
\subsection{The renormalisation group}
\begin{definition} The {\em deformation space} of a $\lstep$-step
  nilpotent quasi-Abelian nilmanifold $M=\Gamma\backslash G$ is the
  space $T(M)$ of all adapted bases of the Lie algebra $\mathfrak g$
  of the group $G$.
\end{definition}
 
Let $\Cal A< SL(\adim+1, \R)$ be the subgroup of all matrices $A$ of
the following form.  For any $\alpha \in \R\setminus \{0\}$, for any
vector $\beta\in \R^{\adim}$ and for any matrix $B \in GL(\adim, \R)$,
let
\begin{equation}
  A:= \begin{pmatrix}
    \alpha & \beta \\ 
    0        &  B 
  \end{pmatrix} \,.
\end{equation}
The group $\Cal A$ acts on the deformation space $T(M)$. In fact, let
$\Cal F=(X,Y)$ be any adapted basis of the Lie algebra $\mathfrak
g$. For any $A\in \Cal A$ the transformed basis is defined as
\begin{equation}
  \begin{pmatrix}   X_A \\ Y_{1,A} \\ \dots \\  Y_{\adim,A} \end{pmatrix}
  = A \,  \begin{pmatrix} X \\ Y_1 \\  \dots \\ Y_\adim
  \end{pmatrix}\,.
\end{equation}
The renormalisation dynamics will be defined as the action of the
diagonal subgroup of the Lie group $\text{SL}(\adim+1, \R)$ on the
deformation space.
 
Let $\rho:=(\rho_1, \dots, \rho_\adim)\in (\R^+)^\adim$ be any vector
such that
\begin{equation*}
  \sum_{j=1}^\adim \rho_j =1\,,
\end{equation*}
there exists a one-parameter subgroup $\{A^\rho_t\}$ of the diagonal
subgroup of $SL(\adim+1, \R)$ defined as follows:
\begin{equation}
  \label{eq:renorm}
  A^\rho_t (X, \dots, Y_i, \dots) =  (e^t X , \dots , e^{-\rho_it} Y_i ,\dots)\,.
\end{equation}
Note that the renormalisation group preserves the set of all
generalised Jordan basis. However, the group $\Cal A$ is not a group
of automorphisms of the Lie algebra. Consequently, the dynamics
induced by the renormalisation group on the deformation space is
trivial (it has no recurrent orbits).
   
\subsubsection{Estimates for rescaled bases }
In this section we prove Sobolev estimates for the invariant
distribution and for the Green operator in any irreducible unitary
representation with respect to rescaled bases.  Let $\Lambda \in
\mathfrak a^*_0$, let $\mathcal F=(X,Y)$ be any adapted basis and let
$\pi^X_\Lambda$ be the induced representation. For all $t \in \R$, let
\begin{equation}
  \label{eq:rescaledbasis}
  \mathcal
  F(t) = (X(t), Y(t))= A^\rho_t (X,Y)
\end{equation}
a rescaled adapted basis and let $U_t: L^2(\R) \to L^2(\R)$ be the
unitary operator defined as follows: for all $f\in L^2(\R)$,
 $$
 (U_t f)(x) = e^{-\frac{t}{2}} f( e^{-t} x) \,, \quad \text{ \rm for
   all }x \in \R \,.
 $$
 \begin{lemma}
   \label{lemma:U_t}
   For all $t\in \R$ the operator $U_t : L^2(\R) \to L^2(\R)$ in
   intertwines the representation $\pi^{X(t)}_\Lambda$ and
   $\pi^X_\Lambda$:
 $$
 U_t^{-1} \pi^X_\Lambda U_t = \pi^{X(t)}_{\Lambda}\,.
 $$
\end{lemma}
\begin{proof} By simple computations it follows that, for all $f\in
  L^2(\R)$,
 $$
 [ U_t^{-1} \pi^X_\Lambda (X(t)) U_t]( f ) = \frac{df}{dx} =
 \pi^{X(t)}_{\Lambda} (X(t))(f) \,,
 $$
 and, for all $i=1, \dots, \adim$ and for all $x\in \R$,
 \begin{equation*}
   \begin{aligned}
     U_t^{-1} \pi^X (Y_i) U_t f (x)&=
     \imath [ \sum_{j\geq 0} \frac{ (\Lambda \circ \text{ad}_X^j) (Y_i)}{j !}  e^{jt} x^j ] f(x) \\
     &= \imath [ \sum_{j\geq 0} \frac{ (\Lambda \circ
       \text{ad}_{X(t)}^j) (Y_i)}{j !}  x^j ] f(x) =
     \pi^{X(t)}_{\Lambda} (Y_i)(f) (x) \,.
   \end{aligned}
 \end{equation*}
 It follows that the representations $U_t^{-1} \pi^X_\Lambda U_t $ and
 $ \pi^{X(t)}_{\Lambda}$ are equal since they coincide on the basis
 $(X(t), Y)$ of the Lie algebra.
\end{proof}
 
The Sobolev norms of the invariant distribution and of the Green
operator with respect to the rescaled basis $\mathcal F(t)$ are given
below.
 
\begin{lemma}
  \label{lemma:Dnorm_scaled}
  For $\sigma >1/2$ and for all $t \in \R$, the following holds:
 $$
 \vert \Cal D^X_\Lambda \vert_{\mathcal F(t),-\sigma} =
 e^{\frac{t}{2}} I_\sigma(\Lambda, \mathcal F(t))\,.
 $$
\end{lemma}
\begin{proof} By Lemma~\ref{lemma:U_t}, by change of variable we have
  that, for all $f\in L^2(\R)$,
 $$
 \mathcal D^X_\Lambda (f) = \mathcal D^{X(t)}_\Lambda (U_t f) =
 e^{\frac{t}{2}} \mathcal D^{X(t)}_\Lambda(f)\,.
 $$
 It follows that, by Lemma~\ref{lemma:Distnorm}, we have
 $$
 \vert \Cal D^X_\Lambda \vert_{\mathcal F(t),-\sigma} =
 e^{\frac{t}{2}} \vert \Cal D^{X(t)}_\Lambda \vert_{\mathcal
   F(t),-\sigma} = e^{\frac{t}{2}} I_\sigma(\Lambda, \mathcal F(t))\,.
 $$
\end{proof}
  
Let $G^{X(t)}_{X,\Lambda}$ denote the Green operator for the
cohomological equation $X(t)u=f$ in the representation
$\pi^X_\Lambda$.  We recall that, according to our definitions above,
see formula \eqref{eq:Green}, $G^{X(t)}_{\Lambda}$ denote the Green
operator for the same cohomological equation $X(t)u=f$ in the
representation $\pi^{X(t)}_\Lambda$.
\begin{lemma}
  \label{lemma:Gbound_scaled}
  For $\sigma >\tau$ and for all $t \in \R$, for all $f \in \mathcal
  K^\infty(\pi^{X(t)}_{\Lambda})$, the following holds:
  $$
  \vert G^{X(t)}_{X,\Lambda} (f) \vert_{\mathcal F(t),\tau} \leq
  J^\tau_\sigma (\Lambda, \mathcal F(t)) \vert f \vert _{\mathcal
    F(t),\sigma} \,.
   $$
 \end{lemma}
 \begin{proof} By Lemma~\ref{lemma:U_t} the operators
   $G^{X(t)}_{X,\Lambda}$ and $G^{X(t)}_{\Lambda}$ are unitarily
   equivalent:
   \[
   G^{X(t)}_{X,\Lambda} = U_t \circ G^{X(t)}_\Lambda \circ U_t^{-1}\,
   \]
   hence by Lemma~\ref{lemma:Greenbound} we have, for all $f \in
   L^2(\R)$,
   \[
   \begin{aligned}
     \vert G^{X(t)}_{X,\Lambda} (f) \vert_{\mathcal F(t),\tau} &=
     \vert U_t^{-1}G^{X(t)}_{\Lambda} (U_t f) \vert_{\mathcal
       F(t),\tau} = \vert G^{X(t)}_{\Lambda} (U_t f) \vert_{\mathcal
       F(t),\tau} \\& \leq J^\tau_\sigma (\Lambda, \mathcal F(t))
     \vert U_t f \vert _{\mathcal F(t),\sigma} = J^\tau_\sigma
     (\Lambda, \mathcal F(t)) \vert f \vert _{\mathcal F(t),\sigma}\,.
   \end{aligned}
   \]
 \end{proof}

 Sobolev estimates of the normalised invariant distributions and of
 Green operators for rescaled bases are thus reduced to bounds on the
 integrals $I_\sigma (\Lambda, \mathcal F(t))$ and $J^\tau_\sigma
 (\Lambda, \mathcal F(t))$, which by Lemma \ref{lemma:IJbounds} can be
 bounded uniformly in terms of the weights $w_{\mathcal
   F(t)}(\Lambda)$. We therefore proceed to estimate the latter.

 Let $\mathcal F=(X,Y_1, \dots, Y_\adim)$ be any adapted basis and let
 $(d_1, \dots, d_\adim)\in \N^\adim$ denote the vector of the degrees
 of the elements $(Y_1, \dots, Y_\adim)$ respectively.  For any $\rho
 =(\rho_1, \dots, \rho_\adim) \in (\R^+)^\adim$, let
 \begin{equation}
   \label{eq:defofr}
   \r:=  \min_{\{i \,:\, d_i\not=0\}} \left( \frac{\rho_i}{d_i}\right)
 \end{equation}
 \begin{lemma}
   \label{lemma:weight_scaling}
   For any adapted basis $\mathcal F =(X,Y)$ and for all $t\geq 0$,
   \begin{equation}
     e^{-(1-\r)t}  w_{\mathcal F}(\Lambda) \leq  w_{\mathcal F(t)}(\Lambda)
     \leq    \left(\max_{\{i \,:\, d_i\not=0\}} \vert \Lambda^{(d_i)}_i(\mathcal F )\vert^{-\frac{1}{d_i}}\right) e^{-(1-\r)t} \,.
   \end{equation}
 \end{lemma}
 \begin{proof}
   Since
   \[
   \vert \Lambda ( \text{ad}^{d_i}_{X(t)} Y_i(t)
   )\vert^{-\frac{1}{d_i}} = e^{-t(1 - \rho_i/d_i)} \vert \Lambda (
   \text{ad}^{d_i}_{X} Y_i )\vert^{-\frac{1}{d_i}}
   \]
   the two inequalities follow immediately from the definition
   \eqref{eq:rep_weight} and the above definition of $\r$.
 \end{proof}
 We also estimate the normalised coefficients $\hat \Lambda(\mathcal
 F(t))$ of the representation. For convenience of notation we
 introduce the following weight: for all $\Lambda \in \mathfrak
 a^*_0$, let
 \begin{equation}
   \label{eq:Lambdaweight}
   \Vert \Lambda\Vert_{\mathcal F}  :=  \vert \Lambda(\mathcal F)\vert \,
   \max_{\{i \,:\, d_i\not=0\}} \left (1 + \frac{1}{ \Lambda^{(d_i)}_i (\mathcal F )}\right)\,.
 \end{equation}
 \begin{lemma}
   \label{lemma:normal_coeff}
   For any adapted basis $\mathcal F =(X,Y)$ and for all $i\in
   \{1,\dots, \adim\}$ the following bound holds: for all $t\geq 0$,
   \begin{equation}
     \label{eq:normal_coeff}
     \vert \hat\Lambda(\mathcal F(t)) \vert \leq  \Vert \Lambda\Vert_{\mathcal F} \,.
   \end{equation}
 \end{lemma}
 \begin{proof}
   By the definitions \eqref{eq:parameters}, \eqref{eq:Lambda_iFnorm},
   ~\eqref{eq:renorm} and~\eqref{eq:rescaledbasis}, we have, for all
   $i=1, \dots, \adim$ and $j=1, \dots , d_i$
  $$
  {\hat\Lambda}_i^{(j)}\big(\mathcal F(t)\big)=
  \Lambda_i^{(j)}\big(\mathcal F(t)\big)\, \big(w_{\mathcal
    F(t)}(\Lambda)\big)^j= e^{(j-\rho_i) t} \,
  \Lambda_i^{(j)}(\mathcal F)\, \big(w_{\mathcal
    F(t)}(\Lambda)\big)^j.
  $$ 
  By Lemma~\ref{lemma:weight_scaling} and observing that $j \le d_i$
  we obtain
  $$
  \vert {\hat\Lambda}_i^{(j)}(\mathcal F(t))\vert\le \vert
  \Lambda_i^{(j)}(\mathcal F)\vert \, \vert \max_{\{i \,:\,
    d_i\not=0\}} \vert \Lambda^{(d_i)}_i (\mathcal F
  )\vert^{-\frac{j}{d_i}};
  $$
  the bound~\eqref{eq:normal_coeff} follows from the elementary
  estimate
  $$
  \vert \Lambda^{(d_i)}_i (\mathcal F )\vert^{-\frac{j}{d_i}} \leq
  \left (1 + \frac{1}{ \vert \Lambda^{(d_i)}_i (\mathcal F
      )\vert}\right)\,,
  $$
  and from the definitions~\eqref{eq:Lambda_Fnorm_nohat}
  and~\eqref{eq:Lambda_Fnorm}.
\end{proof}

We finally conclude the section with the fundamental estimates on the
scaling of invariant distributions of the Green operator.
\begin{theorem}
  \label{thm:Dist_renorm}
  For all $\sigma>1/2$, there exists a constant $D_{\lstep,\sigma}>0$
  such that, for all $t \geq 0$,
  \begin{equation*}
    \vert  \mathcal D^X_\Lambda \vert _{\mathcal F, -\sigma} 
    \leq  D_{\lstep,\sigma} (1+\Vert \Lambda \Vert_{\mathcal F}) ^{\sigma+1}  
    e^{-\frac{\r}{2} t}
    \vert \mathcal D^X_\Lambda \vert _{\mathcal F(t), -\sigma}\,.
  \end{equation*}
\end{theorem}
\begin{proof}
  By Lemma~\ref{lemma:Distnorm} and Lemma~\ref{lemma:Dnorm_scaled} we
  have that
  \begin{equation}
    \label{eq:Iratio}
    \frac{ \vert  \mathcal D^X_\Lambda \vert _{\mathcal F, -\sigma} } 
    {\vert  \mathcal D^X_\Lambda \vert _{\mathcal F(t), -\sigma}} =  e^{-\frac{t}{2}} 
    \frac{I_\sigma(\Lambda, \mathcal F)} {I_\sigma(\Lambda, \mathcal F(t))}\,.
  \end{equation}
  By Lemma~\ref{lemma:IJbounds}, Lemma~\ref{lemma:weight_scaling} and
  Lemma \ref{lemma:normal_coeff} we obtain the estimate
  \begin{equation}
    \label{eq:Iratio_bound}
    \begin{aligned}
      \frac{I_\sigma(\Lambda, \mathcal F)} {I_\sigma(\Lambda, \mathcal
        F(t))} &\leq C^2_{k,\sigma} (1 + \vert \hat \Lambda (\mathcal
      F) \vert) (1 + \vert \hat \Lambda (\mathcal F(t)\vert)^\sigma
      \left( \frac{ w_{\mathcal F}(\Lambda)}{ w_{\mathcal F(t)}(\Lambda)}\right)^{1/2} \\
      &\leq C^2_{k,\sigma} (1 + \vert \Lambda (\mathcal F)
      \vert)^{\sigma+1} \max_{\{i \,:\, d_i\not=0\}} \left (1 +
        \frac{1}{ \Lambda^{(d_i)}_i (\mathcal F )}\right)^{\sigma+1}
      e^{\frac{1-\r}{2} t} \,.
    \end{aligned}
  \end{equation}
  And the statement follows.
\end{proof}

\begin{theorem}
  \label{thm:Green_renorm}
  For $\sigma >\tau(\lstep-1)+1$ there exists a constant $G_{k,\sigma,
    \tau}>0$ such that, for all $t \in \R$ and for all $f \in \mathcal
  K^\infty(\pi^{X(t)}_{\Lambda})$, the following holds:
  \begin{equation}
    \vert  G^{X(t)}_{X,\Lambda} (f) \vert_{\mathcal F(t),\tau}  \leq  
    G_{k,\sigma, \tau} (1+\Vert \Lambda\Vert_{\mathcal F})^{\tau k+2}
    e^{-(1-\r)t}\,   \vert f  \vert _{\mathcal F(t),\sigma} \,.
  \end{equation}
\end{theorem}
\begin{proof} It is an immediate consequence of the
  estimate~\eqref{eq:Jstbounds} of Lemma~\ref{lemma:IJbounds} and of
  Lemma~\ref{lemma:weight_scaling}.
\end{proof}

\subsubsection{A Lyapunov norm}

For convenience we introduce a Lyapunov norm on the space of invariant
distributions in each irreducible unitary representation. For any
adapted basis $\mathcal F$, for all $\Lambda \in \mathfrak a^*_0$ and
for all $\sigma >1/2$, let
\begin{equation}
  \label{eq:Lyapunov}
  \Vert \mathcal D^X_\Lambda \Vert _{\mathcal F, -\sigma} := 
  \inf_{\tau\geq 0} e^{-\frac{\r}{2} \tau} 
  \vert \mathcal D^X_\Lambda \vert _{\mathcal F(\tau), -\sigma} 
\end{equation}
It follows from the definition and from Theorem~\ref{thm:Dist_renorm}
that
\begin{equation}
  \label{eq:Lyap_comp}
  \frac{\vert \mathcal D^X_\Lambda \vert _{\mathcal F, -\sigma} }{ D_{k,\sigma} 
    (1+\Vert \Lambda\Vert_{\mathcal F})^{\sigma+1} }  
  \leq \Vert \mathcal D^X_\Lambda \Vert _{\mathcal F(t), -\sigma} \leq 
  \vert \mathcal D^X_\Lambda \vert _{\mathcal F, -\sigma} \,.
\end{equation}
\begin{lemma}
  \label{lemma:Lyap}
  For all $t\geq 0$, we have
$$
\Vert \mathcal D^X_\Lambda \Vert _{\mathcal F, -\sigma} \leq
e^{-\frac{\r}{2} t } \Vert \mathcal D^X_\Lambda \Vert _{\mathcal F(t),
  -\sigma}\,.
$$
\end{lemma}
\begin{proof} It follows immediately from the definition of the
  norm. In fact,
$$
\begin{aligned}
  \Vert \mathcal D^X_\Lambda \Vert _{\mathcal F, -\sigma} &=
  \inf_{\tau\geq 0} e^{-\frac{\r}{2} \tau} \vert
  \mathcal D^X_\Lambda \vert _{\mathcal F(\tau), -\sigma} \\
  &= e^{-\frac{\r}{2} t } \inf_{t+\tau\geq 0} e^{-\frac{\r}{2} \tau}
  \vert \mathcal D^X_\Lambda \vert _{\mathcal F(t+\tau), -\sigma} \leq
  e^{-\frac{\r}{2} t } \Vert \mathcal D^X_\Lambda \Vert _{\mathcal
    F(t), -\sigma}\,.
\end{aligned}
$$
\end{proof}

\section{Average width estimates}
\label{sec:width}

In this section we prove estimates on the average width of orbits of
quasi-Abelian nilflows.  Let $\alpha:=(\alpha_i^{(m)}) \in \R^{J}$ and
let $X_\alpha$ be the vector field on $M$ defined in formula
\eqref{eq:X_alpha}, that is
\[
X_{\alpha}:=\log \Big[\,x^{-1}\exp\Big( \sum_{(m,i)\in J}
\alpha_i^{(m)} \tilde \eta_i^{(m)}\Big)\,\Big]\,.
\]
Let $\{ \phi_{X_\alpha}^t\}$ be the flow generated by the vector field
$X_\alpha$ on $M$.

\subsection{Almost periodic points}
\label{sec:AP}

Let us recall that $\torus_\theta$ denotes the fibre at $\theta\in
\T^1$ of the fibration $\text{ \rm pr}_2: M \to \T^1$ (see formula
\eqref{eq:fibration_2}), that $\Phi_{\alpha,\theta}$ denotes the first
return map of the quasi-Abelian nilflow $\{\phi_{X_\alpha}^t\}$ to the
transverse torus $\torus_\theta$ (see Lemma~\ref{lem:return_maps}) and
that $\Phi^r_{\alpha,\theta} :\torus_\theta \to\torus_\theta$ denotes,
for all $r\in \Z$, the $r$-th iterate of the map
$\Phi_{\alpha,\theta}$.  By its definition, the map
$\Phi_{\alpha,\theta}$ commutes with the action of the centre $Z(G)$
of quasi-Abelian nilpotent group~$G$; hence, for all $r\in \Z$, the
map $\Phi^r_{\alpha,\theta} - \text{Id}$ induces a quotient map
\[
\Psi^{(r)}_{\alpha,\theta}: \torus_\theta/Z(G) \to \torus_\theta\,.
\]

For every $m=1, \dots, n$ and for all $\theta\in \T^1$, let
$\Z^{i_m}_\theta \subset \R^{i_m}$ be the lattice introduced in
formula~\eqref{eq:theta_lattice} and let $\T^{i_m}_\theta \subset
\T^\adim$ be the related sub-torus introduced in
formula~\eqref{eq:theta_subtorus}.  By Lemma~\ref{lem:return_maps} the
map $\Psi^{(r)}_{\alpha,\theta}$ has a factorisation
$$
\Psi^{(r)}_{\alpha,\theta} \approx \Psi^{(r)}_{\alpha^{(1)},\theta}
\times \dots \times \Psi^{(r)}_{\alpha^{(n)},\theta} \quad \text{ \rm
  on } \quad \big(\T^{i_1}_\theta \times \dots \times
\T^{i_n}_\theta\big) / Z(G)\,,
$$
and, for every $m=1, \dots, n$, the factor map
$\Psi^{(r)}_{\alpha^{(m)},\theta}$ is given in the coordinates
$\bs^{(m)} \in \R^{i_m}\mod\Z^{i_m}_\theta$ by the formulas
\begin{equation}
  \label{eq:Psi_r}
  \begin{aligned}
    \Psi^{(r)}_{\alpha^{(m)},\theta} (\bs^{(m)} )=
    \big(&r\,\alpha^{(m)}_1, r(s^{(m)}_1+\alpha^{(m)}_2)+
    \tbinom{r}{2}\,\alpha^{(m)}_1, \dots,\\
    &\sum_{i=1}^{i_m-1}
    \tbinom{r}{i}(s^{(m)}_{i_m-i}+\alpha^{(m)}_{i_m-i+1}) +
    \tbinom{r}{i_m}\alpha^{(m)}_1\big)\,.
  \end{aligned}
\end{equation}
It is clear that the above formulas define a map on the quotient
$\T^\adim_\theta /Z(G)$. In fact, for all $m=1, \dots, n$, the map
$\Psi^{(r)}_{\alpha^{(m)},\theta}$ does not depend on the coordinate
$s^{(m)}_{i_m} \in \R$.  It is also clear that the first coordinate of
the map $\Psi^{(r)}_{\alpha^{(m)},\theta}$ is constant (equal to
$r\,\alpha^{(m)}_1$), hence the image of the map
$\Psi^{(r)}_{\alpha,\theta} $ is contained in the affine
$(\adim-n)$-dimensional sub-torus
$$
\T^{\adim-n}_{\theta, \alpha,r} := \text{pr}_1^{-1}\{ \text{pr}_1
\big(\Gamma \exp (\theta\xi) \exp(r \sum_{m=1}^n\alpha_1^{(m)}
\tilde\eta_1^{(m)}) \big)\} \subset \T^\adim_\theta\,.
$$
We recall that, according to formulas~~\eqref{eq:fibration_1} the map
$\text{pr}_1: M\to M_1\approx\T^{n+1}$ is the projection on the base
torus.  Let $\meas^{\adim-n}_\theta$ denote the
$(\adim-n)$-dimensional conditional Lebesgue measure on the torus
$\torus_\theta/Z(G)$. By construction the coordinates
$$
(s^{(1)}_1, \dots, s^{(1)}_{i_1-1}, \dots, s^{(n)}_1, \dots,
s^{(n)}_{i_n-1}) \in \R^{\adim-n}\,,
$$ 
taken modulo the action of the lattice $\Z^\adim_\theta \cap
\{s^{(1)}_{i_1}= \dots= s^{(n)}_{i_n} =0\}$, are well-defined
coordinates for the quotient torus $\torus_\theta / Z(G)$; by the
above discussion, the measure $\meas^{\adim-n}_\theta$ can be written
in coordinates as follows:
$$
\meas^{\adim-n}_\theta = ds^{(1)}_1 \dots ds^{(1)}_{i_1-1} \dots
ds^{(n)}_1 \dots ds^{(n)}_{i_n-1} \,.
$$
Similarly, let $\meas^{\adim-n}_{\theta,\alpha,r}$ denote
$(\adim-n)$-dimensional conditional Lebesgue measure on the torus
$\T^{\adim-n}_{\theta,\alpha, r}$. The coordinates
$$
(s^{(1)}_2, \dots, s^{(1)}_{i_1}, \dots, s^{(n)}_2, \dots,
s^{(n)}_{i_1})\in \R^{\adim-n}\,,
$$ 
taken modulo the action of the lattice $\Z^\adim_\theta \cap
\{s^{(1)}_1=\dots= s^{(n)}_1=0\}$ are well-defined coordinates for the
sub-torus $\T^{\adim-n}_{\theta,\alpha, r}$; the measure
$\meas^{\adim-n}_{\theta,\alpha,r}$ can be written in coordinates as
follows:
$$
\meas^{\adim-n}_{\theta,\alpha,r} = ds^{(1)}_2 \dots ds^{(1)}_{i_1}
\dots ds^{(n)}_2 \dots ds^{(n)}_{i_n} \,.
$$

\begin{lemma}
  \label{lem:covering_map}
  For all $\alpha\in \R^{J}$ and all $\theta\in\T^1$ the map
  $\Psi^{(r)}_{\alpha,\theta}$ is a covering map of the torus $
  \torus_\theta/Z(G)$ onto the torus $\T^{\adim-n}_{\theta,\alpha, r}$
  with constant Jacobian. Hence it maps the measure
  $\meas^{\adim-n}_\theta$ onto the the measure $\meas^{\adim
    -n}_{\theta,\alpha,r} $.
\end{lemma}
\begin{proof}
  By formula~\eqref{eq:Psi_r}, for any $\alpha\in \R^{J}$, for every
  $m=1, \dots, n$ and for every $j\in\{2, \dots, i_m\}$, there exists
  a polynomial $p^{(m)}_j(\alpha,r, s^{(m)}_1, \dots , s^{(m)}_{j-2})$
  such that the $j$-th coordinate $\Psi^{(r)}_{\alpha^{(m)},\theta,j}$
  of the map $\Psi^{(r)}_{\alpha^{(m)},\theta}$ is given by the
  following formula: for all $\bs^{(m)} \in \R^{i_m}$,
$$
\Psi^{(r)}_{\alpha^{(m)},\theta,j} (\bs^{(m)} ) = r s^{(m)}_{j-1} +
p_j\big(\alpha,r, s^{(m)}_1, \dots , s^{(m)}_{j-2}\big)\,.
$$
It follows that the Jacobian of $\Psi^{(r)}_{\alpha,\theta}$ is a non
zero constant and the map is a regular covering. Hence the
push-forward of the Lebesgue measure $\meas^{\adim-n}_\theta$ on
$\torus_\theta/Z(G)$ to the torus $\T^{\adim-n}_{\theta,\alpha, r}$
under $\Psi^{(r)}_{\alpha,\theta}$ is the Lebesgue measure
$\meas^{\adim-n}_{\theta,\alpha,r}$.
\end{proof}

Let $\mathcal U_\theta$ be a given neighbourhood of the origin in
$\R^\adim$.  A point $x \in \torus_\theta$ is \emph{$(\mathcal
  U_\theta,r)$-almost-periodic}, that is, it is \emph{$\mathcal
  U_\theta$-almost-periodic} of \emph{period~$r \in \N$} for
$\Phi_{\alpha,\theta}$ if $\Phi^r_{\alpha,\theta} (x)$ belongs to the
neighbourhood $x+ \mathcal U_\theta$ of $x\in \torus_\theta$. For
every $r\in \Z\setminus\{0\}$, let $\AP^r(\mathcal U_\theta) $ be the
set of $(\mathcal U_\theta,r)$-almost-periodic points:
\begin{equation*}
  \AP^r(\mathcal U_\theta)  :=
  \{ x \in \torus_\theta \mid   \Phi^r_{\alpha,\theta} (x) -x \in 
  \mathcal U_\theta \} \,.
\end{equation*}
The next lemma estimates the Lebesgue measure of the set
$\AP^r(\mathcal U_\theta)$ of $(\mathcal U_\theta,r)$-almost-periodic
points. To this purpose let us introduce yet another projection map:
for all $\theta\in \T^1$, let $\pr_\theta: \torus_\theta \to \T^n$ be
the restriction to the torus $\torus_\theta$ of the projection
$\text{pr}_1: M\to \T^{n+1}$ onto the base torus, that is, the map
defined by the following formula: for all $\bs \in \R^\adim$,
\begin{equation*}
  \pr_\theta : \Gamma \exp(\theta \xi) \exp\left(  \sum_{m=1}^n \sum_{i=1}^{i_m} s^{(m)}_{i_m} \tilde
    \eta^{(m)}_{i_m}\right) =
  (s^{(1)}_1, \dots, s^{(n)})  \ (\hbox{\rm mod. } \Z^n)\,.
\end{equation*}

\begin{lemma}
  \label{lem:almost_per}
  Let $\theta\in \T^1$ and let $\mathcal U_\theta\subset
  \torus_\theta$ be any neighbourhood of the origin.  For all $r\in
  \Z\setminus\{0\}$, the $\adim$-dimensional conditional Lebesgue
  measure~$\meas^\adim_\theta $ of the set $\AP^r(\mathcal U_\theta)
  \subset \torus_\theta$ is given as follows.
 
  If $ r \alpha_1 \in \pr_\theta (\mathcal U_\theta)$, then
  \[
  \meas^\adim_\theta \left(\AP^r(\mathcal U_\theta)\right) =
  \meas^{\adim-n}_{\theta,\alpha,r} \left(\mathcal
    U_\theta\cap\T^{\adim-n}_{\theta,\alpha,r}\right) \,;
  \]
  otherwise $\AP^r(\mathcal U_\theta)=\emptyset$.
\end{lemma}
\begin{proof} The sets $\AP^r(\mathcal U_\theta)$ are invariant under
  the action of the centre~$Z(G)$. By definition, the projection of a
  set $\AP^r(\mathcal U_\theta)/Z(G)$ to the quotient torus
  $\torus_\theta/Z(G)$ is the inverse image of the neighbourhood
  $\mathcal U_\theta\subset \torus_\theta$ under the map defined on
  $\torus_\theta/Z(G)$ as
  \[
  x \mod {Z(G)} \mapsto \left( r\alpha_1,
    \Psi^{(r)}_{\alpha,\theta}(x) \right) \in \torus_\theta\,.
  \]
  Thus $\AP^r(\mathcal U_\theta)=\emptyset$ if $r\alpha_1 \not \in
  \pr_\theta(\mathcal U_\theta)$; if $r\alpha_1 \in
  \pr_\theta(\mathcal U_\theta)$, then
  \[
  \AP^r(\mathcal U_\theta)/Z(G)=
  (\Psi^{(r)}_{\alpha,\theta})^{-1}(\mathcal U_\theta \cap
  \T^{\adim-n}_{\theta,\alpha,r} )\,.
  \]
  The result then follows from Lemma~\ref{lem:covering_map}.
\end{proof}

\begin{definition}
  \label{def:AY}
  For any basis $Y=\{Y_1, \dots,Y_\adim\}$ of the Abelian ideal
  $\mathfrak a \subset \mathfrak g$, let $\injconst:=\injconst(Y)$ be
  the supremum of all constants {\color{red}$\injconst'\in ]0,1/2[$}
  such that for any $x\in M$ the map
  \[
  \phi^ Y_x : (s_1, \dots ,s_\adim) \to x \exp( \sum_{i=1}^{\adim} s_i
  Y_i ) \in M\,.
  \]
  is a local embedding (injective) on the domain
$$\{ \bs \in \R^\adim\vert |s_i|< \injconst'
\text{ for all } i=1, \dots, \adim\} \,.
$$

For any $x$, $x'\in M$, we set
$$
\Vert x'-x \Vert_1 = \vert s_1\vert, \dots, \Vert x'-x \Vert_i =\vert
s_i\vert ,\dots, \Vert x'-x \Vert_\adim =\vert s_\adim\vert\,,
  $$
  if there is $\bs:=(s_1, \dots, s_\adim) \in [-\injconst/2,
  \injconst/2]^\adim$ such that
  \[
  x' = x\exp( \sum_{i=1}^{\adim} s_i Y_i) ;
  \]
  otherwise we set $ \Vert x'-x \Vert_1 = \dots = \Vert x'-x
  \Vert_\adim = \injconst\,. $
\end{definition}

\smallskip
  
Let $\rho:=(\rho_1, \dots, \rho_\adim) \in [0,1)^{\adim}$ and let
$\mathcal F_\alpha=(X_\alpha, Y)$ be a normalised \emph{strongly
  adapted basis} (see Definition~\ref{def:adaptedbasis}).  Let us
observe that, since the basis $\mathcal F_\alpha$ is strongly adapted,
the vector
$$
(\Vert \Phi_{\alpha,\theta}^r (x)-x \Vert_1, \dots, \Vert
\Phi_{\alpha,\theta}^r (x)-x \Vert_n)
$$
does not depends upon the choice of $x \in M$, but only depends on
$r\in \Z$; in fact, the subsystem $(Y_{n+1}, \dots, Y_\adim)$ is
tangent to the fibres of the projection $\text{pr}_\theta:
\torus_\theta \to \T^n$, and, for all $x\in \torus_\theta$, we have
$$
\text{pr}_\theta \big( \Phi_{\alpha,\theta}^r (x)-x \big) = r\alpha_1
\mod \Z^n\,.
$$
It follows that, for any $L\geq 1$ and for any $r\in \Z$, we can
define
\begin{equation}
  \label{eq:delta_one}
  \deltaone{r}{L} := \max_{1\leq i \leq n} \min\{\injconst,  L^{\rho_i}
  \Vert  \Phi_{\alpha,\theta}^r (x)-x \Vert_i\}. 
\end{equation}
For $L\geq 1$, $r\in \Z$ and $x\in \torus_\theta$, we also define
\begin{equation*}
  \label{eq:deltaother}
  \deltaother{r}{L} (x) := \max_{n< i \leq \adim} \min \{\injconst,  L^{\rho_i}
  \Vert \Phi_{\alpha,\theta}^r (x) - x  \Vert_i\} \,.
\end{equation*}
Let us observe that the conditions $\deltaone{r}{L}
<\epsilon<\injconst$ and $\delta' <\deltaother{r}{L} (x) <\delta<
\injconst$ are equivalent to saying that
$$
\Phi_{\alpha,\theta}^r (x) = x \exp(\sum_{i=1}^\adim s_i Y_i)
$$
for some vector $\bs=(s_1, \dots, s_\adim)\in
[-\injconst/2,\injconst/2]^\adim$ such that
$$
\begin{aligned}
  |s_i| &< \epsilon L^{-\rho_i} \,, \quad \text{ for all } i\in \{1, \dots,n\}\,; \\
  |s_i| &< \delta L^{-\rho_i} \,, \quad \text{ for all }  i\in \{n+1, \dots, \adim\}\,; \\
  |s_j| &> \delta' L^{-\rho_j} \quad \text{ for some } j\in
  \{n+1,\dots, \adim\}\,.
\end{aligned}
$$
For every $r \in \Z\setminus \{0\}$ and $j\geq 0$, let $\APS{r}{j}{L}
\subset M$ be the sets defined as follows
\begin{equation}
  \label{eq:AP_r}
  \APS{r}{j}{L} :=
  \begin{cases}
    \emptyset \,, &\text{if }\deltaone{r}{L} > \injconst/2 \,; \\
    (\deltaother{r}{L})^{-1} \Big(\big]{2^{-(j+1)}\injconst} ,
    {2^{-j}\injconst} \big]\Big)\,, &\text{otherwise}\,.
  \end{cases}
\end{equation}
\begin{lemma}
  \label{lem:AP_r}
  For all $r\in \Z\setminus\{0\}$, for all $j\in \N$, and for all
  $L\geq 1$, the $(\adim+1)$-dimensional Lebesgue measure
  $\meas^{\adim+1}$ of the set $\APS{r}{j}{L}$ can be estimated as
  follows:
  \[
  \meas^{\adim+1} (\APS{r}{j}{L}) \leq
  \frac{\injconst(Y)^{\adim-n}}{2^{j(\adim-n)}} L^{ -
    \sum_{i=n+1}^{\adim} \rho_i }\,.
  \]
\end{lemma}
\begin{proof} Without loss of generality we can assume that
  $\APS{r}{j}{L} \not=\emptyset$, otherwise there is nothing to prove.
  By Tonelli's Theorem,
  \begin{equation}
    \label{eq:Tonelli_thm}
    \meas^{\adim+1}(\APS{r}{j}{L})
    =\int_0^1\meas^{\adim}_\theta(\APS{r}{j}{L} \cap
    \torus_\theta) \, \D\theta\,,
  \end{equation}
  hence the statement can be reduced to estimates on the
  $\adim$-dimensional Lebesgue measure
  $\meas^{\adim}_\theta(\APS{r}{j}{L})$ for $\theta \in \T^1$. For
  every $j\in \N$, let
  \[
  \mathcal U^{L,j}_{\theta}:=\Big\{x\in \torus_\theta \vert
  \max_{n+1\leq i\leq \adim} L^{\rho_i} \Vert x \Vert_i \leq
  \injconst/ 2^j \Big\}\,.
  \]
  By definition, if $x\in \APS{r}{j}{L} \cap \torus_\theta$, then
  $\Vert \Phi_{\alpha,\theta}^r (x) - x \Vert_i\le 2^{-j} \injconst
  L^{-\rho_i}$ for all $i=n+1, \dots, \adim$, that is, $\APS{r}{j}{L}
  \cap \torus_\theta \subset \AP^r(\mathcal U^{L,j}_\theta)$.  By
  Lemma~\ref{lem:almost_per} we have
  \[
  \begin{split}
    \meas^{\adim}_\theta(\APS{r}{j}{L} \cap \torus_\theta) &\leq
    \meas^{\adim}_\theta(\AP^r(\mathcal U^{L,j}_\theta)) \\ &=
    \meas^{\adim-n}_{\theta,\alpha,r} (\mathcal U^{L,j}_\theta\cap
    \T^{\adim-n}_{\theta,\alpha,r}) =
    \frac{\injconst^{\adim-n}}{2^{j(\adim-n)}} L^{ -
      \sum_{i=n+1}^\adim \rho_i }.
  \end{split}
  \]
  The statement thus follows from Tonelli theorem (see formula
  \eqref{eq:Tonelli_thm}).

\end{proof}

\subsection{Expected width bounds }
\label{sec:expectedbounds}
In this section we prove a bound on the average width of a
quasi-Abelian nilpotent orbit with respect to a rescaled basis in
terms of the ergodic average along the orbit of an appropriate
function on the nilmanifold (which depends on the length of the orbit
and on the rescaling exponents).

The expected value of the average width is thus bounded in terms of
the average of such a function over the nilmanifold. Such an estimate
is then reduced to a Diophantine estimate.

For $L\geq 1$, $r\in \Z\setminus\{0\}$, let us consider the function
\begin{equation}
  \label{eq:h_function_1}
  h_{r,L} :=  \sum_{j=1}^{+\infty}  \min \{ 2^{j(\adim-n)} , (\frac{2}{\deltaone{r}{L}})^n
  \}  \chi_{\APS{r}{j}{L}}\,.
\end{equation}
Let us introduce the cut-off $J_{r,L} \in \N$ by the formula:
\begin{equation}
  \label{eq:cut_off}
  J_{r,L} :=  \max \{  j\in \N \vert   2^{j(\adim-n)} \leq  (\frac{2}{\deltaone{r}{L}})^n \}\,.
\end{equation}
The function in formula~\eqref{eq:h_function_1} can also be written as
follows:
\begin{equation}
  \label{eq:h_function_2}
  h_{r,L} :=  \sum_{j=1}^{J_{r,L}}   2^{j(\adim-n)}  \chi_{\APS{r}{j}{L}} +  
  \sum_{j>J_{r,L}} (\frac{2 } {\deltaone{r}{L} })^n \chi_{\APS{r}{j}{L}} \,.
\end{equation}
For every $L\geq 1$ let $\mathcal F_\alpha^{(L)}$ be the rescaled
strongly adapted normalised basis
\begin{equation}
  \label{eq:width:1}
  \begin{split}
    \mathcal F_\alpha^{(L)}&= (X_\alpha^{(L)},Y_1^{(L)},\dots,
    Y_\adim^{(L)}) \\&= (L \, X_\alpha , L^{- \rho_1} \, Y_1 ,\dots,
    L^{- \rho_\adim} \, Y_\adim ) \,.
  \end{split}
\end{equation}
For $(x, T) \in M\times \R^+$, let $w_{\mathcal F_\alpha^{(L)} }
(x,T)$ denote the average width of the orbit segment
$$
\gamma_{X_\alpha^{(L)}}^T(x):= \{ \phi^t_{X_\alpha^{(L)}} (x) \mid
0\leq t \leq T\}.
$$
We prove below a bound for the average width $w_{\mathcal
  F_\alpha^{(L)}} (x,T)$ of the orbit arc
$\gamma_{X_\alpha^{(L)}}^T(x)$ in terms of the following function:
\begin{equation}
  \label{eq:H_function}
  H^T_{L} : =1 +  \sum_{\vert r\vert=1}^{[TL]}
  h_{r,L}.
\end{equation}

\begin{lemma}
  \label{lemma:width_bound}
  Let $\mathcal F_\alpha=(X_\alpha, Y)$ be any normalised strongly
  adapted basis.  For all $x \in M$ and for all $T, L\ge 1$ we have
  \[
  \frac{1}{w_{\mathcal F_\alpha^{(L)}}(x,T)} \leq \left(\frac 2
    {\injconst(Y)}\right)^{\adim} \frac{1}{T} \int_{0}^T H^T_{L} \circ
  \phi^{t}_{X_\alpha^{(L)}} (x)\,\D t .
  \]
\end{lemma}
\begin{proof} Let $x$, $T$, $L$ and $L$ be defined as in the
  statement.  For every $t \in [0,T]$, we define a set $\Omega(t )
  \subset \{t \} \times \R^\adim$ as follows:
  \begin{itemize}
  \item[(A)] If $\phi^t _{X_\alpha^{(L)}}(x) \not \in \bigcup_{\vert
      r\vert=1}^{[TL]} \bigcup_{j>0} \APS{r}{j}{L}$, let $\Omega(t )$
    be the set of all points $(t , s_1, \dots, s_\adim)$ such that
    \[
    \vert s_1\vert < \injconst/4, \dots, \vert s_\adim \vert <
    \injconst/4 \,.
    \]
  \end{itemize}
  Observe, incidentally, that if $\phi^t _{X_\alpha^{(L)}}(x) \not \in
  \bigcup_{\vert r\vert=1}^{[TL]} \bigcup_{j>0} \APS{r}{j}{L}$, then
  for all $|r|\in [1, TL]$ such that $\deltaone{r}{L}\le \injconst/2$
  we must have that $x\in\APS{r}{0}{L}$; in fact for such an $r$ we
  have $\bigcup_{j\ge 0} \APS{r}{j}{L}=M$.

  To define the set $\Omega(t )$ when $\phi^t _{X_\alpha^{(L)}}(x) \in
  \bigcup_{\vert r\vert=1}^{[TL]} \bigcup_{j>0} \APS{r}{j}{L}$ we
  consider two sub-cases.
  \begin{itemize}
  \item[(B)] if $\phi^t _{X_\alpha^{(L)}}(x) \in \bigcup_{\vert
      r\vert=1}^{[TL]} \bigcup_{j> J^r_L} \APS{r}{j}{L}$, let
    $\Omega(t )$ be the set of all points $(t , s_1, \dots, s_\adim)$
    such that
    \[
    \begin{split}
      \qquad\vert s_i\vert &< \frac{1}{4} \min_{1\le |r| \le [TL]}
      \min_{j>J(\vert r\vert )} \{\deltaone{r}{L}: \phi^t
      _{X_\alpha^{(L)}}(x) \in \APS{r}{j}{L}\} \,,\,\, \text{
        \rm for } i \in \{1, \dots, n\}\,, \\
      \qquad\vert s_i \vert &< \frac{\injconst}{4} \,,\,\, \text{ \rm
        for } i \in \{n+1, \dots, \adim\} \,;
    \end{split}
    \]
  \item[(C)] finally, if $\phi^t _{X_\alpha^{(L)}}(x) \in
    \bigcup_{\vert r\vert=1}^{[TL]} \bigcup_{j\le J^r_L}
    \APS{r}{j}{L}\setminus \bigcup_{\vert r\vert=1}^{[TL]} \bigcup_{j>
      J^r_L}\APS{r}{j}{L}$, let $\Omega(t )$, let $\ell$ be the
    largest integer such that
    \[\phi^t _{X_\alpha^{(L)}}(x) \in \bigcup_{\vert
      r\vert=1}^{[TL]}\bigcup_{\ell\le j\le J^r_L}
    \APS{r}{j}{L}\setminus \bigcup_{\vert r\vert=1}^{[TL]} \bigcup_{j>
      J^r_L}\APS{r}{j}{L}\] and let $\Omega(t )$ be the set of all
    points $(t , s_1, \dots, s_\adim)$ such that
    \[
    \begin{aligned}
      &\vert s_i\vert < \frac{\injconst}{4}\,, \,\,
      \text{ \rm for } i\in \{1,\dots, n\}\, \\
      &\vert s_i \vert < \frac{\injconst}{4} \frac{1} {2^{\ell+1}
      }\,,\,\, \text{ \rm for } i \in \{n+1, \dots, \adim\} \,;
    \end{aligned}
    \]
  \end{itemize}
  Then we set
  \[
  \Omega:= \bigcup_{t \in [0,T]} \Omega(t ) \subset [0,T] \times
  \R^\adim\,.
  \]

  It is clear that the set $\Omega$ contains the segment $[0,T] \times
  \{0\}$ and it is contained in $[0,T] \times
  [-4^{-1}\injconst,4^{-1}\injconst]^{\adim}$.

  We claim that the restriction to $\Omega$ of the map
  \begin{equation}
    \label{eq:width_map}
    (t , s_1, \dots , s_\adim) \in \Omega \mapsto x
    \exp(t  X_\alpha^{(L)}) \exp(s_1 Y_1^{(L)}+\cdots
    +s_\adim Y_\adim^{(L)})\in M
  \end{equation}
  is injective. In fact, let us assume that there exist points $(t ,
  s_1, \dots , s_\adim)\in \Omega$ and $(t ', s_1', \dots , s_\adim')
  \in \Omega$ such that
  \begin{equation}
    \label{eq:overlap}
    \begin{aligned}
      \phi^t _{X_\alpha^{(L)}}(x) &\exp(s_1 Y_1^{(L)}) +\cdots
      +s_\adim Y_\adim^{(L)}) \\ &= \phi^{t
        '}_{X_\alpha^{(L)}}(x)\exp(s'_1 Y_1^{(L)}+\cdots +s'_\adim
      Y_\adim^{(L)})\,.
    \end{aligned}
  \end{equation}
  We can assume $t '\ge t $. By considering the projection on the base
  torus $\T^{n+1}$ we have the following identity:
  \begin{equation}
    \label{eq:basetorus}
    \begin{aligned}
      (t ,s_1, \cdots, s_n) \mod \Z^{n+1}&= \pr_1( \phi^t
      _{X_\alpha^{(L)}}(x) ) \\ = \pr_1( \phi^{t
        '}_{X_\alpha^{(L)}}(x) ) &= (t ',s_1', \cdots, s_n') \mod
      \Z^{n+1} ;
    \end{aligned}
  \end{equation}
  this implies $t \equiv t '$ modulo $\Z$. As $\phi^t
  _{X_\alpha^{(L)}}=\phi^{t L}_{X_\alpha}$, the number $r_0=t '-t $ is
  a non negative integer satisfying $r_0\le TL$; hence $r_0\le [TL]$.

  If $r_0=0$, then $t =t '$ and $s_1=s'_1, \dots, s_\adim=s'_\adim$:
  in fact, by the definition of the constant $\injconst$, the map
  \[
  (s_1, \dots,s_\adim)\in [-4^{-1}\injconst,4^{-1}\injconst]^{\adim}
  \mapsto \phi^t _{X_\alpha^{(L)}}(x) \exp(s_1 Y_1^{(L)}) \cdots
  \exp(s_\adim Y_\adim^{(L)})
  \]
  is injective. We prove below that the overlapping
  identity~\eqref{eq:overlap} leads to a contradiction if we assume
  that $r_0\not =0$.
  
  The condition~\eqref{eq:basetorus} tells us that the points
  \[
  p := \phi^t _{X_\alpha^{(L)}}(x) \quad \text{ and }\quad q:= \phi^{t
    '}_{X_\alpha^{(L)}}(x)
  \]
  belong to the same torus $\torus_\theta$; the definition of $r_0$
  says that $q= \Phi_{\alpha,\theta}^{r_0}(p)$.  From
  identity~\eqref{eq:overlap} we have
  \[
  \begin{split}
    q &= p \exp((s'_1-s_1) Y_1^{(L)}+\cdots
    +   (s'_\adim- s_\adim)Y_\adim^{(L)})\\
    &=p \exp((s'_1-s_1) L^{-\rho_1}Y_1+(s'_2-s_2) L^{-\rho_2}Y_2\cdots
    + (s'_\adim- s_\adim) L^{-\rho_\adim}Y_\adim);
  \end{split}
  \]
  thus, for all $i\in\{1, \dots, \adim\}$,
  \begin{equation}
    \label{eq:resc_dist_bound}
    \begin{split}
      L^{\rho_i} \Vert p - \Phi^{r_0}_{\alpha,\theta}(p) \Vert_i
      &=L^{\rho_i} \Vert q - \Phi^{-r_0}_{\alpha,\theta}(q)
      \Vert_i \\
      &\le L^{\rho_i}| (s'_i-s_i) L^{-\rho_i}|\leq \vert s_i\vert +
      \vert s'_i\vert \,.
    \end{split}
  \end{equation}
  In particular, from formula~\eqref{eq:resc_dist_bound} we obtain
  that $\deltaone{r_0}{L}$, which is a constant on $M$, satisfies the
  following inequality:
  \begin{equation}
    \label{eq:deltaone_bound}
    \deltaone{r_0}{L}= \max_{1\leq i\leq n} L^{\rho_i} \Vert p -
    \Phi^r_{\alpha,\theta}(p) \Vert_i \leq \max_{1\leq i\leq n} \vert s_i\vert +
    \vert s'_i\vert < \injconst/2 \,.
  \end{equation}
  For the same reason, that is, from
  formula~\eqref{eq:resc_dist_bound}, we also obtain that
  \[
  \deltaother{r_0}{L}(p) = \deltaother{-r_0}{L}(q)< \injconst/2\,.
  \]

  By defining~$j_0\in \N$ as the unique non-negative integer such that
  \begin{equation}
    \label{eq:deltaother_level}
    \frac{\injconst}{2^{j_0+1}} < \deltaother{r_0}{L}(p)\le \frac{\injconst}{2^{j_0}} \,,
  \end{equation}
  and by the definition~\eqref{eq:AP_r}, we have that $p \in
  \APS{r_0}{j_0}{L}$ and $q\in \APS{-r_0}{j_0}{L}$.

  If $j_0> J^{r_0}_L=J^{-r_0}_L$, then $p,q \in \bigcup_{0<|r|\le
    [TL]}\bigcup_{j>J^r_L} \APS{r}{j}{L} $; it follows that the sets
  $\Omega(t )$ and $\Omega(t ')$ are both defined according to
  definition~(B); hence, from \eqref{eq:deltaone_bound} and the
  definition~(B), we obtain
$$
\deltaone{r_0}{L} \leq \max_{1\leq i\leq n} \vert s_i\vert + \vert
s'_i\vert \le \frac{1}{2} \deltaone{r_0}{L}\,,
$$
a plain contradiction.

Our conclusion at this point is that if the map in
formula~\eqref{eq:width_map} fails to be injective at the points $(t ,
s_1,\dots,s_\adim)$, $(t ', s'_1, \dots,s'_\adim)$, with $t \le t '$,
then there are integers $r_0 \in [1, TL]$, $ j_0\in [1,J(\vert
r_0\vert)]$ and $\theta \in \T^1$ such that the points $p= x \exp(t
X_\alpha^{(L)})$, $q= x \exp(t X_\alpha^{(L)})$ satisfy
\[
p,q\in \torus_\theta, \quad q = \Psi_{\alpha,\theta}^{r_0}(p) , \quad
p\in \APS{r_0}{j_0}{L},\quad q\in \APS{-r_0}{j_0}{L};
\]
in addition,
\[
p, q \not \in \bigcup_{0<|r|\le [TL]}\bigcup_{j>J^r_L} \APS{r}{j}{L}
\,.
\]

In this case the sets $\Omega(t )$ and $\Omega(t ')$ are both defined
according to definition~(C); by defining $\ell_1$ and $\ell_2$ as the
largest integers such that
\[
p \in \bigcup_{0<|r|\le [TL]}\bigcup_{\ell_1 \le j \le J^r_L}
\APS{r}{j}{L}
\]
and
\[
q \in \bigcup_{0<|r|\le [TL]}\bigcup_{\ell_2 \le j\le J^r_L}
\APS{r}{j}{L}\,,
\]
we have, from definition~(C),
\[
\vert s_i \vert <\frac{\injconst}{4} \frac{1}{2^{\ell_1+1} }\,,\quad
\vert s'_i\vert <\frac{\injconst}{4} \frac{1}{2^{\ell_2+1} }\,, \qquad
\text{ \rm for all }\, i\in \{n+1, \dots, \adim\}\,;
\]
this also leads to a contradiction because from
formulas~\eqref{eq:resc_dist_bound}, ~\eqref{eq:deltaother_level},
since by construction $\ell_1,\ell_2 \ge j_0$, we deduce that
$$
\frac{\injconst}{2^{j_0+1}} \le \deltaother{r_0}{L}(p) \leq
\max_{i\in\{n+1, \dots, \adim\}}\{ \vert s_i\vert + \vert s'_i\vert\}
< \frac{\injconst}{4}\frac{1}{2^{\ell_1+1} } +\frac{\injconst}{4}
\frac{1}{2^{\ell_2+1} }\le \frac{\injconst}{2} \frac{1}{2^{j_0+1} }\,,
$$
again a plain contradiction. The injectivity claim is therefore
proved.

We are finally ready to conclude the proof.  In fact the width
function $w_\Omega$ of the set $\Omega$ is given, by definition, by
the following formulas:
\begin{equation*}
  w_\Omega(t ) =  
  \begin{cases}
    \left(\frac \injconst 2\right)^{\adim}&{\small \text{if } \phi^t
      _{X_\alpha^{(L)}}(x) \not \in \bigcup_{\vert
        r\vert=1}^{[TL]}\bigcup_{j>0} \APS{r}{j}{L}};
    \\
    \left(\frac \injconst 2\right)^{\adim-n}
    \left(\frac{\min\{\deltaone{r}{L}\}}{2}\right)^n &{\small\text{if
      } \phi^t _{X_\alpha^{(L)}}(x) \in \bigcup_{\vert
        r\vert=1}^{[TL]}\bigcup_{j>J^r_L} \APS{r}{j}{L} \text{ and}}
    \\
    & {\small\text{where the min is over all $\vert r\vert\in [1, TL]$}}\\
    &{\small\text{such that } \small \phi^t _{X_\alpha^{(L)}}(x)  \in \bigcup_{j>J^r_L} \APS{r}{j}{L};}\\
    \left(\frac \injconst 2\right)^{\adim} 2^{-(\adim-n)(\ell+1)}
    &{\small \text{in the remaining case, where $\ell$ is the}}\\
    & {\small \text{largest integer $\le \max \{J^r_L \vert r \in [-TL, TL]\}$ }}\\
    &\small\text{such that } \small\phi^t _{X_\alpha^{(L)}}(x) \in
    \bigcup_{\vert r\vert=1}^{[TL]} \bigcup_{j=\ell}^{J^r_L}
    \APS{r}{j}{L}\,.
  \end{cases}
\end{equation*}
In the first case we have
\[
\frac{1}{w_\Omega(t )} \le \left(\frac 2 \injconst\right)^{\adim};
\]
in the second case
\[
\frac{1}{w_\Omega(t )} \le \left(\frac 2 \injconst\right)^{\adim-n}
\sum_{|r|=1}^{[TL]}\sum_{j>J^r_L} \frac{2^n\chi_{\APS{r}{j}{L}}(\phi^t
  _{X_\alpha^{(L)}}(x))}{(\deltaone{r}{L})^n }\,;
\]
in the third and last case
\[
\frac{1}{w_\Omega(t )} \le \left(\frac 2 \injconst\right)^{\adim-n}
\sum_{|r|=1}^{[TL]}\sum_{j=1}^{J^r_L} 2^{(\adim-n)(j+1)}
\chi_{\APS{r}{j}{L}}(\phi^t _{X_\alpha^{(L)}}(x)) \,.
\]
Thus, by the definition of the function $H^T_{L}$ in formula
\eqref{eq:H_function}, we have
$$
\frac{1}{w_\Omega(t )} \leq \left(\frac 2 \injconst\right)^{\adim} \,
H^T_{L} \circ \phi_{X_\alpha^{(L)}}^t (x) \,, \quad \text{ for all } t
\in [0,T]\,.
$$
From the definition~\eqref{eq:av_width} of the average width
$w_{\mathcal F_\alpha^{(L)}}(x,T)$ of the orbit segment $ \{x \exp(t
X_\alpha^{(L)} \vert 0 \le t \le T\}$, we have the estimate
\[
\frac{1}{w_{\mathcal F_\alpha^{(L)}}(x,T)} \le \frac 1 T \int_0^T
\frac {\D t }{w_\Omega(t )} \le \left(\frac 2 \injconst\right)^{\adim}
\frac 1 T \int_0^T H^T_{L} \circ \phi_{X_\alpha^{(L)}}^t (x) {\D t
}\,.
\]
The argument is therefore completed.
\end{proof}

\smallskip By Lemma~\ref{lemma:width_bound}, a bound on the expected
value of the inverse of the average width can be derived from the
following integral estimate:
\begin{lemma}
  \label{lem:exp_width}
  For all $r\in \Z\setminus\{0\}$ and for all $L\geq 1$ the following
  estimate holds:
  \[
  \vert \int_{M} h_{r,L}(x) \,\D x \vert \, \leq \,
  \injconst(Y)^{\adim-n} (1+ J^r_L) L^{-\sum_{i=n+1}^\adim \rho_i} \,.
  \]
\end{lemma}
\begin{proof}
  By Lemma~\ref{lem:AP_r} it follows that, for all $r \neq 0$ and for
  all $j\geq 0$, the Lebesgue measure of the set $\APS{r}{j}{L}$
  satisfies the following bound:
  \begin{equation}
    \label{eq:AP_meas}
    \meas^{\adim+1} ( \APS{r}{j}{L} ) \leq 
    \frac{\injconst^{\adim-n}}{ 2^{j(\adim-n)}} L^{  - \sum_{i=n+1}^\adim \rho_i}\,.
  \end{equation}
  We are now ready to estimate the integral of the function $h_{r,L}$;
  in fact, by formula~\eqref{eq:h_function_2}, it follows that
  \begin{equation*}
    \begin{aligned}
      \int_{M} h_{r,L}(x) \D x \leq 1 &+ \sum_{i=1}^{J^r_L}
      2^{j(\adim-n)} \meas^{\adim+1} ( \APS{r}{j }{L} ) \\ &+ \sum_{j>
        J^r_L} \frac{2^n\meas^{\adim+1} ( \APS{r}{j }{L} ) } {(
        \deltaone{r}{L})^n } \,;
    \end{aligned}
  \end{equation*}
  by the estimate in formula~\eqref{eq:AP_meas}, we immediately have
  that
  $$
  \sum_{i=1}^{J^r_L} 2^{j(\adim-n)} \meas^{\adim+1} ( \APS{r}{j }{L} )
  \leq \injconst^{\adim-n} J^r_L L^{ -\sum_{i=n+1}^\adim \rho_i}\,.
  $$
  By the definition of the cut-off in formula~\eqref{eq:cut_off} we
  have the bound
    $$
    \frac{2^{n-(J^r_L+1)(\adim-n)}} {(\deltaone{r}{L})^n} \leq 1\,,
    $$
    and, by an elementary estimate on a geometric sum,
  $$
  \sum_{j> J^r_L} \frac{2^n\meas^{\adim+1} ( \APS{r}{j }{L} ) } {(
    \deltaone{r}{L})^n } \leq \frac{2^{n-(J^r_L+1)(\adim-n)}}
  {(\deltaone{r}{L})^n} \injconst^{\adim-n} L^{ - \sum_{i=n+1}^\adim
    \rho_i} \,,
  $$
  hence the statement follows.
\end{proof}

\subsection{Diophantine estimates}
\label{sec:Dioph_ est}
In this section we state the relevant Diophantine condition for
simultaneous Diophantine approximations in any dimensions. We then
derive bounds on the expected average width under Diophantine
conditions from the results of the previous section. We also prove
that in dimension one our Diophantine condition is equivalent to the
standard Diophantine condition.

\begin{definition}
  \label{def:barAY}
  For any basis $\bar Y :=\{\bar Y_1, \dots, \bar Y_n\} \subset \R^n$,
  let $\bar \injconst:=\bar \injconst (\bar Y)$ be the supremum of all
  constants $\bar \injconst'>0$ such that the map
  \[
  (s_1, \dots ,s_n) \to \exp( \sum_{i=1}^{n} s_i \bar Y_i ) \in
  \T^n\,.
  \]
  is a local embedding (injective) on the domain
$$\{ \bs \in \R^n \vert |s_i|< \bar\injconst'
\text{ for all } i=1, \dots, n\} \,.
$$

For any $\theta \in \R^n$, let $[\theta] \in \T^n$ its projection onto
the torus $\T^n := \R^n/\Z^n$ and let
$$
\vert \theta \vert_1= \vert s_1\vert, \dots, \vert \theta
\vert_i=\vert s_i\vert ,\dots, \vert \theta \vert_n =\vert s_n\vert\,,
  $$
  if there is $\bs:=(s_1, \dots, s_n) \in [-\bar\injconst/2,
  \bar\injconst/2]^n$ such that
  \[
  [\theta] = \exp( \sum_{i=1}^{n} s_i \bar Y_i) \in \T^n ;
  \]
  otherwise we set $ \vert \theta \vert_1 = \dots = \vert \theta
  \vert_n= \bar\injconst\,.  $
\end{definition}
  
We introduce below our (simultaneous) Diophantine condition in all
dimensions.
\begin{definition}
  \label{def:newD}
  Let $\sigma:=(\sigma_1, \dots,\sigma_n) \in (0,1)^n$ be such that $
  \sigma_1 + \dots +\sigma_n =1$. For any $\alpha:=(\alpha_1, \dots,
  \alpha_n) \in \R^n$, for any $N\in \N$ and for every $\delta>0$, let
$$
\mathcal R^{(\bar Y,\sigma)}_{\alpha}(N, \delta)= \{ r\in [-N, N] \cap
\Z \vert \vert r\alpha \vert_1 \leq \delta^{\sigma_1}, \dots, \vert
r\alpha \vert_n \leq \delta^{\sigma_n}\}\,.
$$
For every $\nu \geq 1$, let $D_n(\bar Y,\sigma,\nu) \subset
(\R\setminus\Q)^n$ be the subset defined as follows: the vector
$\alpha \in D_n(\bar Y,\sigma,\nu) $ if and only if there exists a
constant $C(\bar Y,\sigma, \alpha)>0$ such that, for all $N\in \N$ and
for all $\delta>0$,
\begin{equation}
  \label{eq:DCn}
  \# \mathcal R^{(\bar Y,\sigma)}_\alpha(N, \delta) \leq  C(\bar Y,\sigma, \alpha)\max\{ N^{1-\frac{1}{\nu}}, N\delta \}\,.
\end{equation}
\end{definition}

Let us prove that in all dimensions the above Diophantine condition
implies a standard simultaneous Diophantine condition.

\begin{lemma}
  \label{lem:StD}
  Let $\alpha \in D_n(\bar Y,\sigma,\nu)$. For all $r\in
  \Z\setminus\{0\}$, we have
$$
\max\{ \vert r \alpha \vert_1, \dots, \vert r \alpha\vert_n \} \geq
\min\{ {\frac{\bar\injconst^2}{4}}, \frac{1}{[1+ C(\bar Y,\sigma,
  \alpha)]^{2\nu} }\}\, \frac{1}{\vert r \vert^\nu}\,.
$$
\end{lemma}
\begin{proof}
  Let $C:= \max\{2/\bar\injconst, [1+C(\bar Y,\sigma,
  \alpha)]^{\nu}\}$. Since $\nu \geq 1$, we have
  \begin{equation}
    \label{eq:C_lb}
    C \geq 1+C(\bar Y,\sigma, \alpha)^\nu  \geq 1+ C(\bar Y,\sigma, \alpha) >1\,.
  \end{equation}
  Let us assume by contradiction that there exists $r\in
  \Z\setminus\{0\}$ such that
$$
\max\{ \vert r \alpha \vert_1, \dots, \vert r \alpha \vert_n\} <
\frac{1}{C^2 \vert r\vert^\nu}\,,
$$
For all $k\in \{\vert r\vert, 2\vert r\vert, \dots, [C \vert
r\vert^{\nu-1}] \times \vert r\vert\}$, we have $\frac{k}{C^2 \vert
  r\vert^\nu}\le \bar\injconst/2$. It follows that
$$
\vert k \alpha \vert_i < \frac{1}{C \vert r\vert} \leq (\frac{1}{C
  \vert r\vert})^{\sigma_i} \,, \quad \text{ \rm for all }i\in \{1,
\dots, n\}\,.
$$
hence by the definitions (in particular, by the estimate in formula
\eqref{eq:DCn})
\begin{equation*}
  \begin{aligned}
    {[}C \vert r\vert^{\nu-1} {]} &\leq \# {\mathcal R}^{(\bar
      Y,\sigma)}_\alpha([ C \vert r\vert^{\nu-1}] \times
    \vert r\vert,  \frac{1}{C\vert r\vert} ) \\
    &\leq C(\bar Y,\sigma, \alpha)\max \{ [ C \vert
    r\vert^{\nu-1}]^{1-\frac{1}{\nu}} \times \vert r\vert
    ^{1-\frac{1}{\nu}}, [ C \vert r\vert^{\nu-1}]/C\} \,.
  \end{aligned}
\end{equation*}
Since $C > C(\bar Y,\sigma, \alpha)$ by formula~\eqref{eq:C_lb}, from
the above inequality we derive
\[
[ C \vert r\vert^{\nu-1}] \le C(\bar Y,\sigma, \alpha)^\nu\times \vert
r\vert ^{\nu -1}.
\]
which, by taking into account that $\vert r\vert^{\nu-1} \geq 1$,
implies that $C< 1+C(\bar Y,\sigma, \alpha)^\nu$, in contradiction
with the inequality in formula~\eqref{eq:C_lb}.
\end{proof}

We prove below that the Diophantine condition introduced above in
Definition~\ref{def:newD} follows from a standard simultaneous
Diophantine condition (of different exponent). This results implies
that for any $\nu>1$ our condition holds for a full measure set of
vectors. In dimension one our condition coincides with the classical
Diophantine condition of the same exponent.  The proof in dimension
one is an exercise based on continued fractions.  We owe the proof in
the general case, which we explain below, to a a personal
communication of N.~Chevallier.

\begin{lemma}
  \label{lem:Dlemma}
  Let $\{(q_i, p_i)\} \subset \N \times \Z^n$ denote the sequence of
  best approximation vectors of a vector $\alpha=(\alpha_1, \dots,
  \alpha_n)\in \R^n\setminus \Q^n$ with respect to the sup norm $\Vert
  \cdot \Vert$ (or to any other norm).  For all $i\in \N$, we adopt
  the notation
$$
\epsilon_i:= q_i \alpha - p_i \in \R^n \quad \text{ and } \quad d_i :=
d(q_i \alpha, \Z^n) = \Vert q_i\alpha - p_i\Vert >0\,.
$$
For any vector $\sigma=(\sigma_1, \dots,\sigma_n)\in (0,1)^n$ such
that $\sigma_1 +\dots+\sigma_n=1$, let
$$
m(\sigma) := \min\{\sigma_1, \dots,\sigma_n\} \quad \text{ and } \quad
M(\sigma):= \max\{\sigma_1, \dots,\sigma_n\} \,.
$$
The vector $\alpha \in D_n(\bar Y, \sigma, \nu)$ for any basis $\bar Y
\subset \R^n$ and for any $\nu \geq 1$ under the assumption that there
exists a constant $C_\alpha>0$ such that, for all $i\in \N$,
$$
\begin{aligned}
  (a) \,\, &q_{i+1} \leq C_\alpha q_i^{\nu}\,, \quad (b)\,\,
  q_{i+1}^{M(\sigma)/\nu} \leq C_\alpha
  q_i d_{i-1} ^{n-1}  \,,  \\
  &(c) \,\, q_{i+1}^{1/\nu} \leq C_\alpha q_i d_{i-1}
  ^{(n-2)(1-\frac{m(\sigma)}{M(\sigma)}) } \,.
\end{aligned}
$$
\end{lemma}
\begin{proof}
  Let $N\in \N$ be a positive integer such that $q_i \leq N <
  q_{i+1}$. For any $r \in [-N, N] \cap \Z$, by the Euclidean
  algorithm, there exist $a \in \Z$ and $0\leq b < q_i$ such that $ r=
  a q_i+ b$.  It follows that $r\alpha = b \alpha + a \epsilon_i$
  modulo $\Z^n$, hence by definition we have that
  \begin{equation}
    \label{eq:D1}
    \vert r\alpha \vert_m = \vert b \alpha + a \epsilon_i \vert_m \,, \quad \text{ for all } m\in \{1, \dots,n\}\,.
  \end{equation}
  For any $\theta \in \T^n$ and $\delta\in (0,1)$, let $N^{(\sigma)}_i
  (\theta, \delta)$ be defined as follows:
$$
N^{(\sigma)}_i (\theta, \delta):= \# \{ b \in [0, q_i-1] \cap \N \vert
\vert b \alpha + \theta \vert_1 \leq \delta^{\sigma_1}\, \dots, \vert
b \alpha + \theta \vert_n \leq \delta^{\sigma_n}\} \,.
$$
By formula~\eqref{eq:D1}, it follows that
\begin{equation}
  \label{eq:D2}
  \# {\mathcal R}^{(\bar Y,\sigma)}_{\alpha}(N, \delta)  \leq 
  \sum_{a=0}^{[N/q_i]}   N^{(\sigma)}_i  (a\epsilon_i, \delta) \,.
\end{equation}
We are therefore led to estimate the integers $N^{(\sigma)}_i (\theta,
\delta)$ for any point $\theta \in \T^n$.

For all $i\in \N$, let $\lambda_{i,1} \leq \dots \leq \lambda_{i,n}$
denote the minima of the lattice
$$
\Lambda_i := \Z \frac{p_i}{q_i} + \Z^n\,.
$$ 
Let us remark that by the definition of the best approximation vectors
it follows that the first minimum $\lambda_{i,1}$ satisfies the
estimate
$$
\lambda_{i,1} \leq 2 d_{i-1}\,.
$$
Let $(\tau _1, \dots, \tau _n)\in (0,1)^n$ be a permutation of
$\{\sigma_1, \dots, \sigma_n\}$ such that
$$
\tau_{1} \leq \dots \leq \tau_{n}\,.
$$
There exists a constant $C_n(\bar Y)>0$ such that, for all $\theta\in
\T^n$, we have
$$
\begin{aligned}
  \# \{ z \in \Lambda_i \vert &\vert z+ \theta \vert_1 \leq
  \delta^{\sigma_1}, \dots, \vert z+ \theta \vert_n \leq
  \delta^{\sigma_n} \} \\ &\leq C_n(\bar Y) \left( 1 +
    \frac{\delta^{\tau_1}}{\lambda_{i,1}} + \frac{\delta^{\tau_1+
        \tau_2}}{\lambda_{i,1} \lambda_{i,2}} \dots +
    \frac{\delta}{\lambda_{i,1} \cdots \lambda_{i,n} } \right) \\
  &\leq C_n(\bar Y) n \{ 1+ \max [ \frac{\delta^{\tau_1} }{d_{i-1}},
  \dots, \frac{\delta^{\tau_1 + \dots+\tau_{n-1} } }{d^{n-1}_{i-1}},
  \frac{\delta}{ \det \Lambda_n} ] \} \,.
\end{aligned}
$$
By taking into account that $d(b\alpha, bp_i/q_i) \leq d_i \leq
d_{i-1}$, for all $b\in [0, q_i-1] \cap \N$, it follows from the above
formula that there exists a constant $C'_n(\bar Y)>0$ such that, for
any $\theta \in \T^n$ and $\delta\in (0,1)$ we have
$$
N^{(\sigma)}_i (\theta, \delta) \leq C'_n(\bar Y) \{ 1 + \max [
\frac{\delta^{\tau_1} }{d_{i-1}}, \dots, \frac{\delta^{\tau_1 +
    \dots+\tau_{n-1} } }{d^{n-1}_{i-1}}, \delta q_i]\} \,,
$$
hence by formula~\eqref{eq:D2}, whenever $q_i \leq N < q_{i+1}$ and
for all $\delta\in (0,1)$ we have
$$
\# {\mathcal R}^{(\bar Y,\sigma)}_{\alpha}(N, \delta) \leq C'_n(\bar
Y) \frac{N}{q_i} \{ 1 + \max [ \frac{\delta^{\tau_1} }{d_{i-1}},
\dots, \frac{\delta^{\tau_1 + \dots+\tau_{n-1} } }{d^{n-1}_{i-1}},
\delta q_i] \} \,.
$$
We distinguish two cases:
$$
\begin{aligned}
  &(1) \,\, \max [ \frac{\delta^{\tau_1} }{d_{i-1}}, \dots,
  \frac{\delta^{\tau_1 + \dots+\tau_{n-1} } }{d^{n-1}_{i-1}}] \, \leq
  \, \delta q_i \, ; \\ &(2) \,\, \max [ \frac{\delta^{\tau_1}
  }{d_{i-1}}, \dots, \frac{\delta^{\tau_1 + \dots+\tau_{n-1} }
  }{d^{n-1}_{i-1}}] \, > \, \delta q_i\,.
\end{aligned}
$$
In case $(1)$, by taking into account that by hypothesis $
C_\alpha^{1/\nu} q_i \geq q_{i+1}^{1/\nu} \geq N^{1/\nu}$, we derive
the following upper bound:
\begin{equation}
  \label{eq:D3}
  \# {\mathcal R}^{(\bar Y,\sigma)}_{\alpha}(N, \delta)  \leq 
  C'_n(\bar Y) \{\frac{N}{q_i} + N\delta\} \leq   C'_n(\bar Y) \{ C_\alpha^{1/\nu}  N^{1-1/\nu}+ N \delta\}\,.
\end{equation}
In case $(2)$ we distinguish two sub-cases:
$$
\begin{aligned}
  &(2a) \,\, \max [ \frac{\delta^{\tau_1} }{d_{i-1}}, \dots,
  \frac{\delta^{\tau_1 + \dots+\tau_{n-1} } }{d^{n-1}_{i-1}}] =
  \frac{\delta^{\tau_1 + \dots+\tau_{n-1} } }{d^{n-1}_{i-1}}\,, \\
  &(2b) \,\, \max [ \frac{\delta^{\tau_1} }{d_{i-1}}, \dots,
  \frac{\delta^{\tau_1 + \dots+\tau_{n-1} } }{d^{n-1}_{i-1}}] \not =
  \frac{\delta^{\tau_1 + \dots+\tau_{n-1} } }{d^{n-1}_{i-1}}\,.
\end{aligned}
$$
In case $(2a)$ we have
$$
\delta q_i < \frac{\delta^{\tau_1 + \dots+\tau_{n-1} }
}{d^{n-1}_{i-1}}= \frac{\delta^{1-M(\sigma) } }{d^{n-1}_{i-1}} \,,
$$
so that by our assumption $(b)$ we can derive the following upper
bound:
$$
\frac{\delta^{\tau_1 + \dots+\tau_{n-1} } }{q_i d^{n-1}_{i-1} } \leq
\left( \frac{1}{q_i d^{n-1}_{i-1}}\right)^{\frac{1}{M(\sigma)}} \leq
C^{1/\nu}_\alpha q_{i+1}^{-1/\nu} \leq C^{1/\nu}_\alpha N^{-1/\nu}\,.
$$
In case $(2b)$, let $j<n-1$ be such that
$$
\max [ \frac{\delta^{\tau_1} }{d_{i-1}}, \dots, \frac{\delta^{\tau_1 +
    \dots+\tau_{n-1} } }{d^{n-1}_{i-1}}] = \frac{\delta^{\tau_1 +
    \dots+\tau_{j} } }{d^{j}_{i-1}}\,.
$$
Since the above condition immediately implies that
$\delta^{\tau_{j+1}} \leq d_{i-i}$, it follows from our assumption
$(c)$ that the following upper bound holds:
$$
\frac{\delta^{\tau_1 + \dots+\tau_{j} } }{q_i d^{j}_{i-1} } \leq
\frac{1}{q_i} \left
  (\frac{1}{d^{n-2}_{i-1}}\right)^{1-\frac{m(\sigma)}{M(\sigma)}} \leq
C_\alpha^{1/\nu} q_{i+1}^{-1/\nu} \leq C_\alpha^{1/\nu} N^{-1/\nu} \,.
$$
We have therefore proved that under our assumptions the upper bound in
formula~\eqref{eq:D3} holds also in case $(2)$, hence $\alpha \in
D_n(\bar Y, \sigma, \nu).$
\end{proof}

Let us recall the classical definition of a simultaneously Diophantine
vector.
 
\begin{definition}
  \label{def:oldD}
  A vector $\alpha \in \R^n\setminus \Q^n$ is simultaneously
  Diophantine of exponent $\nu \geq 1$ if there exists a constant
  $c(\alpha)>0$ such that, for all $r\in \N\setminus \{0\}$,
  \[
  \Vert r \alpha\Vert _{\Z^n} \geq \frac{c(\alpha)}{r^{\nu/n}}\,.
  \]
  Let $DC_{n,\nu} \subset \R^n\setminus \Q^n$ denote the set of all
  simultaneously Diophantine vectors of exponent $\nu\geq 1$.
\end{definition}

\begin{lemma}
  \label{lem:oldnewD}
  For all bases $\bar Y\subset \R^n$, for all $\sigma=(\sigma_1,
  \dots, \sigma_n)\in (0,1)^n$ such that $\sigma_1+ \dots +
  \sigma_n=1$ and for all $\nu\geq 1$, the inclusion
  \begin{equation}
    \label{eq:Dincl}
    DC_{n,\mu} \subset D_n(\bar Y, \sigma, \nu) \,.
  \end{equation}
  holds under the assumption that
  \begin{equation}
    \label{eq:Dcond}
    \mu \leq \min \{ \nu, [\frac{M(\sigma)}{\nu} + 1-\frac{1}{n}]^{-1}, [\frac{1}{\nu} + (1-\frac{2}{n})(1-\frac{m(\sigma)}{M(\sigma)})]^{-1} \}\,.
  \end{equation}
  In particular, the set $D_n(\bar Y, \sigma, \nu)$ has full Lebesgue
  measure if
  \begin{equation}
    \label{eq:Dfull}
    1/\nu < \min\{ [M(\sigma)n]^{-1},  1-  (1-\frac{2}{n})(1-\frac{m(\sigma)}{M(\sigma)}) \}\,.
  \end{equation}
\end{lemma}
\begin{proof}
  The inclusion in formula~\eqref{eq:Dincl} under the conditions in
  formula~\eqref{eq:Dcond} follows from Lemma~\ref{lem:Dlemma}. In
  fact, by elementary calculations it is possible to prove, taking
  into account that, for all $i\in \N$, we always have
$$
d_i \leq d_{i-1} \quad \text{ and } \quad q_{i+1} d_i^n \leq 1 \, ,
$$
and under the assumption that $\alpha \in DC_{n,\mu}$ we also have
that, for all $i\in \N$,
$$
d_i = \Vert q_i \alpha -p_i \Vert = \Vert q_i \alpha \Vert_{\Z^n} \geq
c(\alpha) q_i^{-\mu/n}\,,
$$
that the following holds.  Condition $(a)$ of Lemma~\ref{lem:Dlemma}
holds if $\mu \leq \nu$, condition $(b)$ of Lemma~\ref{lem:Dlemma}
holds if
$$
\mu \leq [\frac{M(\sigma)}{\nu} + 1-\frac{1}{n}]^{-1} \,,
$$
and, finally, condition $(c)$ of Lemma~\ref{lem:Dlemma} holds if
$$
\mu \leq [\frac{1}{\nu} +
(1-\frac{2}{n})(1-\frac{m(\sigma)}{M(\sigma)})]^{-1}\,.
$$
The first part of the proof is therefore completed.

The proof of the second part is based on the classical elementary fact
that the set $DC_{n,\mu}\subset \R^n$ has full Lebesgue measure for
all $\mu>1$.  It follows that the set $D_n(\bar Y, \sigma, \nu)$ has
full Lebesgue measure whenever the minimum on the right hand side of
formula~\eqref{eq:Dcond} is strictly larger than $1$.  By an
elementary calculation one can prove that this condition is verified
if the inequality in formula~\eqref{eq:Dfull} holds. The proof of the
second part of the statement is therefore completed as well.
\end{proof}

In dimension one since the vector space $\R$ has a unique basis up to
scaling, the Diophantine condition introduced in
Definition~\ref{def:newD} above is independent of the choice of the
basis $\bar Y \subset \R$ and of the probability vector $\sigma \in
\R$. We therefore omit the pair $(\bar Y, \sigma)$ from the notations
introduced above. The following result is an immediate consequence of
Lemma~\ref{lem:StD} and of Lemma~\ref{lem:oldnewD}.
\begin{lemma}
  \label{lem:oldnewDone}
  For all $\nu \geq 1$ the following identity holds:
$$
DC_{1,\nu} = D_1(\nu) \,.
 $$
\end{lemma}

We can finally proceed to derive our main bound on the expected width
under the above Diophantine condition.

\smallskip Let $\mathcal F_\alpha:= (X_\alpha, Y)$ be a strongly
adapted basis and let $\bar Y =(\bar Y_1, \dots, \bar Y_n)\in \R^n$
denote the projection of the basis $Y=(Y_1, \dots, Y_n)$ of the
Abelian ideal $\mathfrak a \subset \mathfrak g$ onto the Abelianised
Lie algebra $\bar {\mathfrak g} := \mathfrak g/ [ \mathfrak g,
\mathfrak g] \approx \R^n$.

For any $\rho:=(\rho_1, \dots, \rho_\adim)\in [0,1)^\adim$, let us
adopt the notation
$$
\rhobar:=(\rho_1, \dots, \rho_n), \qquad |\rhobar| := \rho_1+ \dots +
\rho_n\,.
$$

Let $\alpha_1:= (\alpha_1^{(1)}, \dots, \alpha_1^{(n)}) \in D_n (\bar
Y, {\rhobar}/{\vert \rhobar\vert}, \nu)$.  For brevity we also adopt
the following notation. Let $C(\bar Y, \rhobar/\vert\rhobar\vert,
\alpha_1)$ denote the constant in the Diophantine condition introduced
in Definition~\ref{def:newD} and let
\begin{equation}
  \label{eq:not_DC}
  C(\alpha_1) := 1+ C(\bar Y, \rhobar/\vert\rhobar\vert, \alpha_1)\,.
\end{equation}
We prove below an upper bound on the cut-off function introduced in
formula~\eqref{eq:cut_off}.  Let us recall that by definition, for all
$L\geq 1$ and for all $r \in \Z\setminus\{0\}$, we have
$$
J^r_L= \max \{ j\in \N \vert 2^{j(\adim-n)} \leq
(\frac{2}{\deltaone{r}{L}})^n \}\,.
$$ 
The following logarithmic upper bound holds.  Let
$\injconst=\injconst(Y)$ and $\bar\injconst= \bar\injconst(\bar Y)$ be
the positive constants introduced in Definition~\ref{def:AY} and
Definition~\ref{def:barAY}. We observe that by definition
$\injconst(Y) \leq \bar\injconst(\bar Y)$ since the basis $\bar Y
\subset \R^n$ is the projection of the basis $Y\subset \mathfrak a$
and the canonical projection commutes with the exponential maps.

\begin{lemma}
  \label{lemma:cutoff}
  For every $\rho\in [0,1)^\adim$, for every $\nu \leq 1/\vert
  \rhobar\vert$ and for every $\alpha_1:= (\alpha_1^{(1)}, \dots,
  \alpha_1^{(n)}) \in D_n (\bar Y, {\rhobar}/{\vert \rhobar\vert},
  \nu)$ there exists a constant $K:=K(\adim, n,\nu)>0$ such that, for
  all $L\geq 1$ and for all $r \in \Z\setminus\{0\}$, the following
  bound holds:
$$
J^r_L \leq K \{ 1+ \log^+[ \injconst(Y)^{-1}] + \log C(\alpha_1)\}
(1+\log\vert r\vert) \,.
$$
\end{lemma}
\begin{proof}
  By Lemma~\ref{lem:StD} and by the definition of $\deltaone{r}{L}$ in
  formula~\eqref{eq:delta_one}, it follows that, since $\alpha_1\in
  D_n(\bar Y,{\rhobar}/{|\rhobar|}, \nu)$, for all $T>0$, $L\geq 1$
  and for all $r \in\Z\setminus\{0\}$, we have
  $$
  \deltaone{r}{L} = \max_{1\leq i \leq n} \min\{\injconst, \vert r
  \alpha_1 \vert_i\} \geq \min\{\injconst, \frac{\bar\injconst^2}{4} ,
  \frac{1}{[1+C(\alpha_1)]^{2\nu}} \} \, \frac{1}{\vert r\vert^{\nu}}
  \,.
  $$
  It follows by the above bound and by the definition of the cut-off
  that
  $$
  J^r_L \leq \frac{n }{\adim-n} (3\log 2 + 3 \log^+ (1/\injconst) +
  2\nu \log [1+C(\alpha_1)]+ \nu \log \vert r \vert) \,,
 $$
 hence the statement follows.
\end{proof}

Let $\mathcal F_\alpha:= (X_\alpha,Y)$ be a strongly adapted basis and
let $\rho \in [0,1)^\adim$ a vector of scaling exponents. Assume that
there exists $\nu \leq 1/\vert \rhobar\vert$ such that $\alpha_1 \in
DC_n(\bar Y, \rhobar/\vert\rhobar\vert,\nu)$.  For brevity, we
introduce the following notation:
\begin{equation}
  \label{eq:H_const}
  {\Cal H}(Y,\rho,\alpha) :=  1 +  \injconst(Y)^{a-n} C(\alpha_1)
  \{1 + \log^+ [\injconst(Y)^{-1}] + \log C(\alpha_1)\}\,.
\end{equation}

The following bound holds.
\begin{theorem}
  \label{thm:main_exp_width}
  For every $\rho\in [0,1)^\adim$, for every $\nu \leq 1/|\rhobar|$
  and for every $\alpha_1:= (\alpha_1^{(1)}, \dots, \alpha_1^{(n)})
  \in D_n (\bar Y,{\rhobar}/{\vert \rhobar\vert}, \nu)$ there exists a
  constant $K':= K'(\adim,n,\nu)>0$ such that, for all $T>0$ and for
  all $L\geq 1$, the following bound holds:
  $$
  \vert \int_{M} H^T_{L}(x) \,\D x \vert \, \leq K' \,{\Cal
    H}(Y,\rho,\alpha) (1+ T) (1+\log^+T +\log L) \,L^{
    (1-\sum_{i=1}^\adim \rho_i)}\,.
  $$
\end{theorem}
\begin{proof} By the definition of the function $H^T_{L}$ in
  formula~\eqref{eq:H_function}, the statement follows from
  Lemma~\ref{lem:exp_width} and Lemma~\ref{lemma:cutoff}. In fact, for
  all $r\in \Z\setminus \{0\}$ and all $j\geq 0$, by
  definition~\eqref{eq:AP_r} the set $\APS{r}{j}{L}$ is non-empty only
  if $\deltaone{r}{L} <\injconst/2$. Since $\nu \leq 1/\vert
  \rhobar\vert$, it follows directly from the definition of the
  Diophantine class $D_n(\bar Y, {\rhobar}/{\vert \rhobar\vert},
  \nu)$, that for all $T$, $L\geq 1$, we have
$$
\# \{ r\in [-TL,TL ] \cap \Z\setminus\{0\} \vert \APS{r}{j}{L} \not =
\emptyset \} \leq C(\bar Y,\rhobar/\vert\rhobar\vert, \alpha_1) (1+T)
\,L^{1-\vert \rhobar\vert} \,,
 $$
 hence by Lemma~\ref{lem:exp_width} and Lemma ~\ref{lemma:cutoff} the
 statement is proved.
\end{proof}

\subsection{Width estimates along orbit segments}
\label{sec:62}
Let $\rho:=(\rho_1, \dots, \rho_\adim) \in [0,1)^{\adim}$ and let
$\mathcal F_\alpha=(X_\alpha, Y)$ be a normalised strongly adapted
basis, and recall the notation~\eqref{eq:width:1} for the rescaled
bases~$\mathcal F _\alpha^{(L)}$.
\begin{definition}
  \label{def:width:1}
  For any increasing sequence $(L_i)$ of positive real numbers, let
  $N_i:=[\log L_i/\log 2]$ and $L_{j,i}:=L_i^{j/N_i}$, for all $j = 0,
  \dots, N_i$.  Let $\dev >0$ and $\wpar>0$.
  
  We say that a point $x\in M$ is a \emph{$(\wpar, (L_i),\dev)$-good
    point for the basis $\mathcal F_\alpha$} if having set $y_i=
  \phi_{X_\alpha}^{L_i}(x)$, for all $i\in \N$ and for all $0\le j \le
  N_i$, we have
  \[
  w_{\mathcal F _\alpha^{(L_{j,i})}}(x,1) \ge \wpar / L_{i}^{\dev},
  \qquad w_{\mathcal F _\alpha^{(L_{j,i})}(y_i,1)} \ge \wpar/
  L_{i}^{\dev}\,.
  \]
\end{definition}
\begin{remark}
  \label{rem:width:1}
  By its definition the set of $(\wpar, (L_i),\zeta)$-good points for
   the basis $\mathcal F_\alpha$ is saturated by the orbit of the
  action of the centre $Z(G)$ of the quasi-Abelian 
  nilpotent Lie group $G$ on $M$. Moreover, if the vector
  $\rho\in [0, 1)^\adim$ of scaling exponents vanishes on all vectors
  of the basis $\mathcal F _\alpha$ which belong to the center
  $Z(\mathfrak g)$ of the Lie algebra $\mathfrak g$ of $G$, then a point 
  $x\in M$ is $(\wpar, (L_i),\zeta)$-good
 for the basis $\mathcal F_\alpha$  if and only if its projection $\bar x \in M/Z(G)$ 
 is  $(\wpar, (L_i),\zeta)$-good for the projection $\overline {\mathcal F}_\alpha$ of 
 the basis $\mathcal F_\alpha$ onto the quotient $\mathfrak g/Z(\mathfrak g)$, that is, onto 
 the Lie algebra of $G/Z(G)$. If follows that  in this case, the set of $(\wpar, (L_i),\zeta)$-good 
 points for the basis $\mathcal F_\alpha$ is not only  invariant under the action of $Z(G)$ 
 on $M$ but its projection onto $M/Z(G)$ is invariant under the action of the center 
 $Z(G/Z(G))$ of $G/Z(G)$ onto $M/Z(G)$.
  \end{remark}
\begin{lemma}
  \label{lem:62}
  Let $\dev >0$ be fixed and let $(L_i)$ be an increasing sequence of
  positive real numbers satisfying the condition
  \begin{equation}
    \label{eq:92}
    \Sigma\bigl((L_i),\dev\bigr):=  \sum _{i\in \N} (\log L_i)^2 L_i^{-\dev} < +\infty.
  \end{equation}
  Let $\rho\in [0,1)\adim$, with $\sum \rho_i =1$, $\nu \leq
  1/|\rhobar|$ and let $\alpha_1:= (\alpha_1^{(1)}, \dots,
  \alpha_1^{(n)}) \in D_n ({\rhobar}/|\rhobar|, \nu)$.  Then the
  Lebesgue measure of the complement of the set ${\Cal
    G}\big(\wpar,(L_i),\dev\big)$ of $\big(\wpar,(L_i),\dev\big)$-good
  points is bounded above as follows: there exists a constant
  $K:=K(\adim,n,\nu)>0$ such that
  $$
  \operatorname{meas}\bigl(M\setminus {\Cal G}(\wpar,(L_i),\dev)
  \bigr) \leq K \,\Sigma\big((L_i),\dev\big) [1/\injconst(Y)]^{\adim}
  {\Cal H}(Y,\rho,\alpha)\, \wpar \,.
  $$
\end{lemma}
\begin{proof}
  For all $i\in \N$ and for all $j=0, \dots, N_i$, let
  \[
  {\Cal S}_{j,i}= \left\{ z \in M\,\,:\,\, w_{\mathcal F
      _\alpha^{(L_{j,i})}}(z,1) < L_{i}^{\dev}/\wpar \right\}.
  \]
  By definition we have
  \begin{equation}
    \label{eq:bad_points}
    M\setminus {\Cal G}\big(\wpar,(L_i),\dev\big)= \bigcup_{i\in \N} \bigcup_{j=0}^{N_i} \big({\Cal S}_{j,i} \cup
    \phi_{X_\alpha}^{-L_i} {\Cal S}_{j,i} \big)\,.
  \end{equation}
  By Lemma~\ref{lemma:width_bound} for all $z\in{\Cal S}_{j,i}$ we
  have
  \[ (\injconst/2)^\adim L_{i}^{\dev}/\wpar < \int_{0}^1 H^1_{L_{j,i}}
  \circ \phi^{\tau}_{X_\alpha^{(L_{j,i})}} (z)\, \D{}\tau= \frac 1
  {L_{j,i}}\int_{0}^{L_{j,i}} H^1_{L_{j,i}} \circ
  \phi^{\tau}_{X_\alpha} (z)\,\D{}\tau.
  \]
  It follows that
  \[
  {\Cal S}_{j,i} \subset {\Cal S}(j,i) :=\left\{ z \in M\,\,:\,\,
    \sup_{J>0} \frac {1}{J}\int_{0}^{J} H^1_{n_{j,i}}\circ
    \phi^{\tau}_{X_\alpha} (z)\, \D\tau > (\injconst/2)^\adim
    L_{i}^{\dev}/\wpar \right\}.
  \]
  By the maximal ergodic theorem, the Lebesgue measure
  $\operatorname{meas}[{\Cal S}(j,i)]$ of the set $ {\Cal S}(j,i) $
  satisfies the inequality
  \[
  \operatorname{meas}[{\Cal S}(j,i)]\, \leq (2/\injconst)^\adim (\wpar
  / L_{i}^{\dev})\int_{M} H^1_{L_{j,i}}(z) \, \D z.
  \]
  For brevity, let $\Cal H:={\Cal H}(Y,\rho,\nu)$ denote the constant
  defined in formula~\eqref{eq:H_const}.  By
  Theorem~\ref{thm:main_exp_width}, since by hypothesis $\nu \leq
  1/\vert \rhobar\vert$ and $\alpha_1\in D_n ({\rhobar}/{\vert
    \rhobar\vert}, \nu)$, there exists a constant
  $K':=K'(\adim,n,\nu)>0$ such that the following bound holds:
  \[
  \Big\vert \int_{M} H^1_{L_{j,i}}(x) \,\D x \Big\vert \, \leq K'
  {\Cal H} \, (1+\log {L_{j,i}}).
  \]
  The definition of the $N_i$ implies
  \begin{equation}
    \label{eq:93}
    N_i \le \log L_i/\log 2 < N_i+1.
  \end{equation}
  Hence by the definition of $L_{j,i}$, we have $\log L_{j,i} <2j \log
  2 $.  Thus, for some constant $K'':=K''(\adim,n,\nu)>0$, we have
  \begin{equation*}
    \label{eq:94}
    \operatorname{meas}( {\Cal S}_{j,i})\, \leq \operatorname{meas}[ {\Cal S}(j,i)]\,
    \leq K''  (2/\injconst)^\adim {\Cal H} \,\wpar 
    (1+j) L_{i}^{-\dev}\,.
  \end{equation*} 
  From this, again by~\eqref{eq:93}, it follows that, for some
  constant $K''':=K'''(\adim,n,\nu)>0$,
  \[
  \operatorname{meas}\Big( \bigcup_{j=0}^{N_i} \big({\Cal S}_{j,i}
  \cup \phi_{X_\alpha}^{-L_i} {\Cal S}_{j,i} \big)\Big) \le K'''
  (2/\injconst)^\adim {\Cal H} \,\wpar (\log L_i)^2 L_i^{-\dev}.
  \]
  By sub-additivity of the Lebesgue measure, we derive the bound
  \[
  \operatorname{meas}\Big( \bigcup_{i\in \N} \bigcup_{j=0}^{N_i}
  \big({\Cal S}_{j,i} \cup \phi_{X_\alpha}^{-L_i} {\Cal S}_{j,i}
  \big)\Big) \le K''' \,\Sigma\big((L_i),\dev\big) (2/\injconst)^\adim
  \,{\Cal H} \, \wpar.
  \]
  By formula~\eqref{eq:bad_points} the above estimate concludes the
  proof.
\end{proof}

\def\ConstDepx{K}
\section{Bounds on ergodic averages}
\label{sec:last_step}

Let $(\xi, \dots, {\tilde \eta}_i^{(m)},\dots)$, with $(m,i)\in J$, be the basis
\eqref{eq:LieGen_1} defining the lattice $\Gamma$, and let $A$ be the
analytic subgroup of\/ $G$ of Lie algebra $\mathfrak a$. We denote by
$(\xi, \dots, {\eta}_i^{(m)},\dots)$ the Jordan basis defined by
\eqref{eq:LieGen_3}. As usual, for $\alpha:=(\alpha_i^{(m)}) \in
\R^{J}$ let $X_\alpha\in \mathfrak g\setminus\mathfrak a$ be the
vector field, introduced in formula~\eqref{eq:X_alpha}, given by the
formula
\[
X_\alpha= \log\left[ x^{-1} \exp \left(\sum_{(m,i)\in J}
    \alpha_i^{(m)} {\tilde \eta}_i^{(m)}\right)\right]\,.
\]
The field $X_{\alpha}$ generates the flow
$\{\phi_\alpha^t\}:=\{\phi^t_{X_\alpha}\}$ on the
nilmanifold~$M=\Gamma\backslash G$.

We denote by $\mathcal F_{\alpha,\eta}$ the Jordan basis \[\mathcal
F_{\alpha,\eta}:=(X_\alpha, \dots, {\eta}_i^{(m)},\dots).\]

The Hilbert space $L^2(M)$ splits as a direct sum of irreducible
sub-repre\-senta\-tions of $G$. The irreducible unitary
representations occurring in the decomposition of $L_0(M)$ are
unitarily equivalent to the representations $\Ind_A^G(\Lambda)$,
obtained by inducing from~$A$ to~$G$ a character $\chi=\exp \imath
\Lambda$ whose coordinates $\Lambda(\tilde \eta_i^{(m)})$ with respect
to the basis $( {\tilde \eta}^{(m)}_i)$ are integer multiples
of~$2\pi$ not all zero.

\begin{definition}
  We say that a linear form $\Lambda\in \Cal O$ is \emph{integral} if the
  coefficients $\Lambda(\tilde \eta_i^{(m)})$,
  $(m,i)\in J$, are integer multiples of~$2\pi$. We denote by
  $\widehat M$ the set of co-adjoint orbits $\Cal O \subset \mathfrak
  a^*$ of integral linear forms $\Lambda \in \mathfrak a^*$.
\end{definition}

By Lemma~\ref{lemma:reduction}, irreducible unitary representations
induced by the character $\chi=\exp \imath \Lambda$ factor through the
filiform group $G/G_\Lambda$, where $G_\Lambda$ is the normal subgroup
of $G$ with Lie algebra given by the ideal, already introduced in
formula~\eqref{eq:Lambda_ideal0},
\[
\mathfrak I_\Lambda = \bigcap_{i=0}^{\lstep-1} \text{ker} (\Lambda
\circ \ad^i(X_\alpha))
\]

\begin{remark}
  \label{rem:laststep:2}
  For our goals it is not restrictive to assume that
  $\Lambda(\tilde\eta^{(m)}_{i_m}) \neq 0$ for some $m\in
  \{1,\dots,n\}$ with $i_m=\lstep$.  In fact, suppose that
  $H_\Lambda\subset L^2_0(M)$ is a sub-representation unitarily
  equivalent to $\Ind_A^G(\Lambda)$ and that
  $\Lambda(\tilde\eta^{(m)}_{i_m}) =0$ for all $m\in \{1,\dots,n\}$
  with $i_m=\lstep$. By letting $G^{(m)}_{i_m}$ be the subgroup of
  $Z(G)$ generated by $\tilde\eta^{(m)}_{i_m}$, we have that the
  representation $\Ind_A^G(\Lambda)$ factorises through the quotient
  group $G'=G/G^{(m)}_{i_m}$ and occurs as a sub-representation in the
  quasi-Abelian $(\lstep-1)$-step nilmanifold $M':=G'/\Gamma'$, where
  $\Gamma'$ is the lattice~$\Gamma G^{(m)}_{i_m}$ of $G'$.
\end{remark}

Let $\widehat M_0 \subset \widehat M\cap \mathfrak a^*$ the subset of
all co-adjoint orbits of forms $\Lambda \in \mathfrak a^*$ such that
$\Lambda(\tilde\eta^{(m)}_{i_m}) \neq 0$ for some $m\in \{1,\dots,n\}$
with $i_m=\lstep$. By definition, for every $\Cal O \in \widehat M_0$
and every $\Lambda \in \Cal O$, the induced irreducible unitary
representation $\pi^{X_{\alpha}}_\Lambda$ has exactly degree
$\lstep-1$. In fact, for any linear functional $\Lambda \in \mathfrak
a^*$ the degree of the representation $\pi^{X_{\alpha}}_\Lambda$ only
depends on its co-adjoint orbit .

For any $\Cal O \in \widehat M_0$, let $H_{\Cal O}$ denote the primary
subspace of $L^2(M)$ which is a direct sum of sub-representations
equivalent to $\Ind_A^G(\Lambda)$, for any $\Lambda \in \Cal O$; this
space space is well-defined since the unitary representations
$\Ind_A^G(\Lambda)$ are unitarily equivalent for all $\Lambda \in \Cal
O$. For any adapted basis $\mathcal F=(X_\alpha,Y)$ and for all $r\in
\R$, we set
\[ W^r(H_{\Cal O}, \mathcal F):=H_{\Cal O}\cap W^{r}(M, \mathcal
F)\,, \] which is a Hilbert space once endowed with the transversal
Sobolev norm $|\cdot|_{\mathcal F, r}$.


For $\mathcal O\in \widehat M_0$ and $\Lambda\in\mathcal O$, let $m_0$
be any integer such that the vector $\eta_1^{(m_0)}$ has maximal degree
$\lstep$. Then $(X_\alpha, \eta_1^{(m_0)},\dots,\eta_\lstep^{(m_0)})$
is a $\lstep$-step filiform basis projecting to a $\lstep$-step
filiform basis of the Lie algebra~$\mathfrak g /\mathfrak I_\Lambda$.
By Lemma~\ref{lemma:reduction} we can complete the system
$$
(X_\alpha, Y_1^{(1)},\dots, Y_\lstep^{(1)}) :=(X_\alpha,\eta_1^{(m_0)},\dots,\eta_\lstep^{(m_0)})
$$ to a basis $(X_\alpha, Y^{(m)}_i)$ of $\mathfrak g$ so that the elements $Y^{(m)}_i$ with  $m\neq m_0$ span the ideal~$\mathfrak I_\Lambda$.  
\begin{definition}
  \label{def:base_adapt_to_lambda}
From now on given a co-adjoint orbit $\Cal O\in \widehat M_0$
  and a linear functional $\Lambda\in \Cal O$, the symbols \/ $\mathcal
  F_{\alpha,\Lambda}$ and $ (X_\alpha, Y_\Lambda)$ will denote the basis
  \[
  \mathcal F_{\alpha,\Lambda}=(X_\alpha,Y_\Lambda):= (X_\alpha, \dots,
  Y^{(m)}_i, \dots)
\] 
obtained by completion of the system $(X_\alpha,\eta_1^{(m_0)},\dots,\eta_\lstep^{(m_0)})$.
\end{definition}

Clearly the basis $\mathcal F_{\alpha,\Lambda}$ is defined up to an
arbitrary choice of the the integer $m_0$. This ambiguity is
irrelevant for what follows; later on we shall make a more precise
choice. By construction the basis $\mathcal F_{\alpha,\Lambda}$ satisfies the estimates of
Lemma~\ref{lemma:reduction} and it is a normalised, Jordan basis which
is also a generalised filiform basis (in the sense of 
Definition~\ref{def:gen_fil}) for the induced representation~$\pi^{X_\alpha}_\Lambda$.

\begin{remark}
  \label{rem:laststep:1}
  The weights introduced in formulas~\eqref{eq:Lambda_Fnorm_nohat}
  and~\eqref{eq:Lambdaweight} have a simple expression for the bases
  $\mathcal F_{\alpha,\eta}$ and $\mathcal F_{\alpha,\Lambda}$.  In
  fact for any co-adjoint orbit $\Cal O\in \widehat M_0$ and for all $\Lambda \in \Cal O$ we have that:
  \[
  |\Lambda(\mathcal F_{\alpha,\eta})|= \max_{1\le m\le n}\max_{1\le
    j\le i_m}|\Lambda(\eta^{(m)}_j)|
 , \qquad |\Lambda(\mathcal  F_{\alpha,\Lambda})|= \max_{1\le j\le k}|\Lambda(\eta^{(m_0)}_j)|
  \]
  and
  \[
 \text{for } \mathcal F=  \mathcal F_{\alpha,\eta}\text{ or } \mathcal  F_{\alpha,\Lambda} \qquad  \|\Lambda\|_{\mathcal F}= |\Lambda(\mathcal
  F)|\,\left(1+\frac{1}{|\Lambda
      (\eta^{(m_0)}_k)|}\right)\,.
  \]
  From the above formula, it is immediate that for any co-adjoint
  orbit~$\Cal O \in \widehat M_0$, for any $\Lambda \in \Cal O $,  since $\Lambda
      (\eta^{(m_0)}_k)$ is a non zero integer,
   we have
  \begin{equation*}
    | \Lambda(\mathcal F_{\alpha,\eta}) | \le \Vert\Lambda \Vert_{\mathcal F_{\alpha,\eta}} \le  2 \,|
    \Lambda(\mathcal F_{\alpha,\eta}) |\quad\text{and}\quad  | \Lambda(\mathcal F_{\alpha,\Lambda}) | \le \Vert\Lambda \Vert_{\mathcal F_{\alpha,\Lambda}} \le  2 \,|
    \Lambda(\mathcal F_{\alpha,\Lambda}) |.  
  \end{equation*}
\end{remark}

\subsection{Coboundary estimates for rescaled bases}
\label{sec:55}
In this section we prove So\-bo\-lev estimates, with respect to
rescaled bases, for the orthogonal projections of the probability
measures supported on orbit segments of a quasi-Abelian nilflow on the
orthogonal complement of the space of invariant distributions. Our
estimates will be derived in every given irreducible unitary
representation from the estimates on coboundaries proved in
section~\ref{sec:CE}. The rescaled norms will be defined with respect
to a generalised filiform basis depending on the irreducible
representation.

Let $\rho=(\rho^{(m)}_i) \in (\R^+)^J$ be a fixed vector (to be chosen
later) such that
$$
\sum_{(m,i)\in J } \rho^{(m)}_i =1\,.
$$
For all $t \in \R$, let $\mathcal F_{\alpha,
  \Lambda}(t)$ denote the rescaled basis
\begin{equation*}
  \label{eq:rescaledbasisbis} 
  \mathcal F_{\alpha,
  \Lambda}(t):= (X_\alpha(t), Y_\Lambda(t))= A^\rho_t (X_\alpha, Y_\Lambda)= 
  (e^t X_\alpha , \dots , e^{-t\,\rho^{(m)}_i}  Y^{(m)}_i ,\dots)\,.
\end{equation*}
Since the basis $(X_\alpha, Y_\Lambda)$ is generalised filiform for
the induced representation $\pi^{X_\alpha}_\Lambda$, in the optimal
choices of the scaling exponents $(\rho^{(m)}_i)$ we will always have
that
$$ \rho^{(m_0)}_\lstep=0 \qquad \text{ and } \qquad \rho^{(m)}_{i} =0, \quad
\text{ for all }m\not=1 \text{ and all } i\not=1\,.
$$

Let us recall that by Definition~\ref{def:degreeY}, for all $(m,i) \in
J$ the degree $d^{(m)}_{i}$ of the element $Y^{(m)}_i \in \mathfrak a$
with respect to the induced representation $\pi^{X_\alpha}_\Lambda$ is
the degree of the polynomial $\Lambda (\text{Ad} (e^{xX_\alpha})
Y^{(m)}_i)$. Since the basis $(X_\alpha,Y_\Lambda)$ is generalised
filiform for the induced representation $\pi^{X_\alpha}_\Lambda$,
which by construction has maximal degree equal to $\lstep-1$, it
follows that
 $$
 d^{(m)}_i = \begin{cases}     \lstep -i\,,  \quad &\text{ for } m=m_0 \text{ and for all } i= 1, \dots, i_1= \lstep \,; \\
   0\,, \quad &\text{ for } m\not=m_0 \text{ and for all } i= 1, \dots,
   i_m\leq \lstep \,. \end{cases}
$$

In the present case the exponent $\r$ defined in
formula~\eqref{eq:defofr} becomes
\[
\r = \lambda(\rho):= \min_{1\le i <\lstep}\Big\{\frac
{\rho^{(m_0)}_i}{\lstep-i}\Big\};
\]
we also set
\begin{equation*}
  \label{eq:defofr_bis}
  \delta(\rho) := \min_{1\le i < \lstep}\{\rho^{(m_0)}_i-
  \rho^{(m_0)}_{i+1} \}.
\end{equation*}

\begin{lemma}
  \label{lem:laststep:1}
  Let
  \[R^{(m_0)}(\rho):= \sum_{1\le i <\lstep} \rho^{(m_0)}_i.\] We have
  \[
  \delta(\rho)\le \lambda(\rho) \le \frac {2
    R^{(m_0)}(\rho)}{\lstep(\lstep-1)}
  \]
  The above inequalities are both strict unless one
  has \[\rho^{(m_0)}_i= \frac{2
    R^{(m_0)}(\rho)(\lstep-i)}{\lstep(\lstep-1)}, \quad \text{for all
  }\, i=1,\dots,\lstep\,, \] in which case they are both equalities.
\end{lemma}

\begin{lemma}
  \label{lem:laststep:2}
  There exists a constant $C>0$ such that, for all $r\in \R^+$ and for
  any function $f\in W^r(H_{\Cal O}, \mathcal F_{\alpha,\Lambda})$ we have
  \[
  \sum_{(m,i)\in J}\big\vert [X_\alpha(t), Y^{(m)}_i(t)] f
  \big\vert_{r, \mathcal F_{\alpha,\Lambda}(t) } \le C e^{t(1-\delta(\rho))} \big\vert
  f \big\vert_{r+1, \mathcal F_{\alpha,\Lambda}(t) }\,.
  \]
\end{lemma}
\begin{proof}
  For all $(m,i)\in J^-$, we have $[X_\alpha(t),
  Y^{(m)}_i(t)]=e^{t(1-\rho^{(m)}_i+\rho^{(m)}_{i+1})}
  Y^{(m)}_{i+1}(t)$.  Since $\Lambda \in \Cal O$, by construction we
  have $Y^{(m)}_i \in \mathfrak I_\Lambda$, for
  $m\not =1$ and for all $i=1, \dots, i_m$. Since $\mathfrak I_{\Lambda}$ coincides with the kernel of the induced representation
  $\pi^{X_\alpha}_\Lambda$, which is unitarily equivalent to the
  representation given by the action of $G$ on $H_{\Cal O}$, in the
  space $H_{\Cal O}$ we have
$$
[X_\alpha(t),Y^{(m)}_i (t)]f = 0\,, \quad \text{ for all } m\not=1\,
\text{ and for all }\, i=1, \dots, i_m\,.
$$
It follows that
\[
\sum_{(m,i)\in J}\big\vert [X_\alpha(t), Y^{(m)}_i(t)] f \big\vert_{r,
  \mathcal F_{\alpha,\Lambda}(t) } \le e^{t(1-\delta(\rho))} \sum_{i=2}^\lstep\big\vert
Y^{(m_0)}_i (t) f \big\vert_{r, \mathcal F_{\alpha,\Lambda}(t)}\,,
\]
thereby concluding the proof.
\end{proof}

\begin{proposition}\label{prop:514}
  Let $r > (\adim/2+1)(\lstep -1)+1$ and let $\mathcal F_{\alpha,\Lambda}(t)$ and
  $\rho\in (\R^+)^\adim$ be defined as above. For $x\in M$ let
  $\gamma_x$ be the Birkhoff average operator
  \begin{equation}
    \label{eq:laststep:1}
    \gamma_x(f) = \frac 1 L \int_0^L f  (\phi_{X_\alpha}^\tau(x))\,\D \tau 
  \end{equation}
  and consider the decomposition of the restriction of the linear
  functional $\gamma_x$ to $ W^{r}_0(H_{\Cal O}, \mathcal F_{\alpha,\Lambda}(t))$ as an
  orthogonal sum $\gamma_x = D(t)+R(t)$ in $W_0^{-r}(H_{\Cal O},
  \mathcal F_{\alpha,\Lambda}(t))$ of a $X_\alpha$-invariant distribution $D(t)$ and an
  orthogonal complement~$R(t)$.

  There is a constant $C^{(m_0)}_r$ such that for all $g\in
  W^{r}(H_{\Cal O}, \mathcal F_{\alpha,\Lambda}(t)) $ and all $t\ge 0$, having set $y =
  \phi_{X_\alpha}^L(x)$, we have
  \[
  \begin{split}
    |R(t)(g)| \le & C^{(1)}_r\, (1+\Vert \Lambda\Vert_{\mathcal
      F_{\alpha,\Lambda}})^{\frac{\lstep(r +1) -2}{\lstep -1}}\, e^{t(\lambda(\rho) -
      \delta(\rho))} \, L^{-1} \\ & \qquad \qquad \times \left(\frac
      1{w_{\mathcal F_{\alpha,\Lambda}(t)}(x,1)^{1/2}}+ \frac 1{w_{\mathcal
          F_{\alpha,\Lambda}(t)}(y,1)^{1/2}}\right) \vert g\vert_{r, \mathcal F_{\alpha,\Lambda}(t)}.
  \end{split}
  \]

\end{proposition}
\begin{proof}
  Fix $t\ge 0$ and for brevity set $D=D(t)$, $R=R(t)$. Let $g\in
  W^{r}(H_{\Cal O}, \mathcal F_{\alpha,\Lambda}(t)) $. We write $g = {g}_D + {g}_R$,
  where $g_R$ is in the kernel of the $X_\alpha$-invariant
  distributions and $ g_D$ is orthogonal to $g_R$ in $ W^{r}(H_{\Cal
    O}, \mathcal F_{\alpha,\Lambda}(t))$. Then $g_R$ is a coboundary and $R( g_D)
  =0$. Let $f:= G^{X_\alpha(t)}_{X_\alpha,\Lambda} (g_R)$.  From $R(
  g_D) =0$ and $D (g_R)=0$ we obtain
  \[
  \begin{split}
    |R (g)| &=|R (g_D + g_R)| = |R (g_R)| = |\gamma_x(g_R)-D (g_R)|\\&
    = |\gamma_x(g_R)| \le \frac 1 L (|f(x)| + |f(y)|).
  \end{split}
  \]
  By Theorem~\ref{thm:SEnil} and Lemma~\ref{lem:laststep:2}, for any
  $\tau>\adim/2+1$ there exist positive constants $C_\tau$ and $C$
  such that for any $z\in M$ we have the estimate
  \[
  |f(z)| \le \frac {C_{\tau}}{w_{\mathcal F_{\alpha,\Lambda}(t)}(z,1)^{1/2}} \left( C
    \,e^{t(1-\delta(\rho))} \,\vert f \vert_{\tau, \mathcal F_{\alpha,\Lambda}(t)} +
    \vert g_R \vert_{\tau -1, \mathcal F_{\alpha,\Lambda}(t)} \right)\,.
  \]

  By Theorem~\ref{thm:Green_renorm} for $r >\tau(\lstep-1)+1$ we have,
  for all $t\ge 0$,
  \begin{equation*}
    \vert f \vert_{\tau, \mathcal F_{\alpha,\Lambda}(t)}  \leq  
    G_{k,r, \tau} (1+\Vert \Lambda\Vert_{\mathcal F_{\alpha,\Lambda}})^{\tau k+2}
    e^{-(1-\lambda(\rho))t}\,   \vert g_R \vert _{r, \mathcal F_{\alpha,\Lambda}(t)} \,.
  \end{equation*} 
  The conclusion follows from the estimates above and the observation
  that, by orthogonality, we have $|g_R|_{r, \mathcal F_{\alpha,\Lambda}(t)} \le
  |g|_{r, \mathcal F_{\alpha,\Lambda}(t)}$.
\end{proof}

\begin{corollary}
  \label{coro:63}
  For every $r > (\adim/2+1)(\lstep -1)+1$, there is a constant
  $C^{(2)}_r$ such that the following holds true for every $\mathcal O
  \in \widehat M_0$, every $\Lambda \in \mathcal O$ and every $x\in M$.  
  Let $\gamma_x$ be the Birkhoff average operator~\eqref{eq:laststep:1} and let 
  $\gamma_x= D+R$ be the decomposition of $\gamma_x$ as an orthogonal sum 
  in $W_0^{-r}(H_{\Cal O}, \mathcal F_{\alpha, \Lambda})$ of an
  $X_\alpha$-invariant distribution $D$ and an orthogonal
  complement~$R$.  Then
  \[
  |R|_{-r, \mathcal F_{\alpha, \Lambda}} \le C^{(2)}_r\,
  [1/\injconst(Y_\Lambda)]^{\adim/2} (1+\Vert \Lambda\Vert_{\mathcal
    F_{\alpha, \Lambda}})^{\frac{\lstep(r +1) -2}{\lstep -1}} \,
  L^{-1}.
  \]
\end{corollary}
\begin{proof} By our definitions, the width function $x\in M \mapsto
  {w_{\mathcal F_{\alpha, \Lambda}}(x,1)^{-1}}$ is uniformly bounded
  on $M$. In fact, for all $x\in M$ we have
   $$
   w_{\mathcal F_{\alpha, \Lambda}}(x,1) \geq
   \left(\frac{\injconst(Y_\Lambda)}{2} \right)^\adim \,.
   $$
   The above statement then follows immediately from
   Proposition~\ref{prop:514} applied to the orthogonal decomposition
   $\gamma_x= D(0) +R(0) $ in the Hilbert space $W^{r}(H_{\Cal O},
   \mathcal F(0))$.
 \end{proof}

 \subsection{Bounds on ergodic averages in a irreducible sub-representation}
 \label{sec:inductive-scheme} 
 
 In this section we derive bounds on ergodic averages of quasi-Abelian
 nilflows for functions in a single irreducible sub-representation. The
 proof follows an inductive argument based on the coboundary bounds
 proved above and on the scaling of invariant distributions proved in
 section~\ref{sec:CE}.

 For brevity, let us set
 \begin{equation}
   \label{eq:laststep:2}
   C_r(\Lambda):= (1+\Vert \Lambda\Vert_{\mathcal
     F_{\alpha, \Lambda}})^{\frac{(2\lstep-1)(r +1)-2}{\lstep -1}}.   
 \end{equation}

 \begin{proposition}
   \label{prop:67}


   Let $r > (\adim/2+1)(\lstep -1)+1$. Let $(L_i)$ be an increasing
   sequence of positive real numbers $\ge 1$, let $0<\wpar\leq
   \injconst(Y_\Lambda)^\adim$ and let $\dev >0$.  There exists a
   constant $C_{r}(\rho)$ such that for every $(\wpar, (L_i),
   \zeta)$-good point $x\in M$ for the basis $\mathcal F_{\alpha,
     \Lambda}$, for all $i\in \N$ and all $f\in W^{r}(H_{\Cal O},
   \mathcal F_{\alpha, \Lambda})$, we have
   \begin{equation}
     \label{eq:laststep:7}
     \left|\frac 1 {L_i}\int_0^{L_i} f\circ \phi^\tau_{X_\alpha}(x) \,
       \D{}\tau\right| 
     \le C_{ r}(\rho) \,C_r(\Lambda) \wpar^{-1/2} \,{L_i}^{- \delta(\rho)
       +\lambda(\rho)/2+\dev/2} \vert f\vert_{r, \mathcal F_{\alpha, \Lambda}}\,.
   \end{equation}
 \end{proposition}
 \begin{proof}
   Let us set $N_i:=[\log L_i/\log 2]$ and $t_{j,i}: = \log L_{j,i}:=
   \log L_i^{j/N_i}$, for all $j = 0, \dots, N_i$, and observe that
   \begin{equation}
     \label{eq:laststep:3}
     N_i \le \log L_i/\log 2 < N_i+1, \quad\text{ hence }
     \quad L_i^{1/(N_i+1)}< 2 \le L_i^{1/N_i}<4.
   \end{equation}
   Let ${\mathcal G}(\wpar,(L_i),\dev)$ be the set of
   $(w,(L_i),\dev)$-good points.  Let us adopt the following notation:
   for all $t>0$, $\mathcal F_{\alpha,\Lambda}(t):= A^t_\rho \mathcal F_{\alpha,
     \Lambda}$ and, for all $i\in \N$, let $y_i :=
   \phi^{L_i}_{X_\alpha}(x)$.  By the definition of a good point (see
   Definition~\ref{def:width:1}), for every good point $x\in {\mathcal
     G}(\wpar,(L_i),\dev)$, for all $i \in \N$ and for all $j=0,\dots,
   N_i$, we have
   \begin{equation}
     \label{eq:98}
     \frac{1}{w_{\mathcal F (t_{j,i})}(x,1)} \leq
     L_i^\dev/\wpar\quad \text{ and }\quad
     \frac{1}{w_{\mathcal F (t_{j,i})}(y_i,1)} \leq
     L_i^\dev/\wpar.
   \end{equation}

   Let $x \in {\mathcal G}(\wpar,(L_i),\dev)$ and fix $i\in \N$. For
   simplicity, we omit the index $i\in \N$ and we set $\gamma (f) =
   \frac 1 {L_i}\int_0^{L_i} f\circ \phi^\tau_{X_\alpha}(x) \,
   \D{}\tau$, $L=L_i$, $N=N_i$, $y=y_i$ and $t_j=t_{j,i}$.

   We will also denote by ${|~\cdot~|_{r, j}}$ and by
   ${\|~\cdot~\|_{r, j}}$, respectively, the transversal Sobolev norms
   $|~\cdot~|_{ \mathcal F(t_j),r}$ and the transverse
   Lyapunov-Sobolev norms $\|~\cdot~\|_{ \mathcal F(t_j),r}$ relative
   to the rescaled bases $\mathcal F(t_j)$, $j=0,\dots, N$ (see
   \eqref{eq:Lyapunov}).

   Our goal is to estimate $| \gamma |_{\mathcal F_{\alpha,\Lambda}, -r}=| \gamma
   |_{-r, 0}$. For each $j=0,\dots, N$, let
   \begin{equation*}
     \label{eq:95}
     \gamma= D_j +R_j 
   \end{equation*}
   be the orthogonal decomposition of $\gamma$ in the Hilbert space
   $W^{-r}(H_{\Cal O}, \mathcal F(t_j))$ into a $X_\alpha$-invariant
   distribution $D_j$ and an orthogonal complement $R_j$.

   By the triangle inequality and Corollary~\ref{coro:63} , we have
   \begin{equation}
     \label{eq:99}
     \begin{split} 
       | \gamma |_{-r, 0} &\le
       | D_0 |_{-r, 0}+ | R_0 |_{-r, 0} \\
       &< | D_0 |_{-r, 0}+ {C^{(2)}_r\, [1/\injconst
         (Y_\Lambda)]^{\adim/2} (1+\Vert \Lambda\Vert_{\mathcal
           F_{\alpha, \Lambda}})^{\frac{\lstep(r +1) -2}{\lstep
             -1}}}\, L^{-1}.
     \end{split}
   \end{equation}
   Thus we turn to estimating $| D_0 |_{-r, 0}$. By the definition of
   the Lyapunov-Sobolev norm \eqref{eq:Lyapunov} and the bounds
   \eqref{eq:Lyap_comp} we have
   \begin{equation}
     \label{eq:laststep:6}
     | D_0 |_{-r,0} \le D_{k,r} (1+\Vert
     \Lambda\Vert_{\mathcal F_{\alpha, \Lambda}})^{r+1} \Vert D_0\Vert
     _{ -r,0}\,.
   \end{equation}
   Let us observe that, since $D_j +R_j =D_{j-1} +R_{j-1}$, we have
   \[
   D_{j-1} = D_j + R_j' \,,
   \]
   where $R'_j$ denotes the orthogonal projection of $R_j$, in the
   space $W^{-r}(H_{\Cal O}, \mathcal F(t_{j-1})$, on the space of
   $X_\alpha$-invariant distributions. It follows that
   \begin{equation}
     \label{eq:96}
     \begin{split}
       \|D_{j-1}\|_{-r, {j-1}} &\le \|D_j\|_{-r,
         {j-1}} +\| R_j'\|_{-r, {j-1}} \\
       &\le \|D_j\|_{-r,
         {j-1}} +| R_j'|_{-r, {j-1}}  \\
       &\le \|D_j\|_{-r, {j-1}} +| R_j|_{-r, {j-1}}
     \end{split}\,.
   \end{equation}

   \begin{sublemma}
     \label{lem:64}
     There exists a constant $C:=C(r)>0$ such that, for all
     $j=0,\dots, N$,
     \[
     C^{-1} | \cdot|_{-r, j}\le |\cdot |_{-r, j-1} \le C | \cdot
     |_{-r, j}.
     \]
   \end{sublemma}
   \begin{proof}[Proof of the sub-lemma]
     Let us observe that $\mathcal F(t_j) = A^{t_j - t_{j-1}}_\rho
     \mathcal F(t_{j-1})$ and that, by the
     inequalities~\eqref{eq:laststep:3}, $t_j - t_{j-1} = (\log L)/N
     \le 2 \log 2$. Thus in passing from the frame $ \mathcal
     F(t_{j-1})$ to the frame $\mathcal F(t_j) $ the distortion of the
     corresponding transversal Sobolev norms (and of their dual norms)
     is uniformly bounded.
   \end{proof}

   By the above sub-lemma, the inequality~\eqref{eq:96} becomes
   \begin{equation}
     \label{eq:laststep:4}
     \|D_{j-1}\|_{-r, {j-1}} \le \|D_j\|_{-r, {j-1}} + C\,| R_j|_{-r, {j}} . 
   \end{equation}
   By Lemma~\ref{lemma:Lyap}, with respect to the Lyapunov-Sobolev
   norms, we have that, for any $X_\alpha$-invariant distribution~$D$
   and for all $t\ge s$,
   \[
   \Vert D \Vert_{ \mathcal F(s),-r} \leq e^{-\lambda(\rho) (t-s)/2}
   \Vert D \Vert_{\mathcal F_{\alpha,\Lambda}(t),-r}\,,
   \]
   from which, by taking into account that $\mathcal F(t_j) = A^{t_j -
     t_{j-1}}_\rho \mathcal F(t_{j-1})$, we obtain
   \[
   \Vert D_j \Vert_{-r,j-1} \le L^{-\lambda(\rho)/2N} \Vert D_j
   \Vert_{-r,j}.
   \]
   Then, setting $\beta = {\lambda(\rho)}/{2N}$, from
   \eqref{eq:laststep:4} we conclude by finite induction that
   \begin{equation} \label{eq:97}
     \begin{split}
       \Vert D_0 \Vert_{-r, 0} & \leq L^{-\lambda(\rho)/2} \Big(\Vert
       D_N
       \Vert_{-r, N}   \\
       &+ C \sum_{\ell=0}^{N-1} L^{(\ell+1)\beta} \vert R_{N-\ell}
       \vert_{-r,N-\ell}\Big) = L^{-\lambda(\rho)/2} (I +II)
     \end{split}
   \end{equation}

   \begin{sublemma}
     \label{lem:65}
     For any $r> \adim/2$ there exists a constant $C_r>0$ such that,
     for all good points $x\in \mathcal \mathcal G(\wpar, (L_i),
     \dev)$, we have
     \[
     \| D_N \|_{-r, N} \le { C_r L^{\dev/2}} / \wpar^{1/2}.
     \]
   \end{sublemma}
   \begin{proof}[Proof of the sub-lemma]
     By the definition of Lyapunov-Sobolev norms and by orthogonality
     we have $ \| D_N \|_{-r, N} \le | D_N |_{-r, N} \le | \gamma
     |_{-r, N}$.
     \begin{sloppypar}
       The orbit segment $(\phi^\tau_{X_\alpha}(x))_{0\le \tau \le L}
       $ coincides with the orbit segment
       $(\phi^\tau_{X_\alpha(t_N)}(x))_{0\le \tau \le 1} $ since
       $X_{\alpha}(t_N)=X_{\alpha}(\log L)= LX_{\alpha}$. Hence, using
       the notation of Theorem~\ref{thm:BA_apriori}, we have $ \gamma=
       B^L_{X_{\alpha}}(x) = B^1_{X_{\alpha}(t_N)}(x)$. By that
       theorem, we have $| \gamma |_{-r, N} = | \gamma |_{ \mathcal
         F(t_N),-r} = | B^1_{X_{\alpha}(t_N)}(x) |_{\mathcal
         F(t_N),-r}\le {C_r}\,{w_{\mathcal F(t_N)}(x,1)^{-1/2}}$. By
       the inequality~\eqref{eq:98} we also have ${w_{\mathcal
           F(t_N)}(x,1)^{-1/2}} \le L^{\dev/2}/\wpar^{1/2}$, thereby
       proving the statement.
     \end{sloppypar}
   \end{proof}

   \begin{sublemma}
     \label{lem:65b}
     For every $r > (\adim/2+1)(\lstep -1)+1$ there is a constant
     $C_r(\rho)$ such that for all good points $x\in \mathcal \mathcal
     G(\wpar,(L_i), \dev)$, we have
    
    \[
    \sum_{\ell=0}^{N-1} L^{(\ell+1)\beta} | R_{N-\ell} |_{-r, N-\ell}
    \le C_r(\rho) \wpar^{-1/2}(1+\Vert \Lambda\Vert_{\mathcal
      F_{\alpha, \Lambda}})^{\frac{\lstep(r +1)-2}{\lstep
        -1}}{L}^{\dev/2+ \lambda(\rho)- \delta(\rho)}.
    \]
  \end{sublemma}
  \begin{proof}[Proof of the sub-lemma]
    The orbit segment $(\phi^\tau_{X_\alpha}(x))_{0\le \tau \le L} $
    has length $ L^{\ell/N}$ with respect to the generator
    $X_{\alpha}(t_{N-\ell})= X_{\alpha}((1-\ell/N)\log L)=
    L^{1-\ell/N} X_{\alpha}$.  Thus, by Proposition~\ref{prop:514},
    with $e^{t_{N-\ell}(\lambda(\rho) - \delta(\rho))}=
    L^{(1-\frac\ell N)(\lambda(\rho) - \delta(\rho))}$, we obtain
    \[
    \begin{split}
      |R_{N-\ell} (g)|_{-r, N-\ell} &\le C^{(1)}_r \,(1+\Vert
      \Lambda\Vert_{\mathcal F_{\alpha, \Lambda}})^{\frac{\lstep(r
          +1)-2}{\lstep -1}}
      L^{(1-\frac\ell N)(\lambda(\rho) - \delta(\rho))- \frac \ell N}  \\
      &\qquad\times \left(\frac 1{w_{\mathcal
            F(t_{N-\ell})}(x,1)^{1/2}}+ \frac 1{w_{\mathcal
            F(t_{N-\ell})}(y,1)^{1/2}}\right)\\
      &\le 2C^{(1)}_r\, \wpar^{-1/2} (1+\Vert \Lambda\Vert_{\mathcal
        F_{\alpha, \Lambda}})^{\frac{\lstep(r +1)-2}{\lstep -1}}
      L^{(1-\frac\ell N)(\lambda(\rho) - \delta(\rho))- \frac \ell N
        +\dev/2}\,,
    \end{split}
    \]
    where in the last upper bound we used the
    inequalities~\eqref{eq:98}. Writing, for simplicity, $C:=
    2C^{(1)}_r\, \wpar^{-1/2} (1+\Vert \Lambda\Vert_{\mathcal
      F_{\alpha, \Lambda}})^{\frac{\lstep(r +1)-2}{\lstep -1}} $ and
    recalling that $\beta = \lambda(\rho)/2N$ we obtain
    \begin{equation*}
      \begin{split}
        \sum_{\ell=0}^{N-1} L^{(\ell+1)\beta} | R_{N-\ell} |_{-r,
          N-\ell} &\le C \, {L}^{\dev/2+ \lambda(\rho) - \delta(\rho)}
        \sum_{\ell=0}^{N-1}  L^{(\ell+1)\beta} L^{-\frac\ell N(\lambda(\rho)- \delta(\rho))- \frac \ell N} \\
        &\le C\, {L}^{\dev/2+ \lambda(\rho) -
          \delta(\rho)+\lambda(\rho)/2N}
        \sum_{\ell=0}^{N-1}  L^{-\frac\ell N(1 +\lambda(\rho)/2 - \delta(\rho))} \\
        &\le 2^r C\, {L}^{\dev/2+ \lambda(\rho)-
          \delta(\rho)}\sum_{\ell=0}^{\infty} 2^{-\ell (1
          +\lambda(\rho)/2 - \delta(\rho))}\,,
      \end{split}
    \end{equation*}
    where we have used the inequalities $2 \le L^{1/N}< 4$. By
    Lemma~\ref{lem:laststep:1} we have $1 +\lambda(\rho)/2 -
    \delta(\rho) >1/2$, concluding the proof of the sub-lemma.
  \end{proof}
  By applying the two previous sub-lemmata to the
  formula~\eqref{eq:97}, and observing that, by
  Lemma~\ref{lem:laststep:1}, $\delta(\rho) -\lambda(\rho)/2\le
  \lambda(\rho)/2$, we obtain that there exists a constant
  $C^{(1)}_r(\rho)$ such that
  \begin{equation*}
    \Vert D_0 \Vert_{-r, 0}  \leq  
    C^{(1)}_r(\rho) \,\wpar^{-1/2} (1+\Vert \Lambda\Vert_{\mathcal
      F_{\alpha, \Lambda}})^{\frac{\lstep(r +1)-2}{\lstep
        -1}} {L}^{- \delta(\rho) +\lambda(\rho)/2+\dev/2}.
  \end{equation*}
  From~\eqref{eq:99}, \eqref{eq:laststep:6} and the above we conclude
  that there exists a constant $C^{(2)}_r(\rho)$ such that, whenever
  $0<\wpar \leq \injconst(Y_\Lambda)^\adim$,
  \[ | \gamma |_{-r, \mathcal F} \le C^{(2)}_r(\rho) \,\wpar^{-1/2}
  (1+\Vert \Lambda\Vert_{\mathcal F_{\alpha,
      \Lambda}})^{\frac{(2\lstep-1)(r +1)-2}{\lstep -1}} {L}^{-
    \delta(\rho) +\lambda(\rho)/2+\dev/2},
  \]
  thereby concluding the proof of the proposition.
\end{proof}
\begin{notation}
  \label{not:inductive-scheme-1}
  Let
  \[
  \widetilde M_0 = \bigcup_{\Cal O \in \widehat M_0}\{ \Lambda \in
  \Cal O \mid \Lambda \text{ integral}\}.
  \]
\end{notation}

\begin{theorem}
  \label{thm:laststep:1}
  Let $\sigma= (\sigma_1, \dots, \sigma_n)\in (0,1)^n$ be such that
  $\sigma_1 +\dots+\sigma_n=1$.  For any $\Lambda \in\widetilde M_0$,
  let
  \[
  \sigma_\Lambda:=\max \{\sigma_m \mid m=1, \dots, n ,
  \Lambda(\eta_\lstep^{(m)}) \not =0\}.
  \]
  Let $\nu \in [1, 1+ (\lstep/2 -1)\sigma_\Lambda]$.  Then for any $r
  > (\adim/2+1)(\lstep -1)+1$, there exists a constant
  $C_r(\sigma,\nu)>0$ such that the following holds true. For every
  $\error>0$ there exists a constant $K_\error (\sigma,\nu)>0$ such
  that, for every $\alpha:= (\alpha^{(m)}_i) \in \R^\adim$ such that
  $\alpha_1:= (\alpha_1^{(1)}, \dots, \alpha_1^{(n)}) \in D_n (\sigma,
  \nu)$ and for every $\wpar \in (0, \injconst(Y_\Lambda)^\adim]$
  there exists a measurable set $\mathcal G_\Lambda(\sigma,\error,
  \wpar)$ satisfying the estimate
  \begin{equation}
    \label{eq:laststep_meas_est}
    \operatorname{meas}\left(M \setminus 
      \mathcal G_\Lambda(\sigma,\error, \wpar) \right) \, \leq \,
    K_\error (\sigma,\nu) \bigl( \frac{w}{\injconst (Y_\Lambda)^\adim} \bigr) \,
    {\mathcal H}( Y_\Lambda, \rho, \alpha)\,,
  \end{equation}
  with the property that for every $x\in \mathcal G_\Lambda(\sigma, \error,
  \wpar)$, for every $f\in W^{r}(H_{\Cal O}, \mathcal F)$ and every
  $L\geq 1$ we have
  \begin{equation*}
    \left|\frac 1 L\int_0^{L} f\circ \phi^\tau_{X_\alpha}(x) \, \D{}\tau \right |
    \le \frac{C_r(\sigma,\nu) C_r(\Lambda)} {\wpar^{1/2}} L^{ -(1-\error)
      \frac{ 2 \sigma_\Lambda} { 3(\lstep-1) [(\lstep-2)\sigma_\Lambda +2]}   } \vert
    f\vert_{r, \mathcal F_{\alpha, \Lambda}}.
  \end{equation*}
  Furthermore, if $\wpar'<\wpar$ we have $\mathcal G_\Lambda(\error,
  \wpar, \lstep)\subset \mathcal G_\Lambda(\error, \wpar' , \lstep)$.
\end{theorem}

\begin{proof}
  It is not restrictive to assume, up to renumbering the coordinates
  of the vector $(\sigma_1, \dots, \sigma_n)$, that $\sigma_\Lambda =
  \sigma_1$. Let $\rho= (\dots, \rho^{(m)}_i, \dots)$ be the vector
  given by the following formulas:
  $$
  \begin{aligned}
    &\rho^{(m)}_1:= \frac{2\sigma_m}{(\lstep -2)\sigma_1 +2}\,, \quad \text{ for all } m=1, \dots, n\,; \\
    &\rho^{(1)}_i := \frac{ 2 \sigma_1 (\lstep-i)} { (\lstep-1)
      [(\lstep-2)\sigma_1 +2]} \,,
    \quad \text{ for all } i=2, \dots, \lstep \,; \\
    &\rho^{(m)}_i :=0\,, \quad \text{ for all } m \not =1 \text{ and
      all } i\not=1 \,.
  \end{aligned}
   $$
   We can verify that by the hypothesis and the above definition
   \begin{equation}
     \label{eq:delta_lambda}
     \lambda (\rho)  = \delta(\rho) =   \frac{ 2 \sigma_1} { (\lstep-1) [(\lstep-2)\sigma_1 +2]} \,.
   \end{equation}
   Let us set $\dev:=2\delta(\rho)/3 -\lambda(\rho)/3$. It is not
   restrictive to assume that $\dev > 0$, otherwise the statement is
   trivially true, since by the Sobolev embedding theorem any function
   $f \in W^{r}(H_{\Cal O}, \mathcal F)$ is (uniformly) bounded.
  
   Let $\error>0$ and, for all $i\in \N$, let us set $L_i =
   i^{(1+\error) \dev^{-1}}$. Then there exists a
   constant~$K_\error(\rho)>0$ such that
  $$ 
  \Sigma (\wpar, (L_i), \zeta)= \sum_i (\log L_i)^2 L_i^{-\dev} \leq
  K_\error (\rho) \,.
  $$
  Let $\mathcal G = \mathcal G_\Lambda(\sigma, \error, \wpar):=
  \mathcal G(\wpar, (L_i), \dev)$ be the set of $(\wpar, (L_i),
  \dev)$-good points for the basis $\mathcal F_{\alpha, \Lambda}$.
  The estimate in formula~\eqref{eq:laststep_meas_est} follows from
  Lemma~\ref{lem:62} and the last statement of the Theorem from the
  definition of good points.  By Proposition~\ref{prop:67}, for all
  $x\in {\mathcal G}$ and for every $f\in W^{r}(H_{\Cal O}, \mathcal
  F)$ the estimate in formula~\eqref{eq:laststep:7} holds true.

  Let $L\in [L_i, L_{i+1}]$. Then
  \[
  \int_0^{L} f\circ \phi^\tau_{X_\alpha}(x) \, \D{}\tau = \int_0^{L_i}
  f\circ \phi^\tau_{X_\alpha}(x) \,\D{}\tau + \int_{L_i}^L f\circ
  \phi^\tau_{X_\alpha}(x) \, \D{}\tau= (I)+(II).
  \]
  For brevity, let $C:= C_r(\rho) C_r(\Lambda) / \wpar^{1/2}$. The
  first term is estimated by formula~\eqref{eq:laststep:7}:
  \[
  (I)\le C\,{L}^{1- \delta(\rho) +\lambda(\rho)/2 +\dev/2} \vert
  f\vert_{\sigma, \mathcal F_{\alpha, \Lambda}} = C\,{L}^{1-
    2\delta(\rho)/3 +\lambda(\rho)/3} \vert f\vert_{\sigma, \mathcal
    F_{\alpha, \Lambda}}\,.
  \]
  For the second term, the statement follows from an elementary
  estimate. In fact, let us set $\beta:= (1+\error) \dev^{-1}$ and
  observe that $\beta^{-1}=\dev (1+\error)^{-1} \geq (1-\error)\dev$.
  We have
  \[
  \begin{aligned}
    (II) \le (L-L_i) \|f\|_\infty &\le \beta 2^{\beta -1} L^{1
      -\beta^{-1}} \|f\|_\infty \\ &\le C' (\rho) L^{1 - (1-\error)
      (2\delta(\rho)/3 -\lambda(\rho)/3) } \|f\|_{r, \mathcal
      F_{\alpha, \Lambda}}\,.
  \end{aligned}
  \]
  By the above estimates on the terms $(I)$ and $(II)$ and by the
  identities in formula~\eqref{eq:delta_lambda} for the exponents
  $\lambda(\rho)$ and $\delta(\rho)$, the proof is completed.
 
\end{proof}
\subsection{General bounds on ergodic averages}
 \label{sec:GeneralBounds}
  In this section the bounds on ergodic averages obtained above for
 functions belonging to a single irreducible sub-representation are
 generalised to all sufficiently smooth functions. The main idea is to
 use extra regularity of the datum to obtain estimates that are
 uniform across all irreducible sub-representations.
\smallskip

For all $\Cal O\in \widehat M_0$ and $\Lambda\in \Cal O$, the vector
\[
(\Lambda(\tilde\eta^{(1)}_{i_1}), \Lambda(\tilde\eta^{(2)}_{i_2}),
\dots, \Lambda(\tilde\eta^{(n)}_{i_n})) = (\Lambda(\eta^{(1)}_{i_1}),
\Lambda(\eta^{(2)}_{i_2}), \dots, \Lambda(\eta^{(n)}_{i_n})) \,,
\]
which obviously depends only on $\Cal O$, is integral.

For $\Cal O \in {\widehat M}_0$ we define a canonical $\Lambda_{\Cal
  O} \in \Cal O$ in the following way. For $\Lambda \in \Cal O$, let
\[
\vert\Cal O\vert= \max_{m=1,\dots,n} \{|\Lambda (\tilde\eta_k^{(m)})| \,\mid
i_m=k\}.
\]
By the above remarks $\vert\Cal O\vert$ does not depend on the choice
of $\Lambda \in \Cal O$ and by the definition of ${\widehat M}_0$ we
have $\vert\Cal O\vert\not=0$.  Let $m(\Cal O)\in\{1,\dots,n\}$ be the
smallest integer $m$ such that
\[
i_m=k, \quad\text{and} \quad|\Lambda (\tilde\eta_k^{(m)})|= \vert\Cal O\vert.
\]
Recall that the basis $\mathcal F_{\alpha,\Lambda}=(X_\alpha, Y_\Lambda)$ was defined by the choice of an integer $m_0$ such that the element $\eta_1^{(m_0)}$ had
degree $\lstep$, i.e. such that $\Lambda(\tilde \eta_\lstep^{(m_0)})\neq 0$.

\medskip \emph{We shall assume henceforth that $m_0=m(\Cal O)$ making,
  in this way, a unique choice of the basis $\mathcal
  F_{\alpha,\Lambda}$. After relabelling the elements of the basis
  $\eta$ we may also assume that $m_0=m(\Cal O)=1$. We shall do
  so, for simplicity of notation.}

\medskip
We first prove estimates for the constants $\injconst(Y_{\Lambda})$, introduced in
Definition~\ref{def:AY}, and the constant ${\mathcal H}( Y_\Lambda, \rho,
  \alpha)$, introduced in formula~\eqref{eq:H_const}, in terms of the weight
  $|\Lambda(\mathcal F_{\alpha,\eta})|$,  introduced in formula~\eqref{eq:Lambda_Fnorm_nohat} (see also Remark~\ref{rem:laststep:1}).

\smallskip
From Definition~\ref{def:AY} and Lemma~\ref{lemma:reduction} we derive
the following estimates:
\begin{lemma}
  \label{lem:laststep:3}
   For every $\mathcal O \in \widehat M_0$ and for every $\Lambda \in \mathcal O$, we have
  \[\injconst(Y_{\Lambda}) \ge \frac 1 \lstep
  \left(1+\frac{|\Lambda(\mathcal
      F_{\alpha,\eta})|}{\vert\mathcal O\vert
    }\right)^{-\lstep}.
  \]
\end{lemma}
\begin{proof}
  The return time of the flow~$X_{\alpha}$ to any orbit of the
  codimension one Abelian subgroup $A \subset G$ is equal to
  $1$. Hence, by Definition~\ref{def:AY}, we have
  $\injconst(\eta)= 1 /2$ for the basis $\eta:=(\eta_i^{(j)})$. From
  the construction above $\eta_{1}^{(1)}$ is an element of degree
  $\lstep-1$ and $\ad^{\lstep-1}\eta_{1}^{(1)} = \eta_{k}^{(1)}$. By
  these observations, the statement follows easily from
  Definition~\ref{def:AY} and the estimate of
  Lemma~\ref{lemma:reduction} of the coefficients of the matrix of
  change of bases $C^{\eta,Y_{\Lambda}}$ (since the basis $\Cal
  F_{\alpha,\eta} = (X_\alpha, \eta)$ is Jordan) .
\end{proof}

\begin{lemma}
  \label{lem:laststep:4}
  For every $\mathcal O \in \widehat M_0$ and for every $\Lambda \in \mathcal O$, we have
  \[
  \injconst(Y_\Lambda)^{-\adim} \, {\mathcal H}( Y_\Lambda, \rho,
  \alpha) \le \frac 2 {\lstep^\adim} \,C(\alpha_1)\,\big(1 + \log
  C(\alpha_1)\big)\, \left(1+\frac{|\Lambda(\mathcal
      F_{\alpha,\eta})|}{|\mathcal O|
    }\right)^{\adim\lstep}.
  \]
\end{lemma}
\begin{proof}
  By Definition~\ref{def:AY} we have $\injconst(\eta)\le 1/2$ and
  by the definition \eqref{eq:not_DC} of $C(\alpha_1)$ we have
  $C(\alpha_1)\ge 1$. Then from the definition~\eqref{eq:H_const} of
  the constant ${\mathcal H}( Y_\Lambda, \rho, \alpha)$, using $\adim
  >n$ , we obtain
  \[
  \begin{split}
    \injconst(Y_\Lambda)^{-\adim} \, {\mathcal H}( Y_\Lambda, \rho,
    \alpha)&=
    \injconst(Y_\Lambda)^{-\adim} \\
    &\qquad+ \injconst(Y_\Lambda)^{-n} C(\alpha_1) \Big(1 + \log^+
    [\injconst(Y)^{-1}] + \log
    C(\alpha_1)\Big)\\
    &\le C(\alpha_1)\,\big(1 + \log
    C(\alpha_1)\big) \\
    &\qquad\times\Big( \injconst(Y_\Lambda)^{-\adim} +
    \injconst(Y_\Lambda)^{-n}  \log^+ [\injconst(Y)^{-1}]\Big) \\
    & \le 2\,C(\alpha_1)\,\big(1 + \log C(\alpha_1)\big)\,
    \injconst(Y_\Lambda)^{-\adim} .
  \end{split}
  \]
  The lemma now follows from Lemma~\ref{lem:laststep:3}.
\end{proof}

We then construct sets of large measure on which bounds for ergodic integrals hold
for functions in each irreducible sub-representation with appropriate constants.
 
\begin{corollary}
  \label{cor:laststep:1}
  For given $\Cal O \in \widehat M_0$,  $\Lambda\in \Cal O$, $\wpar>0$ and  $\error>0$, let
  \[
  \wpar_{\Lambda}:=\wpar\cdot |\Lambda(\mathcal
  F_{\alpha,\eta})|^{-\adim(\lstep
    +1)+2-\error}.
  \]
  For $\sigma$, $\nu$, $r$, $\error$ and $\alpha\in \R^\adim$ as in
  Theorem~\ref{thm:laststep:1} let $\mathcal G_{\Lambda}(\sigma, \error, \wpar_{\Lambda})$ be the set given by
  that theorem. Then, for every $\wpar>0$ and $\error >0$ the set
  \[
  \mathcal G(\sigma, \error,\wpar):=\ \bigcap_{\Lambda \in \widetilde
    M_0} \mathcal G_{\Lambda}(\sigma, \error,
  \wpar_{\Lambda})
  \]
  has measure greater than
  \[ 1- C \wpar\error^{-1}, \quad \text{ with } C:= \lstep^{-\adim}
  K_\error (\sigma,\nu) \,C(\alpha_1)\,\big(1 + \log
  C(\alpha_1)\big)\,. \] Furthermore, if $\wpar'<\wpar$ we have
  $\mathcal G (\error, \wpar, \lstep)\subset \mathcal G (\error,
  \wpar' , \lstep)$.
\end{corollary}
\begin{proof}
  Recalling that $\Lambda(\eta_{\lstep}^{(1)})=
  \vert\mathcal O\vert$ and $|\Lambda(\mathcal
  F_{\alpha,\eta})|$ are integral multiples of $2\pi$, by
  Theorem~\ref{thm:laststep:1}, Lemma~\ref{lem:laststep:4} and
  the definition of $\wpar_{\Lambda}$ we have
  \[
  \begin{split}
    \operatorname{meas}\left(M \setminus \mathcal G_{\Lambda}(\sigma, \error, \wpar_{\Lambda})\right) \, &\leq
    \pi^{-\lstep \adim}\,C \, |{\Lambda}(\mathcal
    F_{\alpha,\eta})|^{-\adim -\error}\,,
  \end{split}
  \]
  where $C= 2 \lstep^{-\adim} K_\error (\sigma,\nu)
  \,C(\alpha_1)\,\big(1 + \log C(\alpha_1)\big)$. Since the cardinal
  of integral linear forms $\Lambda\in \widetilde M_0$ such that
  $|\Lambda(\mathcal F_{\alpha,\eta})| = 2\pi\ell$ is bounded by
  $(2\ell)^{\adim -1}$ we have
  \[
  \begin{split}
    \sum_{\Lambda \in \widetilde M_0} \operatorname{meas}&\left(M
      \setminus \mathcal G_{\Lambda}(\sigma, \error,
      \wpar_{\Lambda}) \right)\le
    \\
    & 2^{-\adim} C'\wpar\sum_{\ell > 0}\quad\sum_{\Lambda \in
      \widetilde M_0 \colon |\Lambda(\mathcal F_{\alpha,\eta})| =
      2\pi\ell}
    \ell^{-\adim -\error} \\
    & \quad\le 2^{-1} C' \wpar \sum_{\ell > 0} \ell^{-1-\error} <
    2^{-1} C' \wpar \error^{-1}.
  \end{split}
  \]
  The final statement on the monotonicity of the set $\mathcal G
  (\error, \wpar, \lstep)$ with respect to $\wpar>0$ follows from the
  definition of this set and the analogous statement in
  Theorem~\ref{thm:laststep:1}.
\end{proof}

To sum up the estimates for ergodic integrals 
we will bound the constants in our estimates for each irreducible sub-representation  
in terms of higher norms of the
datum. This step can be accomplished by making a particular choice of 
a linear form in every co-adoint orbit. 

\begin{definition}
  For every $\Cal O \in {\widehat M}_0$ we define $\Lambda_{\Cal O}$ as the
  unique integral linear form $\Lambda\in \Cal O$ such that
  \[
  0\le \Lambda(\tilde\eta_{k-1}^{(1)}) <\vert\Cal O\vert.
  \]
\end{definition}
The existence and uniqueness of $\Lambda_{\Cal O}$ follows immediately
from the observations that $ \Lambda\circ
\Ad(\exp(tX_\alpha))(\tilde\eta_{k-1}^{(1)})=
\Lambda(\tilde\eta_{k-1}^{(1)}) + t \vert\Cal O\vert$ and
that the form $ \Lambda\circ \Ad(\exp(tX_\alpha))$ is integral for all
integer values of $t\in \R$.

\begin{lemma}
  \label{lemma:Lambdabound}
  There exists a constant $C(\Gamma)>0$ such that on the primary
  subspace $C^\infty(H_{\Cal O})$ the following estimate holds true:
  \begin{equation*}
    |\Lambda_{\Cal O}(\mathcal F_{\alpha,\eta})| \,\operatorname{Id}    \le C(\Gamma) \,(1 + \Delta_{\mathcal F_{\alpha,\eta}})^{\lstep/2}\,.   
  \end{equation*}
\end{lemma}

\begin{proof}
  Let $x_0= - \Lambda_{\Cal O}(\eta_{\lstep-1}^{(1)})/\vert\mathcal O\vert$. Then there exists a unique $\Lambda'\in \Cal O$ such
  that $\Lambda' (\eta_{\lstep-1}^{(1)})=0$ given by the
  formula $\Lambda' =\Lambda_{\Cal O} \circ \Ad(e^{x_0 X_\alpha})$.
  Let us recall that, for any $\Lambda \in \mathfrak a^*$, the
  elements $V\in \mathfrak a$ are represented in the representation
  $\pi_\Lambda^{X_\alpha}$ as multiplication operators by the
  polynomials
  \begin{equation}
    \label{eq:laststep:9}
    \imath P(\Lambda, V)(x)= \imath \Lambda ( \Ad(e^{x X_\alpha} )V ).
  \end{equation}  
  By the definition of the linear form $\Lambda_{\Cal O}\in \Cal O$,
  the identity $[X_\alpha,\eta^{(1)}_{\lstep-1}]=\eta^{(1)}_{\lstep}$ immediately implies
  \[
  P(\Lambda', \eta^{(1)}_{\lstep-1})(x)= \vert\mathcal O\vert\, x.
  \]
  From~\eqref{eq:laststep:9}, we have $P(\Lambda', \Ad(e^{-x
    X_\alpha}) V)(x)= \Lambda'( V)$ for all $V\in \mathfrak a$, or
  equivalently,
  \[
  \sum_j \frac{(-x)^j}{j!}\, P(\Lambda', \ad(X_\alpha)^j V)= \Lambda'
  ( V),\quad \text{ for all } V\in \mathfrak a;
  \]
  hence, for every element $V\in \mathfrak a$ we obtain
  \begin{equation}
    \label{eq:Lambdaid}
    \begin{split}
      \Lambda' ( V) &=\sum_j \frac{(-1)^j}{j!}\left(\frac{P(\Lambda',
          \eta^{(1)}_{\lstep-1})}
        {\vert\mathcal O\vert}\right)^j\, P(\Lambda', \ad(X_\alpha)^j V)\\
      &=\vert\mathcal O\vert^{1-\lstep}\sum_j \frac{(-1)^j}{j!}{P(\Lambda',
        \eta^{(1)}_{\lstep-1})}^j {P(\Lambda', \eta^{(1)}_{\lstep})}^{\lstep-1-j}\, P(\Lambda', \ad(X_\alpha)^j
      V).
    \end{split}
  \end{equation}
  Let us recall that, for any $\Lambda\in \mathfrak a^*$, the
  transversal Laplacian for a basis~$\mathcal F$ in the representation
  $\pi_\Lambda^{X_\alpha}$ is the operator of multiplication by the
  polynomial $$\Delta_{\Lambda,\mathcal F}=\sum_{V\in \mathcal F}
  \pi_\Lambda^{X_\alpha}(V)^2 = \sum_{V\in \mathcal F} P(\Lambda,
  V)^2\,,$$ hence for all $(m,j) \in J$, the following bound holds:
  $$
  \vert P( \Lambda', \eta_{j}^{(m)}) \vert \leq (1+ \Delta_{\Lambda',
    \mathcal F_{\alpha, \eta}})^{1/2}\,,
  $$
  Since by the identity in formula~\eqref{eq:Lambdaid} the constant
  operators $ \Lambda'( \eta_{j}^{(m)})$ are given by polynomials
  expressions of degree $\lstep$ in the operators $P( \Lambda',
  \eta_{j}^{(m)})$ we obtain the estimate
  \begin{equation*}
    |\Lambda'(\mathcal F_{\alpha,\eta}) | \,\operatorname{Id}  
    \le C_1(\Gamma) (1 +  \Delta_{\Lambda', \mathcal F_{\alpha, \eta}})^{\lstep/2}\,.   
  \end{equation*}
  Since the representations $\pi_{\Lambda'}^{X_\alpha}$ and
  $\pi_{\Lambda_0}^{X_\alpha}$ are unitarily intertwined by the
  translation operator by $x_0$ and since constant operators commute
  with translations, we also have
  \begin{equation*}
    |\Lambda'(\mathcal F_{\alpha,\eta}) | \,\operatorname{Id}  
    \le C_1(\Gamma) (1 +  \Delta_{\Lambda_{\mathcal O}, \mathcal F_{\alpha, \eta}})^{\lstep/2}\,.   
  \end{equation*}
  Finally, the inequality $0\le
  \Lambda(\tilde\eta_{k-1}^{(1)}) <\vert\mathcal O\vert$ implies that
  $x_0$ is bounded by a constant depending only on $\lstep$. Hence the
  norms of the linear maps $\Ad(\exp(\pm x_0X_\alpha))$ are
  bounded by a constant depending only on $\lstep$. In follows that
  $|\Lambda_{\Cal O}(\mathcal F_{\alpha,\eta}) | \le C_2(\lstep)
  |\Lambda'(\mathcal F_{\alpha,\eta}) | $ and the statement of the
  lemma follows.
\end{proof}

\begin{corollary}
  \label{cor:laststep:2}
  There exists a constant $C'(\Gamma)$ such that for all $\Cal O \in
  \widehat M_0$ and for any sufficiently smooth function $f\in
  H_{\Cal O}$ we have
  \[
  C_r(\Lambda_{\mathcal O})\, w_{\Lambda_{\mathcal O}}^{-1/2} \,\vert
  f\vert_{r, \mathcal F_{\alpha, \Lambda}}\le C'(\Gamma)\wpar^{-1/2}
  \vert f\vert_{r+e, \mathcal F_{\alpha, \eta}}\] where
  $e=\adim(\lstep +1)\lstep /2 + \lstep\frac{(2\lstep-1)(\sigma
    +1)-2}{2(\lstep -1)}$.
\end{corollary}
\begin{proof}
  Recall that by the definition~\eqref{eq:laststep:2} we have
  $C_r(\Lambda_{\mathcal O}) = (1 + |\Lambda_{\mathcal O}(\mathcal
  F_{\alpha,\Lambda_{\mathcal O}})|)^{e_1}$ with $e_1
  :=\frac{(2\lstep-1)(r +1)-2}{\lstep -1}$.  Since $|\Lambda_{\mathcal
    O}(\mathcal F_{\alpha,\Lambda_{\mathcal O}})| \le
  |\Lambda_{\mathcal O}(\mathcal F_{\alpha,\eta})|$ we have
  \[
  C_r(\Lambda_{\Cal O})\, \wpar_{\Lambda_{\mathcal O}}^{-1/2} \le\wpar^{-1/2}(1 +
  |\Lambda_{\mathcal O}(\mathcal F_{\alpha,\eta})|)^{e_2}
  \]
  with $e_2:=e_1+\adim(\lstep +1)$. By Lemma~\ref{lemma:Lambdabound}
  on the space $H_{\Cal O}$ we have $$ (1 + |\Lambda_{\Cal O}(\mathcal
  F_{\alpha,\eta})|)^{e_2} \le C'(\Gamma)(1 + \Delta_{\mathcal
    F_{\alpha,\eta}})^{e_2\lstep/2}\,.$$
\end{proof}

We are finally ready to derive global estimates for ergodic integrals.

\begin{proposition}
  \label{prop:laststep:6}
  Let $r > (\adim+1)(\lstep -1)+1$. Let $\sigma= (\sigma_1, \dots,
  \sigma_n)\in (0,1)^n$ be a positive vector such that $\sigma_1
  +\dots+\sigma_n=1$.  Let us set
  \[
  \sigma_{min}:=\min \{\sigma_m \mid m=1, \dots, n , i_m=k\}.
  \]
  Let us assume that $\nu \in [1, 1+ (\lstep/2 -1)\sigma_{min}]$ and
  let $\alpha:= (\alpha^{(m)}_i) \in \R^\adim$ be such that
  $\alpha_1:= (\alpha_1^{(1)}, \dots, \alpha_1^{(n)}) \in D_n (\sigma,
  \nu)$. For every $\error>0$ and $\wpar>0$, there exists a measurable
  set $\mathcal G(\sigma,\error,\wpar)$ satisfying
  \[
  \operatorname{meas}\left(M \setminus \mathcal G(\sigma,\error,
    \wpar) \right) \, \leq C \wpar\error^{-1}\,, \quad \text{with }C:=
  \lstep^{-\adim} K_\error (\sigma,\nu) \,C(\alpha_1)\,\big(1 + \log
  C(\alpha_1)\big),
  \]
  such that for every $x\in \mathcal G(\sigma, \wpar,w)$, for every
  $f\in W^{r}(M)$ and every $L\geq 1$ we have
  \begin{equation}
    \left|\frac 1 L\int_0^{L} f\circ \phi^\tau_{X_\alpha}(x) \, \D{}\tau \right |
    \le \wpar^{-1/2}  L^{ -(1-\error)
      \frac{ 2 \sigma_{min}} { 3(\lstep-1) [(\lstep-2)\sigma_{min} +2]}   } \vert
    f\vert_{r, \mathcal F_{\alpha, \eta}}.
  \end{equation}
  Furthermore, if $\wpar'<\wpar$ we have $\mathcal G (\error, \wpar,
  \lstep)\subset \mathcal G (\error, \wpar' , \lstep)$.
\end{proposition}

\begin{proof}
  We have $\tau :=r -\adim \lstep/2 > (\adim/2+1)(\lstep -1)+1$. Let
  $f \in W^{\sigma}(M, \mathcal F)$ and let $f = \sum _{\Cal O\in
    \widehat M_0} f_{\Cal O}$ be its orthogonal decomposition onto the
  primary subspaces $ H_{\mathcal O}$. Clearly $f_{\Cal O} \in
  W^{\tau}(H_\Cal O, \mathcal F)$ and the decomposition is also
  orthogonal in $W^{\tau}(H_\Cal O, \mathcal F)$.

  Having defined for each $\Cal O\in \widehat M_0$ the constant
  $\wpar_{\Lambda_{\mathcal O}}$ as in Corollary~\ref{cor:laststep:1}, by the
  same corollary the set
  \[
  \mathcal G(\sigma, \error,w):=\ \bigcap_{\Cal O \in \widehat M_0}
  \mathcal G_{\Lambda_{\mathcal O}}(\sigma, \error, \wpar_{\mathcal
    O})
  \]
  has measure greater than $ 1- C \wpar\error^{-1}$, where $C=
  \lstep^{-\adim} K_\error (\sigma,\nu) \,C(\alpha_1)\,\big(1 + \log
  C(\alpha_1)\big)$, and satisfies the required monotonicity property
  with respect to $\wpar>0$.
  
  If $x\in \mathcal G(\sigma, \error,w)$, then by
  Theorem~\ref{thm:laststep:1} and by Corollary~\ref{cor:laststep:2}
  the following estimate holds true for every $\Cal O\in \widehat
  M_0\setminus\widehat M_0(x)$ and all $L\ge 1$:
  \[
  \left|\frac 1 L\int_0^{L} f_{\Cal O}\circ \phi^\tau_{X_\alpha}(x) \,
    \D{}\tau \right | \le C_\tau(\sigma,\nu) L^{-(1-\error) \frac{ 2
      \sigma_{min}} { 3(\lstep-1) [(\lstep-2)\sigma_{min}+2]} }
  \,\wpar^{-1/2}\, \vert f_{\Cal O}\vert_{\tau, \mathcal F_{\alpha,
      \eta}}.
  \]
  For any $\tau>0$ and any $\error'>0$, by
  Lemma~\ref{lemma:Lambdabound}, we have
  \[
  \begin{split}
    \Big|\sum_{\Cal O\in \widehat M_0} |f_{\Cal O}|_{\tau, \mathcal
      F_{\alpha, \eta}} \Big|^2 &\le \sum_{\Cal O \in \widehat M_0}
    (1+|\Lambda_{\Cal O}(\mathcal F_{\alpha,\eta})|)^{-\adim-\error'}
    \sum_{\Cal O \in \widehat M_0} (1+|\Lambda_{\Cal O}(\mathcal
    F_{\alpha,\eta})|)^{\adim+\error'} |f_{\Cal O}|^2_{\tau, \mathcal
      F_{\alpha,
        \eta}}\\
    & \le C(\adim) \, |f|^2_{\tau+(\adim+\error')\lstep/2, \mathcal
      F_{\alpha, \eta}}.
  \end{split}
  \]
  and the theorem follows by the linearity of ergodic averages after
  renaming the constants.
\end{proof}

\begin{theorem}
  \label{thm:laststep:6}
  Let $r > (\adim+1)(\lstep -1)+1$. Let $\sigma= (\sigma_1, \dots,
  \sigma_n)\in (0,1)^n$ be a positive vector such that
  $\sigma_1+\dots+\sigma_n=1$.  Let
  \[
  \sigma_{min}:=\min \{\sigma_m \mid m=1, \dots, n , i_m=k\}.
  \]
  Let us assume that $\nu \in [1, 1+ (\lstep/2 -1)\sigma_{min}]$ and
  let $\alpha:=(\alpha^{(m)}_i) \in \R^\adim$ be such that $\alpha_1:=
  (\alpha_1^{(1)}, \dots, \alpha_1^{(n)}) \in D_n (\sigma, \nu)$. For
  every $\error>0$ there exists a full measure measurable set
  $\mathcal G(\sigma,\error)$ and a measurable function
  $\ConstDepx_\error:M\to \R^+$ with $\ConstDepx\in L^p(M)$ for every
  $p\in [1,2[$ such that the following holds.  For every $f\in
  W^{r}(M)$, for every $x\in \mathcal G(\sigma, \error)$ and every
  $L\geq 1$ we have
  \begin{equation}
    \left|\frac 1 L\int_0^{L} f\circ \phi^\tau_{X_\alpha}(x) \, \D{}\tau \right |
    \le  \ConstDepx_\error(x)  L^{ -(1-\error)
      \frac{ 2 \sigma_{min}} { 3(\lstep-1) [(\lstep-2)\sigma_{min} +2]}   } \vert
    f\vert_{r, \mathcal F_{\alpha, \eta}}.
  \end{equation}
  The set $\mathcal G(\sigma,\error)$ and the function $\ConstDepx_\error\in L^p(M)$
  are invariant under the action of $Z(G)$ on $M$, moreover the set $\mathcal G(\sigma,\error)/Z(G)$
  and the function $\ConstDepx_\error\in L^p(M/Z(G))$ are well-defined and invariant 
  under the action of $Z(G/Z(G))$  on $M/Z(G)$. 
\end{theorem}

\begin{proof}
  For $i\in \N^+$ let $\wpar_i:= 1/2^iC$ and $\mathcal G_i:= \mathcal
  G(\sigma,\error, \wpar_i)$, where $ G(\sigma,\error, \wpar)$ is the
  set given by the previous proposition and $C= \lstep^{-\adim}
  K_\error (\sigma,\nu) \,C(\alpha_1)\,\big(1 + \log
  C(\alpha_1)\big)$.  Set $\ConstDepx_\error(x):= 1/\wpar_i^{1/2}$ if
  $x\in \mathcal G_i \setminus \mathcal G_{i-1}$.  By
  Proposition~\ref{prop:laststep:6}, the sets $\mathcal G_i$ are
  increasing and satisfy $\operatorname{meas}\left(M \setminus
    \mathcal G_i \right) \leq 1/ 2^i \error$. Hence the set $\mathcal
  G(\sigma,\error):=\bigcup_{i\in \N^+} \mathcal G_i $ has full
  measure and the function $\ConstDepx$ is in $L^p(M)$ for every $p\in
  [1,2[$. By the same proposition for every $x\in \mathcal
  G(\sigma,\error)$ and every every $f\in W^{r}(M)$ and every $L\geq
  1$ we have
  \begin{equation}
    \left|\frac 1 L\int_0^{L} f\circ \phi^\tau_{X_\alpha}(x) \, \D{}\tau \right |
    \le  \ConstDepx_\error(x)  L^{ -(1-\error)
      \frac{ 2 \sigma_{min}} { 3(\lstep-1) [(\lstep-2)\sigma_{min} +2]}   } \vert
    f\vert_{r, \mathcal F_{\alpha, \eta}}.
  \end{equation}
  By Remark~\ref{rem:width:1}, the stated invariance properties of the set $\mathcal G(\sigma,\error)$ and of the function  $\ConstDepx_\error\in L^p(M)$  under the action of the groups $Z(G)$ and
  $Z(G/Z(G))$ follow immediately from the above definitions. This concludes the proof.
\end{proof}

\begin{proof}[Proof of Theorem~\ref{thm:main_intro}]
  It follows immediately from the above Theorem~\ref{thm:laststep:6}
  by choosing $\sigma=(1/n,1/n,\dots,1/n)$.
\end{proof}

A particular case of the above theorem is obtained when the group $G$
is a $\lstep$-step filiform group $\Filk$. Then $n=1$ and the Lie
algebra $\mathfrak g=\filk$ is generated by the pair~$(\xi,\eta_1)$;
the only non-trivial commutation relations are
$$
[\xi,\eta_i]=\eta_{i+1} \,, \quad \text{ for } \, i=1, \dots \lstep
-1\,.
$$
As usual the formulas~\eqref{eq:basis_change_2} define another basis
$(\tilde \eta_i)$ of the Abelian ideal $\mathfrak a = \<\eta_1, \dots,
\eta_\lstep\>$ and a lattice $\Gamma_\lstep$ is defined as
in~\eqref{eq:lattice_def}.

Let~$M(\Filk)= \Gamma_\lstep\backslash \Filk$ denote the compact
manifold obtained in this particular case. For
$\alpha=(\alpha_1,\dots,\alpha_\lstep)\in \R^{\lstep}$ the vector
field $X_\alpha $ is now given by
$$
X_\alpha :=\log\big[\exp(-\xi)\exp\big( \sum_{i=1}^\lstep \alpha_i
\tilde \eta_i\big)\big]\,.
$$

Let us recall that when $n=1$, by Lemma~\ref{lem:oldnewD}, the
classical Diophantine condition $DC_{\nu}$ implies the Diophantine
condition $D(1, \nu)$. Hence we have:

\begin{theorem}[Filiform case]
  \label{thm:laststep:3}
  Let $r >\lstep^2$. Let $\nu\in [1, \lstep/2]$ and let $\alpha_1 \in
  DC_\nu $.  For every $\error>0$ there exists a full measure
  measurable set $\mathcal G_\error \subset M(\Filk)$ and a measurable
  function $\ConstDepx_\error:\mathcal G_\error\to \R^+$, with
  $\ConstDepx_\error \in L^p(M(\Filk))$ for every $p\in [1,2[$, such
  that for every $x\in \mathcal G_\error$, for every $f\in
  W^{r}(M(\Filk), \mathcal F)$ of average zero and for all $L\geq 1$
  we have
  \[
  \left|\frac 1 L\int_0^{L} f\circ \phi^\tau_{X_\alpha}(x) \, \D{}\tau
  \right | \le \ConstDepx_\error(x) L^{ -(1-\error) \frac{ 2 } {
      3(\lstep-1) \lstep} } \vert f\vert_{r, \mathcal F_{\alpha,
      \eta}}.
  \]
  The set $\mathcal G_\error \subset M(\Filk)$ and the positive
  function $\ConstDepx_\error \in L^p(M(\Filk))$, defined on $\mathcal G_\error$, 
  are invariant under the action the centre $Z(\Filk)$ of the filiform group $\Filk$
  on $M(\Filk)$, moreover the set $\mathcal G_\error/ Z(\Filk)$ and the function 
  $\ConstDepx_\error \in L^p(M(\Filk)/ Z(\Filk))$ are well-defined and invariant 
  under the action of  the quotient $\Filk/Z(\Filk)$ on the quotient filiform
  nilmanifold $M(\Filk)/Z(\Filk)$.
\end{theorem}

\begin{proof}[Proof of Corollary~\ref{cor:main_intro}] We refer to the
  notation introduced in section~\ref{sec:Weyl_sums}.
  
  Let $\alpha= (\alpha_1, 0, \dots, 0)\in \R^\lstep$. By the above
  Theorem~\ref{thm:laststep:3} and by Lemma~\ref{lem:reduction} we
  have that if $\alpha_1 \in DC_\nu $, with $\nu\in [1, \lstep/2]$ and
  $r> \lstep^2 $, for any function $f\in H^r({\mathbb T}_o^\lstep)$ of
  average zero the following bound holds.  There exists a full measure
  measurable set $\mathcal G_\error \subset \T^{\lstep-2}$ and a
  measurable function $\ConstDepx_\error: \T^{\lstep-2} \to \R^+$,
  with $\ConstDepx_\error\in L^p( \T^{\lstep-2})$ for every $p\in
  [1,2[$, such that for all $(s_1, \dots, s_{\lstep -2}) \in\mathcal
  G_\error$ and for all $N \geq 1$, we have
  \[
  \Big|\sum_{\ell=0}^{N-1} f(P_\lstep(\alpha, \bs, \ell)) \Big| \le
  \ConstDepx_\error (s_1, \dots, s_{\lstep-2}) N^{1-\frac{2}
    {3\lstep(\lstep -1)}+ \error} \vert f\vert_{r, \mathcal F}\,.
  \]
  By Lemma~\ref{lem:coefficients} we see that the coefficients $a_0,
  a_1, a_2, \dots, a_{\lstep-1}$ of the polynomial $$P_\lstep(\alpha,
  \bs, N)= \sum_{j=0}^\lstep a_j N^j$$ are linear functions of the
  coordinates $(s_1, \dots, s_\lstep )\in \T^{\lstep}$. In particular,
  as the $(\lstep-2)$-tuple $(s_1, \dots, s_{\lstep -2})$ ranges in a
  set of full measure $\mathcal G_\error \subset \T^{\lstep-2}$ the
  coefficients $a_2, \dots, a_{\lstep-1}$ of the polynomial
  $P_\lstep(\alpha, \bs, N)$ also range in a subset of full measure of
  $\T^{\lstep-2}$, while for every fixed $(\lstep-2)$-tuple $(s_1,
  \dots, s_{\lstep -2})$ as the pair $(s_{\lstep-1}, s_\lstep)$ ranges
  over all $\T^2$, the pair of coefficients $(a_0, a_1)$ also ranges
  over all $\T^2$.
\end{proof}

\providecommand{\bysame}{\leavevmode\hbox to3em{\hrulefill}\thinspace}
\providecommand{\MR}{\relax\ifhmode\unskip\space\fi MR }
\providecommand{\MRhref}[2]{%
  \href{http://www.ams.org/mathscinet-getitem?mr=#1}{#2}
}
\providecommand{\href}[2]{#2}


\begin{thebibliography}{Woo12}

\bibitem[FF06]{flafor:heis}
Livio Flaminio and Giovanni Forni, \emph{Equidistribution of nilflows and
  applications to theta sums}, Ergodic Theory Dynam. Systems \textbf{26}
  (2006), no.~2, 409--433. \MR{2218767 (2007c:37003)}

\bibitem[FF07]{MR2261071}
\bysame, \emph{On the cohomological equation for nilflows}, J. Mod. Dyn.
  \textbf{1} (2007), no.~1, 37--60. \MR{2261071 (2008h:37003)}

\bibitem[Fur81]{MR603625}
H.~Furstenberg, \emph{Recurrence in ergodic theory and combinatorial number
  theory}, Princeton University Press, Princeton, N.J., 1981, M. B. Porter
  Lectures. \MR{603625 (82j:28010)}

\bibitem[GT12]{GreenTao}
B.~Green and T.~Tao, \emph{The quantitative behaviour of polynomial orbits on
  nilmanifolds}, Annals of Math. \textbf{175} (2012), 465--540.

\bibitem[Woo12]{wooley}
T.~D. Wooley, \emph{Vinogradov's mean value theorem via efficient
  congruencing}, Annals of Math. \textbf{175} (2012), 1575--1627.

\end{thebibliography}

\end{document}